\documentclass[10pt,a4paper]{article}
\usepackage[T1]{fontenc}
\usepackage[english]{babel}
\usepackage[a4paper, margin=1in]{geometry}
\usepackage{graphicx}
\usepackage{mathtools}
\usepackage{amssymb}
\usepackage{amsthm}
\usepackage{thmtools}
\usepackage{xcolor}
\usepackage{nameref}
\usepackage{babel}
\usepackage{enumerate}
\usepackage{hyperref}
\usepackage{lipsum}
\usepackage{amsmath}
\usepackage{xfrac}
\usepackage{faktor}
\usepackage{mathrsfs}
\usepackage{tkz-graph}
\usepackage{latexsym}
\usetikzlibrary{arrows}
\usepackage{tikz}
\usepackage{adjustbox}
\usepackage{comment}
\usepackage{multirow}
\usepackage{subcaption} 
\usepackage{dsfont}
\usepackage{pagecolor}
\usepackage{cleveref} 
\usepackage{titlesec}
\titleformat{\section}
[hang] 
{\normalfont\Large\bfseries} 
{\thesection.} 
{0.4em} 
{} 

\titleformat{\subsection}
[hang] 
{\normalfont\large\bfseries} 
{\thesubsection.} 
{0.3em} 
{} 

\definecolor{myred}{RGB}{255,0,0}

\crefname{equation}{(equation)}{(equations)}
\Crefname{equation}{(Equation)}{(Equations)}

\crefformat{(equation)}{\textcolor{myred}{(#1)}}
\hypersetup{
	colorlinks=true,
	linkcolor=myred,
	citecolor=myred,
	urlcolor=myred
}

\allowdisplaybreaks

\theoremstyle{plain}
\newtheorem{thm}{Theorem}[section]

\newtheorem{lem}{Lemma}[section]

\newtheorem{prop}[thm]{Proposition}

\newtheorem{defi}{Definition}[section]

\usepackage[framemethod=TikZ]{mdframed}

\renewcommand{\leq}{\leqslant}
\renewcommand{\geq}{\geqslant}

\newcommand{\grad}{\ensuremath{\nabla}}

\numberwithin{equation}{section}

\newmdenv[linewidth=2pt,
linecolor=black,
innertopmargin=0.2cm,
innerbottommargin=0.5cm,
innerleftmargin=0.7cm,
innerrightmargin=0.7cm,
tikzsetting={fill=blue!30!white}] {squarebox}
\title{\textbf{Averaging principle for a slow-fast stochastic nonlinear fractional Schr\"odinger equation}}
\author{ Manil T. Mohan\textsuperscript{\tiny{+}}, Debopriya Mukherjee\footnote{Department of Mathematics, Indian Institute of Technology Indore, Indore 453552, India\\ \textsuperscript{\, \quad\tiny{+}}Department of Mathematics, Indian Institute of Technology Roorkee, Roorkee 247667, India}\textsuperscript{\,\ ,}\footnote{Corresponding author.\newline
E-mail addresses: phd2301141004@iiti.ac.in (S. Roy),
debopriya@iiti.ac.in (D. Mukherjee),
maniltmohan@ma.iitr.ac.in (M. T. Mohan)} ,
Sandip Roy\footnotemark[1]
}
\begin{document}
\maketitle
\begin{abstract}
	We establish an averaging principle for a structural multiscale stochastic nonlinear fractional Schr\"odinger system on the one-dimensional torus driven by a multiplicative Wiener noise. The slow component is governed by a fractional Schr\"odinger operator with a general polynomial nonlinearity, while the fast component evolves on a shorter time scale and exhibits dissipative diffusion, nonlinear interactions, and stochastic forcing. 	
	 Under suitable dissipative assumptions, we have shown that, as the scale separation parameter tends to zero, the slow component converges strongly to an effective stochastic fractional Schr\"odinger equation. The effective drift is obtained by averaging the coupling term with respect to the unique invariant measure of the frozen fast dynamics.  The proof relies on uniform a priori estimates, ergodicity of the fast equation, H\"older time regularity of the slow component obtained via a vanishing viscosity method, and a Khasminskii-type time discretization argument adapted to fractional dispersive operators. The analysis is technically challenging due to limited smoothing of the fractional Schr\"odinger semigroup and the presence of general polynomial nonlinearities, which are handled through refined estimates and viscosity approximation.

\end{abstract}

\section{Introduction}

In the present paper, we aim to study the averaging principle for a multiscale stochastic nonlinear fractional Schr\"odinger equation. The equation of our interest is,
\begin{equation}\label{multiscaleeq}
	\left\{
	\begin{aligned}
		&\mathrm{d} {\mathfrak{u}}^{\varepsilon} = \left[-i(-\Delta)^{\alpha} {\mathfrak{u}}^{\varepsilon} + \mathcal{N}({\mathfrak{u}}^{\varepsilon}) + F({\mathfrak{u}}^{\varepsilon}, {\mathfrak{v}}^{\varepsilon}) \right]\mathrm{d}t + \Sigma_1({\mathfrak{u}}^{\varepsilon}) \, \mathrm{d}\mathcal{W}_1 && \text{in } \mathbb{T} \times [0,T],\\
		&\mathrm{d}{\mathfrak{v}}^{\varepsilon} = \frac{1}{\varepsilon} \left[(1 + i)(-\Delta)^{\rho} {\mathfrak{v}}^{\varepsilon} + \mathcal{N}({\mathfrak{v}}^{\varepsilon}) - \lambda {\mathfrak{v}}^{\varepsilon} + G({\mathfrak{u}}^{\varepsilon}, {\mathfrak{v}}^{\varepsilon}) \right]\mathrm{d}t + \frac{1}{\sqrt{\varepsilon}} \Sigma_2({\mathfrak{u}}^{\varepsilon}, {\mathfrak{v}}^{\varepsilon}) \, \mathrm{d}\mathcal{W}_2 && \text{in } \mathbb{T} \times [0,T],\\
		&{\mathfrak{u}}^{\varepsilon}(x,0) = {\mathfrak{u}}_0(x) && \text{in } \mathbb{T}, \\
		&{\mathfrak{v}}^{\varepsilon}(x,0) = {\mathfrak{v}}_0(x) && \text{ in } \mathbb{T},
	\end{aligned}
	\right.
\end{equation}
where $\mathcal{N}({\mathfrak{u}})=-(1+i \gamma) |{\mathfrak{u}}|^{\beta-1}{\mathfrak{u}}, \, \beta\in (1,\infty)$, $\mathcal{W}_1$ and $\mathcal{W}_2$ are two real-valued mutually independent Wiener processes defined over a filtered probability space $\mathfrak{U}$ and $\mathbb{T}=\mathbb{R}/2\pi \mathbb{Z}$ is the one-dimensional torus. Here $(-\Delta)^{\alpha},(-\Delta)^{\rho}$ is $2\alpha, 2\rho$-order fractional Laplacian operator with $\alpha, \rho \in (\tfrac{1}{2},1)$, see \Cref{deffraclap}.

The Schr\"odinger equation is the fundamental equation of quantum mechanics. It describes how the quantum state of a physical system evolves over time through a complex-valued wave function. The nonlinear Schr\"odinger equation arises when interactions within the medium or among particles cannot be neglected. In contrast to the linear case, the presence of nonlinear terms allows for rich dynamical behaviors such as soliton formation, wave collapse, modulation instability, and pattern formation. It serves as an effective model in diverse contexts where dispersion and nonlinearity act simultaneously, providing a unified description of wave propagation in nonlinear optical fibers, fluid surfaces, plasma environments, and condensed matter systems, see \cite{agrawal2000nonlinear,chen1984introduction,shankar2012principles,sulem2007nonlinear}. In the context of mathematics, its structural properties make it a influential model for analyzing questions of well-posedness, stability, and the long-term behavior of nonlinear dispersive wave equations.

The fractional Schr\"odinger equation extends the classical model by replacing the standard Laplacian with a fractional-order operator, typically defined via nonlocal or spectral formulations. This modification captures anomalous dispersion and long-range interactions that cannot be represented by local differential operators. As a result, the fractional model is particularly suitable for describing systems exhibiting L\'evy-type dynamics, nonlocal transport, or memory effects, see \cite{MR1755089, MR1948569, MR3059423, MR3692358, MR3636299}. From an analytical perspective, the fractional framework introduces additional challenges due to its nonlocal nature, leading to new phenomena in regularity theory, scattering behavior, and the qualitative properties of solutions.

Many dynamical systems arising in mathematics, physics, and engineering exhibit the presence of multiple time scales, see \cite{MR524817, MR1374108, MR5031242, MR5026697}. In such systems, some components evolve slowly and capture the macroscopic behavior, while others fluctuate rapidly and represent microscopic or highly oscillatory effects. Direct analysis or simulation of the full system is often difficult due to the interaction between these disparate scales. In many realistic situations, the interaction between multiple time scales is further influenced by random perturbations arising from environmental fluctuations, measurement uncertainty, or unresolved microscopic effects. In such stochastic multiscale systems, noise is not merely a secondary effect but can fundamentally alter the qualitative and long-time behavior of solutions, see \cite{MR3668579, MR4925941, MR4953418, MR4904518, MR5024936}. It may induce transitions between states, modify stability properties, and affect invariant measures and statistical equilibria. In the context of our multiscale stochastic fractional Schr\"odinger model, the presence of multiplicative noise in both the slow and fast components reflects intrinsic randomness in the medium and external forcing. Understanding how these stochastic perturbations interact with the separation of time scales is therefore essential for capturing the effective macroscopic dynamics through an appropriate averaging principle.

The concept of averaging originated in classical mechanics as a technique for studying systems with small perturbations and multiple time scales, see \cite{MR547943, MR1956518, MR189312}. It has since evolved into a fundamental analytical tool for ordinary and partial differential equations, as well as stochastic dynamical systems, see \cite{MR686231, MR2316999, MR524817}. In modern terms, the averaging principle asserts that for systems with coupled fast and slow components depending on a small parameter, the slow variable converges-over finite time intervals-to the solution of an effective equation as the parameter tends to zero. This effective equation is obtained by averaging the interaction terms with respect to the long-term statistical behavior of the fast dynamics. Consequently, the detailed fast fluctuations need not be tracked, their cumulative influence is captured by a single averaged term, which governs the effective slow evolution under suitable conditions.

In stochastic and infinite-dimensional settings, the principle becomes more delicate but remains fundamental for model reduction. When the fast subsystem is ergodic and admits a unique invariant measure, its rapid fluctuations influence the slow component only through averaged statistical effects. Consequently, one obtains an effective evolution equation that reflects the dominant large-scale behavior while filtering out microscopic oscillations. Establishing such results rigorously not only provides a justification for reduced models but also deepens the understanding of how multiscale randomness shapes the observable dynamics of complex systems.

\subsection{Literature review}
The stochastic nonlinear Schr\"odinger equation has been studied from various kinds of aspects, see \cite{MR1706888, MR1696311, MR1954077, MR3215081, MR3232027, MR3980316, MR3859442}. Compare to that, the stochastic nonlinear fractional Schr\"odinger equation remains relatively less developed . In particular, the wellposedness, blow-up ciretria and existence of invariant measure have been studied for various forms of of fractional Schr\"odinger equation in \cite{MR3372864, MR4709547, MR3016651, MR4950332, MR4899400, MR4811663, MR2763340, zhang2023stochastic}.

The averaging principle for stochastic partial differential equations has been extensively studied in the literature. For the deterministic systems, the averaging principle is first studied by Bogoliubov and Mitropolsky in \cite{MR141845}. The averaging principle in stochastic setting is introduced by Khasminskii in \cite{MR260052}. Some of early results include \cite{MR2537194, MR2480788}, where an averaging principle is established for a broad class of stochastic reaction–diffusion systems. Fractional stochastic parabolic equations with two time scales on unbounded domains are treated in \cite{MR4426159}. The averaging principle for one-dimensional stochastic Burgers equations with slow–fast time scales is proved in \cite{MR3848236}, while slow–fast stochastic differential equations with time-dependent locally Lipschitz coefficients are considered in \cite{MR4047972}. More recently, averaging results for slow–fast stochastic PDEs with rough coefficients are obtained in \cite{MR4877362}. In \cite{MR4064032}, the averaging principle is studied for the stochastic real Ginzburg–Landau equation driven by an $\alpha$-stable process, and the stochastic Kuramoto–Sivashinsky equation with slow and fast time scales is investigated in \cite{MR3917783}. Some more works in this direction can be found in \cite{MR3556788, MR2831776, MR4997740, mohan2020averaging, MR4998963, MR4895740, MR4597485, MR4656981, MR2744917, MR3679916}.

Compared with these predominantly parabolic settings, averaging principles for stochastic Schr\"odinger-type equations, especially those involving fractional dispersion and multiscale structures, are considerably more delicate due to the lack of strong smoothing properties, which motivates further investigation in this direction.

In the context of the nonlinear Schr\"odinger equation, relatively less work has been done on the averaging principle. In \cite{MR3664698}, the authors have investigated an averaging principle for cubic nonlinear Schr\"odinger equations subject to rapidly oscillating potentials and external forces. The article  \cite{MR4087366} is devoted to the study of an averaging principle for a multiscale stochastic fractional Schr\"odinger equation and \cite{MR3800899} for a higher-order nonlinear Schr\"odinger equation with random fast oscillations. More recently, the averaging principle for a multiscale stochastic fractional Schr\"odinger–Korteweg–de Vries coupled system is obtained in \cite{MR4179788}. The equation of interest in \cite{MR4087366} has the following form:
\begin{align}\label{multiscalegao}
	\left\{
	\begin{aligned}
		&\mathrm{d} {u}^{\varepsilon} = \left[-i(-\Delta)^{\alpha} {u}^{\varepsilon} + F({u}^{\varepsilon}) + f({u}^{\varepsilon}, {v}^{\varepsilon}) \right]\mathrm{d}t + \sigma_1({u}^{\varepsilon}) \, \mathrm{d}B_1 && \text{in } \mathbb{T} \times [0,T],\\
		&\mathrm{d}{v}^{\varepsilon} = \frac{1}{\varepsilon} \left[(1 + i)\Delta {v}^{\varepsilon} + F({v}^{\varepsilon}) - \lambda {v}^{\varepsilon} + g({u}^{\varepsilon}, {v}^{\varepsilon}) \right]\mathrm{d}t + \frac{1}{\sqrt{\varepsilon}} \sigma_2({u}^{\varepsilon}, {v}^{\varepsilon}) \, \mathrm{d}B_2 && \text{in } \mathbb{T} \times [0,T],\\
		&{u}^{\varepsilon}(x,0) = {u}_0(x) && \text{in } \mathbb{T}, \\
		&{v}^{\varepsilon}(x,0) = {v}_0(x) && \text{ in } \mathbb{T},
	\end{aligned}
	\right.
\end{align}
where $F({\mathfrak{u}})=-(1+i \gamma) |{u}|^{2}{u}$ is the cubic polynomial nonlinearity and $(-\Delta)^{\alpha}$ is $2\alpha$-order fractional Laplacian operator with $\alpha \in (\tfrac{1}{2},1)$. 

In our case, the equation of interest is \eqref{multiscaleeq}, which in comparison to \eqref{multiscalegao} allows a general polynomial nonlinearity and consider the fractional operator in fast motion equation as well. The presence of the fractional operator weakens the smoothing and regularizing properties typically available for the parabolic operator in the fast motion equation, making it more difficult to establish uniform estimates and to analyze the ergodicity of the frozen fast subsystem.
\subsection{Novelties and difficulties} 
%
%
%
\indent In \cite{MR4087366}, the averaging principle is established for a multiscale stochastic fractional Schr\"odinger equation with cubic nonlinearity in which the slow component is governed by a fractional dispersive operator and the fast component is parabolic. The analysis crucially exploits the dissipative and smoothing properties of the fast stochastic heat equation and relies on the specific structure of the cubic term and the corresponding energy estimates.

In the present work, we consider a more general and challenging setting in which both the slow and the fast components are driven by fractional operators and involve general polynomial nonlinearities. The fractional nature of the fast dynamics substantially weakens the smoothing and regularizing properties typically available in the classical setting, thereby making it more difficult to establish uniform estimates and to analyze the ergodic behavior of the frozen fast subsystem.

Furthermore, the inclusion of general polynomial nonlinearities introduces additional analytical difficulties. In particular, it complicates the derivation of uniform a priori bounds and requires careful control of higher order growth terms arising from the coupling between the slow and fast variables. All technical lemmas required to address these issues are proved in \Cref{techlemmas}. 

A further difficulty stems from the fact that, unlike stochastic heat equations, fractional Schr\"odinger semigroups do not provide sufficient regularization to directly establish the time continuity required for the averaging procedure. To overcome this issue, we introduce a vanishing viscosity approximation for both components, which yields enhanced time regularity and uniform bounds independent of the viscosity parameter. These uniform estimates play a crucial role in implementing a Khasminskii-type time discretization scheme and in proving the strong convergence of the slow component toward the averaged equation. 
\subsection{Organization of the paper}
The paper is organized as follows. In \Cref{secass}, we present the assumptions and mathematical framework. First we introduce the notations in Subsection \ref{subsecmodelsetup}, then we state the necessary assumptions in Subsection \ref{main assumptions} and some preliminaries in Subsection \ref{subsecpreli}. \Cref{secmainresultandkey} is devoted to the main result (Subsection \ref{subsecmainresult}), the key strategies of the proof (Subsection \ref{keystrategy}) and some technical lemmas (Subsection \ref{techlemmas}). 

\Cref{secuniformes} focuses on establishing the uniform estimates for the approximated solutions. First, we established the uniform estimates for viscous approximated solution in \Cref{uniformestimates} and for viscous averaging solution in \Cref{viscousaveragingsolution}. Then, we estimate the error of viscous approximated solution from the original solution and the error between viscous averaging solution and averaging solution. 

In \Cref{secstepthree}, we prove the averaging proinciple for the viscous system \eqref{viscous}. First, we establish the H\"older continuity in time variable for the viscous approximated solution in Subsection \ref{subsecholdercts}. Next, we introduce the auxiliary processes in Subsection \ref{subsecdefofaux} and establish the error estmates of the auxiliary processes from viscous approximated solution in Subsection \ref{subsecerroraux}. Then, we estimate the error between the auxiliary processes and viscous averaging solution in Subsection \ref{subsecerrorviscous}.

Finally, in \Cref{secpfofmainresult}, we combine all the established uniform estimates, the error bounds and the averaging proinciple for the viscous system to prove our main result.
\section{Assumptions and mathematical framework}\label{secass}
\subsection{Model set up}\label{subsecmodelsetup}
We consider a filtered probability space $\mathfrak{U}=(\Omega, \mathcal{F}, \mathbb{F}, \mathbb{P})$, where \(\mathbb{F} = \{ \mathcal{F}_t \}_{t \geq 0}\), satisfying the usual conditions:
\begin{enumerate}
	\item $\mathbb{P}$ is complete on $(\Omega, \mathcal{F})$, 
	
	\item for each $t \ge 0$, $\mathcal{F}_t$ contains all $(\mathcal{F}, \mathbb{P})$-null sets,
	
	\item the filtration $\mathbb{F}$ is right continuous, that is, 
	\begin{align*}
		\mathcal{F}_t = \bigcap_{s > t} \mathcal{F}_s,
		\qquad \text{for all } t \ge 0.
	\end{align*}
\end{enumerate}


\noindent
We denote the space of all Lebesgue measurable square integrable functions on $\mathbb{T}$ by $L^2(\mathbb{T})$. We denote the inner product on $L^2(\mathbb{T})$ as $({\mathfrak{u}},{\mathfrak{v}})=\operatorname{Re}\int_{\mathbb{T}} {\mathfrak{u}} \bar{{\mathfrak{v}}} dx$ for any ${\mathfrak{u}}, {\mathfrak{v}} \in L^2(\mathbb{T})$. The norm in $ L^2(\mathbb{T})$ is defined as $\|{\mathfrak{u}}\|^2=({\mathfrak{u}},{\mathfrak{u}})$, for any ${\mathfrak{u}}\in L^2(\mathbb{T})$.
We denote by $H^1(\mathbb{T})$ the Sobolev space
\begin{align*}
	H^1(\mathbb{T}) 
	:= \left\{ {\mathfrak{u}} \in L^2(\mathbb{T}) : \partial_x {\mathfrak{u}} \in L^2(\mathbb{T}) \right\},
\end{align*}
endowed with the inner product
\begin{align*}
	({\mathfrak{u}}, {\mathfrak{v}})_{H^1}
	= \operatorname{Re} \int_{\mathbb{T}} 
	\left( {\mathfrak{u}} \overline{{\mathfrak{v}}} 
	+ \partial_x {\mathfrak{u}} \, \overline{\partial_x {\mathfrak{v}}} \right) dx.
\end{align*}
The space of all continuous functions ${\mathfrak{u}} : [0,T] \to H^1(\mathbb{T})$ is denoted by $C([0,T]; H^1(\mathbb{T}))$ and the norm is defined as
\begin{align*}
	\|{\mathfrak{u}}\|_{C([0,T];H^1)}
	:= \sup_{t \in [0,T]} \|{\mathfrak{u}}(t)\|_{H^1}.
\end{align*}
The space $L^2(\Omega, C([0,T]; H^1(\mathbb{T})))$ denotes the collection of all $H^1$-valued stochastic processes, namely
\begin{align*}
	L^2(\Omega, C([0,T]; H^1(\mathbb{T})))
	:= \Big\{ {\mathfrak{u}} : \Omega \to C([0,T]; H^1(\mathbb{T})) 
	\, \Big| \ 
	\mathbb{E} \Big[ \sup_{t \in [0,T]} 
	\|{\mathfrak{u}}(t)\|_{H^1}^2 \Big] < \infty \Big\}.
\end{align*}
\subsection{Standing hypothesis}\label{main assumptions}
To establish the strong convergence rate of the averaging principle for the multiscale SNLSE \eqref{multiscaleeq}, we consider the following assumptions on the operator $F,G$ and the noise coefficients $\Sigma_1, \Sigma_2$:
\vspace{0.1cm} 

\noindent \textbf{Assumption 1:} The maps $F,G: L^2(\mathbb{T}) \times L^2(\mathbb{T}) \to \mathbb{R}$ (they might be nonlinear) are globally Lipschitz continuous with linear growth rate. More precisely, for any ${\mathfrak{u}}, {\mathfrak{v}}, {\mathfrak{u}}_1, {\mathfrak{v}}_1, {\mathfrak{u}}_2, {\mathfrak{v}}_2 \in L^2(\mathbb{T})$, there exist positive numbers $L_F, L_G$ such that 
\begin{align*}
	&\|F({\mathfrak{u}},{\mathfrak{v}})\| \leq L_F(1 + \|{\mathfrak{u}}\| + \|{\mathfrak{v}}\|), \\
	&\|G({\mathfrak{u}},{\mathfrak{v}})\| \leq L_G(1 + \|{\mathfrak{u}}\| + \|{\mathfrak{v}}\|), \\
	&	\|F({\mathfrak{u}}_1, {\mathfrak{v}}_1) - F({\mathfrak{u}}_2, {\mathfrak{v}}_2)\| \leq L_F(\|{\mathfrak{u}}_1 - {\mathfrak{u}}_2\| + \|{\mathfrak{v}}_1 - {\mathfrak{v}}_2\|), \\
	&\|G({\mathfrak{u}}_1, {\mathfrak{v}}_1) - G({\mathfrak{u}}_2, {\mathfrak{v}}_2)\| \leq L_G(\|{\mathfrak{u}}_1 - {\mathfrak{u}}_2\| + \|{\mathfrak{v}}_1 - {\mathfrak{v}}_2\|).
\end{align*}
For any ${\mathfrak{u}} \in H^1(\mathbb{T})$, ${\mathfrak{v}} \in L^2(\mathbb{T})$,
\begin{align*}
	\|F({\mathfrak{u}},{\mathfrak{v}})\|_{H^1} &\leq L_F(1 + \|{\mathfrak{u}}\|_{H^1} + \|{\mathfrak{v}}\|).
\end{align*}
\textbf{Assumption 2:}
The noise coefficient $\Sigma_1: L^2(\mathbb{T})\to \mathbb{R}$ is globally Lipschitz with linear growth rate that is, for any ${\mathfrak{u}}, {\mathfrak{u}}_1, {\mathfrak{v}}_1, {\mathfrak{u}}_2, {\mathfrak{v}}_2 \in L^2(\mathbb{T})$ there exists a positive number $L_{\Sigma_1}$ such that 
\begin{align*}
	&	\|\Sigma_1({\mathfrak{u}})\| \leq L_{\Sigma_1}(1 + \|{\mathfrak{u}}\|), \\
	&	\|\Sigma_1({\mathfrak{u}}_1) - \Sigma_1({\mathfrak{u}}_2)\| \leq L_{\Sigma_1}(\|{\mathfrak{u}}_1 - {\mathfrak{u}}_2\| + \|{\mathfrak{v}}_1 - {\mathfrak{v}}_2\|), \\
\end{align*}
and for any ${\mathfrak{u}} \in H^1(\mathbb{T})$, ${\mathfrak{v}} \in L^2(\mathbb{T})$,
\begin{align*}
	\|\Sigma_1({\mathfrak{u}})\|_{H^1} &\leq L_{\Sigma_1}(1 + \|{\mathfrak{u}}\|_{H^1}).
\end{align*}
The noise coefficient $\Sigma_2: L^2(\mathbb{T})\times L^2(\mathbb{T}) \to \mathbb{R}$ is bounded and globally Lipschitz on the first variable, that is, for any ${\mathfrak{u}}, {\mathfrak{v}}, {\mathfrak{u}}_1, {\mathfrak{v}}_1, {\mathfrak{u}}_2, {\mathfrak{v}}_2 \in L^2(\mathbb{T})$, there exist positive constants $ L_{\Sigma_2}, M$ such that
\begin{align*}
	&	\|\Sigma_2({\mathfrak{u}},{\mathfrak{v}})\| \leq M, \\
	&	\|\Sigma_2({\mathfrak{u}}_1, {\mathfrak{v}}_1) - \Sigma_2({\mathfrak{u}}_2, {\mathfrak{v}}_2)\| \leq L_{\Sigma_2}\|{\mathfrak{u}}_1 - {\mathfrak{u}}_2\|.
\end{align*}
For technical evaluation in proving \Cref{diffnonrenega} and \Cref{propinnernegative}, which are crucial for the subsequent analysis, we impose the following assumption:\\
\textbf{Assumption 3:} We also assume $\alpha, \rho \in \left(\tfrac{1}{2}, 1\right)$, $|\gamma| \leq \frac{2\sqrt{\beta}}{\beta-1}$ and $\lambda > 3L_G + 2L_{\Sigma_2}^2$, where $\alpha, \rho,\beta, \gamma, \lambda$ are appearing in \eqref{multiscaleeq}.

\subsection{Preliminaries}\label{subsecpreli}
In this subsection, we gather the basic definitions and preliminary results that will be used throughout the sequel.
%
\begin{lem}[The Young inequality]\label{Youngs inequality}
	Let $x,y \in [0,\infty)$ and $\eta>0$. Then 
	\begin{align*}
		xy\leq \eta x^p+ \frac{1}{(\eta p)^{\frac{q}{p}}q}y^q,
	\end{align*}
	where $1<p<\infty, \frac{1}{p}+\frac{1}{q}+1.$
\end{lem}
\begin{lem}[Burkholder-Davis-Gundy inequality, {\cite[Theorem 1]{MR3463679}}]\label{BurkholderDavisGundy inequality}
	Let $\{W(t)\}_{t\geq0}$ be a standard Brownian motion and $F$ be an $L^2$-valued continuous $\mathbb{F}$-local martingale. Then, for any $p>0$, there exists a constant $C(p)>0$ such that for any stopping time $\tau$,
	\begin{align*}
		\mathbb{E}\Big[\sup_{0\le t\le \tau}\Big|\int_0^t F(s)\,dW\Big|^p\Big]
		\le C(p)\,\mathbb{E}\left(\int_0^\tau F^2(s)\,ds\right)^{\frac{p}{2}}.
	\end{align*}
\end{lem}
	\begin{lem}[The Gr\"onwall inequality, {\cite[P-625]{MR2597943}}]\label{Gronwalls inequality}
		Let $v(t)$ be a nonnegative integrable function on $(0,T)$ that satisfies 
		\begin{align*}
			v(t)\le C_1\int_0^t v(s)\,ds + C_2,
		\end{align*}
		for a.e. $t$ and $C_1,C_2\geq0$. Then we have
		\begin{align*}
			v(t)\le C_2\bigl(1 + C_1 t e^{C_1 t}\bigr),
		\end{align*}
		for a.e. $t\in(0,T)$.
	\end{lem}
	\begin{defi}[Fractional Laplacian]\label{deffraclap}
		If $v$ can be expressed by the Fourier series $v = \sum_{k \in \mathbb{Z}} v_k e^{ikx},$ the fractional laplacian operator $(-\Delta)^{\alpha}$ is defined as
		\begin{align*}
			(-\Delta)^{\alpha} v := \sum_{k \in \mathbb{Z}} |k|^{2\alpha} v_k e^{ikx}.
		\end{align*}
	\end{defi}
	\begin{lem}\label{fractionallemma}
		If $h,k \in H^{2\alpha}(\mathbb{T})$, then
		\begin{align*}
			\int_{\mathbb{T}} (-\Delta)^{\alpha} h \cdot k \, dx
			= \int_{\mathbb{T}} (-\Delta)^{\alpha_1} h \cdot (-\Delta)^{\alpha_2} k \, dx,
		\end{align*}
		where $\alpha_1,\alpha_2$ are nonnegative constants satisfying $\alpha_1 + \alpha_2 = \alpha.$
	\end{lem}

\section{Statement of the main result and key strategies of the proof}\label{secmainresultandkey}
In this section, we state the main result and give a brief sketch of the proof. Then we establish some technical lemmas, also which will be used later in the consequent analysis.
\subsection{Main result}\label{subsecmainresult}
\begin{thm}\label{mainresult}
	Under the assumptions in \Cref{main assumptions}, let $({\mathfrak{u}}_0, {\mathfrak{v}}_0) \in H^1(\mathbb{T}) \times H^1(\mathbb{T})$. Let $({\mathfrak{u}}^\varepsilon, {\mathfrak{v}}^\varepsilon)$ denote the solution of \eqref{multiscaleeq} and $\bar{{\mathfrak{u}}}$ be the solution of the corresponding effective dynamics equation
	\begin{equation}\label{averagedequation}
		\left\{
		\begin{aligned}
			&\mathrm{d}\bar{{\mathfrak{u}}} = [-i(-\Delta)^{\alpha} \bar{{\mathfrak{u}}} +\mathcal{N}(\bar{{\mathfrak{u}}})+ \bar{F}(\bar{{\mathfrak{u}}})] \mathrm{d}t + \Sigma_1(\bar{{\mathfrak{u}}}) \, \mathrm{d}\mathcal{W}_1 \quad && \text{in } \mathbb{T} \times [0,T], \\
			&\bar{{\mathfrak{u}}}(x,0) = {\mathfrak{u}}_0(x) \quad &&\text{in } \mathbb{T}.
		\end{aligned}
		\right. 
	\end{equation}
	Then, for any $T > 0$, we have the following convergence:
	\begin{equation*}
		\lim_{\varepsilon \to 0} \sup_{0 \leq t \leq T} \mathbb{E} \big[ \|{\mathfrak{u}}^\varepsilon(t) - \bar{{\mathfrak{u}}}(t)\|^2\big] = 0,
	\end{equation*}
	where the average drift is defined by 
	\begin{equation}\label{fbar}
		\bar{F}({\mathfrak{u}}) := \int_{L^2(\mathbb{T})} F({\mathfrak{u}},{\mathfrak{v}}) \, \mu^{\mathfrak{u}}(\mathrm{d}{\mathfrak{v}}).
	\end{equation}
	Here $\mu^{\mathfrak{u}}$ is the unique invariant measure associated with the fast dynamics corresponding to the frozen slow variable
	\begin{equation}
		\left\{
		\begin{aligned}
			&\mathrm{d}{\mathfrak{v}} = [(1+i)(-\Delta)^{\rho} {\mathfrak{v}} + \mathcal{N}({\mathfrak{v}}) - \lambda {\mathfrak{v}} + G({\mathfrak{u}},{\mathfrak{v}})] \mathrm{d}t + \Sigma_2({\mathfrak{u}},{\mathfrak{v}}) \, \mathrm{d}\mathcal{W}_2 \quad &&\text{in } \mathbb{T} \times [0,T], \\
			&{\mathfrak{v}}(x,0) = {\mathfrak{v}}_0(x) \quad &&\text{in } \mathbb{T},
		\end{aligned}
		\right.
	\end{equation}
	where ${\mathfrak{u}} \in L^2(\mathbb{T})$.
\end{thm}
\begin{prop}
	For any fixed $\varepsilon\in (0,1)$ and initial data $({\mathfrak{u}}_0,{\mathfrak{v}}_0) \in H^1(\mathbb{T})\times H^1(\mathbb{T})$, \eqref{multiscaleeq} has a unique solution $({\mathfrak{u}}^{\varepsilon},{\mathfrak{v}}^{\varepsilon})\in L^2(\Omega, C([0,T];L^2(\mathbb{T})))\times L^2(\Omega, C([0,T];L^2(\mathbb{T}))) $.
\end{prop}
\begin{proof}
	For a detailed proof, one can use the Galerkin method. The idea is similar to \cite{MR3483877, MR3359820}.
\end{proof}
\subsection{Key strategies of the proof:}\label{keystrategy}
To study the averaging principle, H\"older continuity in time variable for ${\mathfrak{u}}^{\varepsilon}$ plays a crucial role and the smoothness of the semigroup $\{S_0(t)\}_{t\geq 0}$ is not enough to get it. In particular, the smoothness property \eqref{smoothness} is important to get H\"older continuity in time variable, but $\{S_0(t)\}_{t\geq 0}$ does not provide that. To fix this complexity, we add a viscous term $\nu \Delta {\mathfrak{u}} $ ($\nu >0$ is the viscosity). With the help of smoothing effect of the semigroup $\{S_{\nu}(t)\}_{t\geq 0}$ associated to $A_\nu:=-i(-\Delta)^{\alpha}{\mathfrak{u}}+\nu \Delta {\mathfrak{u}}$, we can establish the H\"older continuity of time variable for ${\mathfrak{u}}^{\varepsilon,\nu}$. This method is known as vanishing viscosity.\\
\textit{\textbf{Step I.}} 
Let us consider the viscous system, 
\begin{equation}\label{viscous}
	\left\{
	\begin{aligned}
		&\mathrm{d}{\mathfrak{u}}^{\varepsilon, \nu} = \left[A_\nu {\mathfrak{u}}^{\varepsilon, \nu} + \mathcal{N}({\mathfrak{u}}^{\varepsilon, \nu}) + F({\mathfrak{u}}^{\varepsilon, \nu}, {\mathfrak{v}}^{\varepsilon, \nu}) \right]\mathrm{d}t + \Sigma_1({\mathfrak{u}}^{\varepsilon, \nu}) \, \mathrm{d}\mathcal{W}_1 && \text{in } T \times [0,T],\\
		&\mathrm{d}{\mathfrak{v}}^{\varepsilon, \nu} = \frac{1}{\varepsilon} \left[(1 + i)(-\Delta)^{\rho} {\mathfrak{v}}^{\varepsilon, \nu} + \mathcal{N}({\mathfrak{v}}^{\varepsilon, \nu}) - \lambda {\mathfrak{v}}^{\varepsilon, \nu} + G({\mathfrak{u}}^{\varepsilon, \nu}, {\mathfrak{v}}^{\varepsilon, \nu}) \right]\mathrm{d}t + \frac{1}{\sqrt{\varepsilon}} \Sigma_2({\mathfrak{u}}^{\varepsilon, \nu}, {\mathfrak{v}}^{\varepsilon, \nu}) \, \mathrm{d}\mathcal{W}_2 && \text{in } T \times [0,T],\\
		&{\mathfrak{u}}^{\varepsilon, \nu}(x,0) = {\mathfrak{u}}_0(x) && \text{in } T, \\
		&{\mathfrak{v}}^{\varepsilon, \nu}(x,0) = {\mathfrak{v}}_0(x) && \text{ in } T,
	\end{aligned}
	\right.
\end{equation}
 and the viscous averaged equation is, 
\begin{equation}\label{viscousavg}
	\left\{
	\begin{aligned}
		&\mathrm{d}\bar{{\mathfrak{u}}}^{\nu} = [A_\nu \bar{{\mathfrak{u}}}^{\nu}+\mathcal{N}(\bar{{\mathfrak{u}}}^{\nu}) + \bar{F}(\bar{{\mathfrak{u}}}^{\nu})] \mathrm{d}t + \Sigma_1(\bar{{\mathfrak{u}}}^{\nu}) \, \mathrm{d}\mathcal{W}_1 \quad &&\text{in } \mathbb{T} \times [0,T], \\
		&\bar{{\mathfrak{u}}}^{\nu}(x,0) = {\mathfrak{u}}_0(x) \quad &&\text{in } \mathbb{T}.
	\end{aligned}
	\right. 
\end{equation}
\textit{\textbf{Step II.}} 
Our first aim is to prove,
\begin{equation}
	\lim_{\nu \to 0} \sup_{\varepsilon \in (0,1)} \sup_{0 \leq t \leq T} \mathbb{E} \big[\|{\mathfrak{u}}^{\varepsilon, \nu}(t) - {\mathfrak{u}}^{\varepsilon}(t)\|^2\big] = 0,
	\quad 
	\lim_{\nu \to 0} \sup_{0 \leq t \leq T} \mathbb{E} \big[\|\bar{{\mathfrak{u}}}^{\nu}(t) - \bar{{\mathfrak{u}}}(t)\|^2 \big]= 0.
\end{equation}
using apriori estimates of ${\mathfrak{u}}^{\varepsilon, \nu}$ and ${\mathfrak{v}}^{\varepsilon, \nu}$ which are uniform in $\varepsilon, \nu$.\\

\noindent\textit{\textbf{Step III.}} Exploiting the smoothing properties of the semigroup $\{S_{\nu}(t)\}_{t\geq 0}$ together with the apriori estimates, we establish averaging principle for \eqref{viscousavg}, i.e. for any fixed $\nu \in (0,1)$,
\begin{equation}
	\lim_{\varepsilon \to 0} 
	\mathbb{E} \left[ \sup_{0 \leq t \leq T} \|{\mathfrak{u}}^{\varepsilon, \nu}(t) - \bar{{\mathfrak{u}}}^{\nu}(t)\|^2 \right] = 0.
\end{equation}
\textit{\textbf{Step IV.}} Then by limiting argument, as the viscosity coefficient $\nu$ goes to zero, we conclude \autoref{mainresult}. 
	\subsection{Some technical lemmas}\label{techlemmas}
	In this subsection, we establish several technical lemmas that will be used repeatedly in proving the main result.
\begin{lem}\label{liptypeofnonlinear}
	For any ${\mathfrak{u}}_1, {\mathfrak{u}}_2 \in \mathbb{C}$, the following holds:
	\begin{equation*}
		\big|\mathcal{N}({\mathfrak{u}}_1)- \mathcal{N}({\mathfrak{u}}_2)\big| \leq C(\big|{\mathfrak{u}}_1\big|^{\beta-1}+\big|{\mathfrak{u}}_2\big|^{\beta-1})\big|{\mathfrak{u}}_1- {\mathfrak{u}}_2\big|.
	\end{equation*}
\end{lem}
\begin{proof}
	We have given,
	\begin{align*}
		\mathcal{N}({\mathfrak{u}})= -(1+i \gamma) |{\mathfrak{u}}|^{\beta-1} {\mathfrak{u}}.
	\end{align*}
	Therefore,
	\begin{align*}
		\big|\mathcal{N}({\mathfrak{u}}_1) - \mathcal{N}({\mathfrak{u}}_2)\big| &=\big|-(1+i \gamma) \{|{\mathfrak{u}}_1|^{\beta-1} {\mathfrak{u}}_1-|{\mathfrak{u}}_2|^{\beta-1} {\mathfrak{u}}_2\}\big|\\
		&\leq C \big||{\mathfrak{u}}_1|^{\beta-1} {\mathfrak{u}}_1-|{\mathfrak{u}}_2|^{\beta-1} {\mathfrak{u}}_2\big|.
	\end{align*}
	Consider 
	\begin{align*}
		N_1({\mathfrak{u}})= |{\mathfrak{u}}|^{\beta-1} {\mathfrak{u}},
	\end{align*}
	then
	\begin{align*}
		N_1'({\mathfrak{u}})(h)=|{\mathfrak{u}}|^{\beta-1}h+(\beta-1){\mathfrak{u}}|{\mathfrak{u}}|^{\beta-3} \operatorname{Re}({\mathfrak{u}}\bar{h}).
	\end{align*}
	Employing the fundamental theorem of calculas and Lagrange's mean value theorem, we infer
	\begin{align*}
		&\big|N_1({\mathfrak{u}}_1) - N_1({\mathfrak{u}}_2)\big| =\Big|\int_{0}^{1}\frac{d}{d\theta}N_1(\theta {\mathfrak{u}}_1+(1-\theta){\mathfrak{u}}_2)d\theta\Big|=\Big|\int_{0}^{1} N_1'(\theta {\mathfrak{u}}_1+(1-\theta){\mathfrak{u}}_2) ({\mathfrak{u}}_1-{\mathfrak{u}}_2) d\theta\Big|\\
		&\leq \int_{0}^{1} \big|N_1'(\theta {\mathfrak{u}}_1+(1-\theta){\mathfrak{u}}_2) ({\mathfrak{u}}_1-{\mathfrak{u}}_2)\big| d\theta \leq\beta \int_{0}^{1} \big|\theta {\mathfrak{u}}_1+(1-\theta){\mathfrak{u}}_2\big|^{\beta-1} \big|{\mathfrak{u}}_1-{\mathfrak{u}}_2\big| d\theta\\
		&\leq\beta \int_{0}^{1} (|{\mathfrak{u}}_1|+|{\mathfrak{u}}_2|)^{\beta-1} |{\mathfrak{u}}_1-{\mathfrak{u}}_2| d\theta \leq\beta (|{\mathfrak{u}}_1|^{\beta-1}+|{\mathfrak{u}}_2|^{\beta-1}) |{\mathfrak{u}}_1-{\mathfrak{u}}_2|.
	\end{align*}
	Therefore, we conclude 
	\begin{equation*}
		\big|\mathcal{N}({\mathfrak{u}}_1)- \mathcal{N}({\mathfrak{u}}_2)\big| \leq C(\big|{\mathfrak{u}}_1\big|^{\beta-1}+\big|{\mathfrak{u}}_2\big|^{\beta-1})\big|{\mathfrak{u}}_1- {\mathfrak{u}}_2\big|.
	\end{equation*}
\end{proof}
\begin{lem}(\cite[Lemma 2.2]{MR1224619})\label{imlemma}
	For any $0 < \beta < \infty$ and $x,y \in \mathbb{C}$, the following inequality holds true:
	\begin{equation}
		\bigl|\operatorname{Im}(x-y)\bigl(\bar x|x|^{\beta -1}-\bar y|y|^{\beta -1}\bigr)\bigr|
		\le
		\frac{|\beta -1|}{2\sqrt{\beta}}\,
		\operatorname{Re}(x-y)\bigl(\bar x|x|^{\beta -1}-\bar y|y|^{\beta -1}\bigr).
	\end{equation}
\end{lem}

\begin{lem}\label{imlemmafornon}
	For any $1< \beta < \infty$ and ${\mathfrak{u}}_1, {\mathfrak{u}}_2 \in \mathbb{C}$, we have
	\begin{equation}
		\bigl|\operatorname{Im}\big( ({N_1}({\mathfrak{u}}_1) - {N_1}({\mathfrak{u}}_2))\, \overline{({\mathfrak{u}}_1 - {\mathfrak{u}}_2)}\, \big)\bigr|
		\le
		\frac{|\beta -1|}{2\sqrt{\beta}}\,
		\operatorname{Re} \big[({N_1}({\mathfrak{u}}_1) - {N_1}({\mathfrak{u}}_2))\, \overline{({\mathfrak{u}}_1 - {\mathfrak{u}}_2)}\, \big].
	\end{equation}
\end{lem}
\begin{proof}
	Using the above lemma (\Cref{imlemma}), we get
	\begin{align*}
		&\bigl|\operatorname{Im}\big( ({N_1}({\mathfrak{u}}_1) - {N_1}({\mathfrak{u}}_2))\, \overline{({\mathfrak{u}}_1 - {\mathfrak{u}}_2)}\, \big)\bigr|=\bigl|-\operatorname{Im}\overline{\big( ({N_1}({\mathfrak{u}}_1) - {N_1}({\mathfrak{u}}_2))\, \overline{({\mathfrak{u}}_1 - {\mathfrak{u}}_2)}\, \big)}\bigr|\\
		&=\bigl|-\operatorname{Im}\overline{\big( ({N_1}({\mathfrak{u}}_1) - {N_1}({\mathfrak{u}}_2))}\,({\mathfrak{u}}_1 - {\mathfrak{u}}_2)\, \big)\bigr|=\bigl|\operatorname{Im}\overline{\big( ({N_1}({\mathfrak{u}}_1) - {N_1}({\mathfrak{u}}_2))}\,({\mathfrak{u}}_1 - {\mathfrak{u}}_2)\, \big)\bigr|\\
		&=\bigl|\operatorname{Im}\big( (|{\mathfrak{u}}_1|^{\beta-1} \bar{\mathfrak{u}}_1 - |{\mathfrak{u}}_2|^{\beta-1} \bar{\mathfrak{u}}_2)\,({\mathfrak{u}}_1 - {\mathfrak{u}}_2)\, \big)\bigr|\le
		\frac{|\beta -1|}{2\sqrt{\beta}}\,
		\operatorname{Re}\big[ (|{\mathfrak{u}}_1|^{\beta-1} \bar{\mathfrak{u}}_1 - |{\mathfrak{u}}_2|^{\beta-1} \bar{\mathfrak{u}}_2)\,({\mathfrak{u}}_1 - {\mathfrak{u}}_2)\, \big]\\
		&\le
		\frac{|\beta -1|}{2\sqrt{\beta}}\,
		\operatorname{Re}\big[\overline{ ({N_1}({\mathfrak{u}}_1) - {N_1}({\mathfrak{u}}_2))}\,({\mathfrak{u}}_1 - {\mathfrak{u}}_2)\big]=
		\frac{|\beta -1|}{2\sqrt{\beta}}\,
		\operatorname{Re}\big[ ({N_1}({\mathfrak{u}}_1) - {N_1}({\mathfrak{u}}_2))\,\overline{({\mathfrak{u}}_1 - {\mathfrak{u}}_2)}\big].
	\end{align*}
\end{proof}
\begin{lem}\label{nonlinearitynegative}
	For any $1 < \beta < \infty$ and ${\mathfrak{u}}_1, {\mathfrak{u}}_2 \in \mathbb{C}$, we have
	\begin{equation*}
		\operatorname{Re} \big[ ({N_1}({\mathfrak{u}}_1) - {N_1}({\mathfrak{u}}_2))\, \overline{({\mathfrak{u}}_1 - {\mathfrak{u}}_2)}\, \big] \leq 0.
	\end{equation*}
\end{lem}
\begin{proof}
	We have, 
	\begin{align*}
		\operatorname{Re} \big[ ({N_1}({\mathfrak{u}}_1) - {N_1}({\mathfrak{u}}_2))\, \overline{({\mathfrak{u}}_1 - {\mathfrak{u}}_2)}\, \big]= \operatorname{Re} \big[ (|{\mathfrak{u}}_1|^{\beta-1} {\mathfrak{u}}_1 - |{\mathfrak{u}}_2|^{\beta-1} {\mathfrak{u}}_2)\, \overline{({\mathfrak{u}}_1 - {\mathfrak{u}}_2)}\, \big].
	\end{align*}
	Let
	\begin{align*}
		{\mathfrak{u}_1} = r_1(\cos\alpha_1 + i\sin\alpha_1), 
		\qquad
		{\mathfrak{u}_2} = r_2(\cos\alpha_2 + i\sin\alpha_2),
	\end{align*}
	where $r_1,r_2 \ge 0$ and $\alpha_1,\alpha_2 \in [0,2\pi)$.
	Now,
	\begin{align*}
		\operatorname{Re}\Big( |{\mathfrak{u}_1}|^{\beta-1}{\mathfrak{u}_1}(\overline{{\mathfrak{u}_1}}-\overline{{\mathfrak{u}_2}}) \Big)
		= r_1^{\beta+1}
		- r_1^\beta r_2 \cos(\alpha_1-\alpha_2),	
	\end{align*}
	and
	\begin{align*}
		\operatorname{Re}\Big( |{\mathfrak{u}_2}|^{\beta-1}{\mathfrak{u}_2}(\overline{{\mathfrak{u}_1}}-\overline{{\mathfrak{u}_2}}) \Big)
		= r_1 r_2^\beta \cos(\alpha_1-\alpha_2)
		- r_2^{\beta+1}.
	\end{align*}
	Combining,
	\begin{align*}
		\operatorname{Re}\Big( (|{\mathfrak{u}_1}|^{\beta-1}{\mathfrak{u}_1} - |{\mathfrak{u}_2}|^{\beta-1}{\mathfrak{u}_2})
		(\overline{{\mathfrak{u}_1}}-\overline{{\mathfrak{u}_2}}) \Big)
		= r_1^{\beta+1} + r_2^{\beta+1}
		- (r_1^\beta r_2 + r_1 r_2^\beta)\cos(\alpha_1-\alpha_2).
	\end{align*}
	Since $\cos(\alpha_1-\alpha_2) \le 1$ and $(r_1^{\beta} - r_2^{\beta}), \,(r_1 - r_2)$ have same sign, we obtain
	\begin{align*}
		\operatorname{Re}\Big( (|{\mathfrak{u}_1}|^{\beta-1}{\mathfrak{u}_1} - |{\mathfrak{u}_2}|^{\beta-1}{\mathfrak{u}_2})
		(\overline{{\mathfrak{u}_1}}-\overline{{\mathfrak{u}_2}}) \Big)
		&\geq r_1^{\beta+1} + r_2^{\beta+1}
		- (r_1^\beta r_2 + r_1 r_2^\beta)\\
		&=(r_1^{\beta} - r_2^{\beta})(r_1 - r_2)\geq0.
	\end{align*}
	This completes the proof.
\end{proof}
\begin{lem}\label{diffnonrenega}
	For any ${\mathfrak{u}}_1, {\mathfrak{u}}_2 \in \mathbb{C}$, and $|\gamma|\leq \frac{2\sqrt{\beta}}{\beta-1}$, we have
	\begin{equation*}
		\operatorname{Re} \big[ (\mathcal{N}({\mathfrak{u}}_1) - \mathcal{N}({\mathfrak{u}}_2))\, \overline{({\mathfrak{u}}_1 - {\mathfrak{u}}_2)}\, \big] \leq 0.
	\end{equation*}
	Moreover, if $|\gamma| \leq \frac{2\sqrt{\beta}}{\beta-1}$, for any ${\mathfrak{u}}_1, {\mathfrak{u}}_2 \in H^1(\mathbb{T})$,
	\begin{equation*}
		\big( \mathcal{N}({\mathfrak{u}}_1) - \mathcal{N}({\mathfrak{u}}_2), {\mathfrak{u}}_1 - {\mathfrak{u}}_2 \big) \leq 0.
	\end{equation*}
\end{lem}
\begin{proof}
	We observe that
	\begin{align*}
		&\operatorname{Re} \big[(\mathcal{N}({\mathfrak{u}}_1) - \mathcal{N}({\mathfrak{u}}_2))\, \overline{({\mathfrak{u}}_1 - {\mathfrak{u}}_2)}\, \big]=\operatorname{Re} \big[(-1-i\gamma)({N_1}({\mathfrak{u}}_1) - {N_1}({\mathfrak{u}}_2))\, \overline{({\mathfrak{u}}_1 - {\mathfrak{u}}_2)}\, \big]\\
		&=-\operatorname{Re} \big[({N_1}({\mathfrak{u}}_1) - {N_1}({\mathfrak{u}}_2))\, \overline{({\mathfrak{u}}_1 - {\mathfrak{u}}_2)}\, \big]-\gamma\operatorname{Re} \big[i({N_1}({\mathfrak{u}}_1) - {N_1}({\mathfrak{u}}_2))\, \overline{({\mathfrak{u}}_1 - {\mathfrak{u}}_2)}\, \big] \\
		&=-\operatorname{Re} \big[({N_1}({\mathfrak{u}}_1) - {N_1}({\mathfrak{u}}_2))\, \overline{({\mathfrak{u}}_1 - {\mathfrak{u}}_2)}\, \big]+\gamma\operatorname{Im} \big[({N_1}({\mathfrak{u}}_1) - {N_1}({\mathfrak{u}}_2))\, \overline{({\mathfrak{u}}_1 - {\mathfrak{u}}_2)}\, \big].
	\end{align*}
	Since $|\gamma|\leq \frac{2\sqrt{\beta}}{\beta-1}$, using \Cref{imlemmafornon} and \ref{nonlinearitynegative}, we get
	\begin{align*}
		&\operatorname{Re} \big[(\mathcal{N}({\mathfrak{u}}_1) - \mathcal{N}({\mathfrak{u}}_2))\, \overline{({\mathfrak{u}}_1 - {\mathfrak{u}}_2)}\, \big]\\
		&\leq-\operatorname{Re} \big[({N_1}({\mathfrak{u}}_1) - {N_1}({\mathfrak{u}}_2))\, \overline{({\mathfrak{u}}_1 - {\mathfrak{u}}_2)}\, \big]+\frac{\gamma(\beta -1)}{2\sqrt{\beta}}\,
		\operatorname{Re} \big[({N_1}({\mathfrak{u}}_1) - {N_1}({\mathfrak{u}}_2))\, \overline{({\mathfrak{u}}_1 - {\mathfrak{u}}_2)}\, \big]\\
		&\leq \left(\frac{\gamma(\beta -1)}{2\sqrt{\beta}}-1\right)\,
		\operatorname{Re} \big[({N_1}({\mathfrak{u}}_1) - {N_1}({\mathfrak{u}}_2))\, \overline{({\mathfrak{u}}_1 - {\mathfrak{u}}_2)}\, \big]\\	
		&\leq 0.				
	\end{align*}
	This completes the proof.
\end{proof}
\begin{lem}\label{propinnernegative}
	For any ${\mathfrak{u}} \in H^1(\mathbb{T})$, and $|\gamma|\leq \frac{2\sqrt{\beta}}{\beta-1}$, 
	\begin{equation*}
		\operatorname{Re} \big[\big( \mathcal{N}({\mathfrak{u}}) \big)_x\bar{\mathfrak{u}}_x \big] \leq 0.
	\end{equation*}
\end{lem}
\begin{proof}
	A detailed proof is given in \cite[Lemma 1]{MR0867341}. Our result follows by choosing $p=\beta-1$ in that reference.
\end{proof}
\begin{lem}\label{Itoapplicationlem}
	Let ${\mathfrak{u}}$ be the solution of $
	d{\mathfrak{u}} = (A_\nu {\mathfrak{u}} + F)dt + G\,d\mathcal{W}_1 \quad \text{in } \mathbb{T}$. Then,
	\begin{align}\label{a1}
		d\|{\mathfrak{u}}\|^2 = [-2\nu \|{\mathfrak{u}}_x\|^2 + 2({\mathfrak{u}},F) + \|G\|^2]dt + 2({\mathfrak{u}},G)d\mathcal{W}_1.
	\end{align}
	Therefore, we have
	\begin{align*}
		\|{\mathfrak{u}}(t)\|^2 + 2\nu \int_0^t \|{\mathfrak{u}}_x\|^2 ds
		= \|{\mathfrak{u}}(0)\|^2 + 2\int_0^t ({\mathfrak{u}},F)ds + \int_0^t \|G\|^2 ds + 2\int_0^t ({\mathfrak{u}},G)d\mathcal{W}_1.
	\end{align*}
\end{lem} 
\begin{proof}
	For a detail proof, see \cite[Lemma 3.1]{MR4087366}.
\end{proof}

\begin{lem}\label{property of semigroup}
	For $\nu \ge 0$, $A_\nu$ generates a $C_0$-semigroup $S_\nu$ of contractions on $L^2(\mathbb{T})$.  
	Moreover, for $\nu > 0$, there exists a constant $C$ such that for any ${\mathfrak{u}}_0 \in L^2(\mathbb{T})$,
	\begin{align*}
		\|S_\nu(t){\mathfrak{u}}_0\| &\le \|{\mathfrak{u}}_0\|, \quad t \in [0,T], \\
		\|S_\nu(t){\mathfrak{u}}_0\|_{H^1} &\le \frac{C}{\sqrt{t}} \|{\mathfrak{u}}_0\|, \quad t \in (0,T],
	\end{align*}
	where $C = C(\nu, T)$.
\end{lem}
\begin{proof}
	For a detail proof, see \cite[Proposition 3.1]{MR4087366}.
\end{proof}
\begin{lem}\label{difffbarf}
	For the averaged drift $\bar F$ defined in \eqref{fbar}, we have 
	\begin{align}\label{difoffand fbar}
		\|F({\mathfrak{u}},{\mathfrak{v}})- \bar F({\mathfrak{u}})\|^{2p}\leq C(1+\|{\mathfrak{u}}\|^{2p}+\|{\mathfrak{v}}\|^{2p}).
	\end{align}
\end{lem}
\begin{proof}
	We have 
	\begin{align*}
		&\|F({\mathfrak{u}},{\mathfrak{v}})- \bar F({\mathfrak{u}})\| \leq 	\|F({\mathfrak{u}},{\mathfrak{v}})\|+\| \bar F({\mathfrak{u}})\|\leq L_F \left(1+\|{\mathfrak{u}}\| + \|{\mathfrak{v}}\|\right)+ \int_{L^2(\mathbb{T})} L_F \left(  1+\|{\mathfrak{u}}\| + \|{\mathfrak{v}}\|  \right) \mu(\mathrm{d}{\mathfrak{v}})\\
		&\leq C \left(  1+\|{\mathfrak{u}}\| + \|{\mathfrak{v}}\|  \right) + C\int_{L^2(\mathbb{T})} \|{\mathfrak{v}}\|   \mu(\mathrm{d}{\mathfrak{v}})\leq C \left(  1+\|{\mathfrak{u}}\| + \|{\mathfrak{v}}\|  \right) + C\left(\int_{L^2(\mathbb{T})} \|{\mathfrak{v}}\|^2   \mu(\mathrm{d}{\mathfrak{v}})\right)^{\frac{1}{2}}\\
		&\leq C \left(  1+\|{\mathfrak{u}}\| + \|{\mathfrak{v}}\|  \right) + C\left(1+ \|u\|^2\right)^{\frac{1}{2}}\leq C \left(  1+\|{\mathfrak{u}}\| + \|{\mathfrak{v}}\|  \right).
	\end{align*}
\end{proof}

\noindent In the following proposition, we prove the ergodicity for the fast motion equation. 
\begin{prop}\label{propertiesofivpsolution}
	Let ${\mathfrak{u}},X,Y \in L^{2}(\mathbb{T})$, and let ${\mathfrak{v}}^{{\mathfrak{u}},X}$ be the solution of
	\begin{equation}\label{initial data problem}
		\begin{cases}
			d{\mathfrak{v}} = \big((1+i)(-\Delta)^{\rho} {\mathfrak{v}} + \mathcal{N}({\mathfrak{v}}) - \lambda {\mathfrak{v}} + G({\mathfrak{u}},{\mathfrak{v}})\big)\,dt
			+ \Sigma_{2}({\mathfrak{u}},{\mathfrak{v}})\,d\mathcal{W}_{2}, & \text{in } \mathbb{T}\times(0,\infty),\\[2mm]
			{\mathfrak{v}}(x,0)=X(x), & \text{in } \mathbb{T}.
		\end{cases}
	\end{equation}
	\begin{enumerate}
		\item[(i)] There exists a positive constant $C$ such that ${\mathfrak{v}}^{{\mathfrak{u}},X}$ satisfies
		\begin{align}\label{ia}
			\mathbb{E}\big[\|{\mathfrak{v}}^{{\mathfrak{u}},X}(t)\|^{2}\big]
			\le e^{-\lambda t}\|X\|^{2} + C\big(\|{\mathfrak{u}}\|^{2}+1\big),
		\end{align}
		and
		\begin{align}\label{ib}
			\mathbb{E}\big[\|{\mathfrak{v}}^{{\mathfrak{u}},X}(t)-{\mathfrak{v}}^{{\mathfrak{u}},Y}(t)\|^{2}\big]
			\le \|X-Y\|^{2}e^{-\lambda t},
			\qquad t\ge 0.
		\end{align}
		\item[(ii)] There exists a unique invariant measure $\mu^{{\mathfrak{u}}}$ for the Markov semigroup
		$\{P^{{\mathfrak{u}}}_{t}\}_{t\geq 0}$ associated with system \eqref{initial data problem} in $L^{2}(\mathbb{T})$.
		Moreover, we have
		\begin{align}\label{ii}
			\int_{L^{2}(\mathbb{T})} \|z\|^{2}\,\mu^{{\mathfrak{u}}}(dz)
			\le C\big(1+\|{\mathfrak{u}}\|^{2}\big).
		\end{align}
		\item[(iii)] There exists a positive constant $C$ such that ${\mathfrak{v}}^{{\mathfrak{u}},X}$ satisfies
		\begin{align}\label{iii}
			\big\|\mathbb{E}\big[F({\mathfrak{u}},{\mathfrak{v}}^{{\mathfrak{u}},X}(t))\big] - \bar F({\mathfrak{u}})\big\|^{2}
			\le C\big(1+\|X\|^{2}+\|{\mathfrak{u}}\|^{2}\big)e^{-\lambda t},
			\qquad t\ge 0.
		\end{align}
	\end{enumerate}
\end{prop}
\begin{proof}
		\textbf{(i)} An application of the infinite-dimensional It\^o formula to the process $\|{\mathfrak{v}}^{{\mathfrak{u}}, X}\|^{2}$ yields
	\begin{align*}
		d \|{\mathfrak{v}}^{{\mathfrak{u}}, X}\|^{2}
		&= 2 \left( (1+i) (-\Delta)^{\rho} {\mathfrak{v}}^{{\mathfrak{u}}, X} + \mathcal{N}({\mathfrak{v}}^{{\mathfrak{u}}, X}) - \lambda {\mathfrak{v}}^{{\mathfrak{u}}, X} + G({\mathfrak{u}}, {\mathfrak{v}}^{{\mathfrak{u}}, X}), {\mathfrak{v}}^{{\mathfrak{u}}, X} \right) dt \notag \\
		&\quad +  \|\Sigma_2({\mathfrak{u}}, {\mathfrak{v}}^{{\mathfrak{u}}, X})\|^2 dt  + 2 \left( {\mathfrak{v}}^{{\mathfrak{u}}, X}, \Sigma_2({\mathfrak{u}}, {\mathfrak{v}}^{{\mathfrak{u}}, X}) \right) d\mathcal{W}_2, \quad \mathbb{P}\text{-a.s.}
	\end{align*}
Integrating both sides over time and taking mathematical expectation, we obtain
	\begin{align*}
		\mathbb{E}\big[ \|{\mathfrak{v}}^{{\mathfrak{u}}, X}(t)\|^{2}\big] &= \|{\mathfrak{v}}_0\|^{2p} 
		+ {2} \mathbb{E} \bigg[\int_0^t \big( (1+i)(-\Delta)^{\rho}{\mathfrak{v}}^{{\mathfrak{u}}, X}(s) + \mathcal{N}({\mathfrak{v}}^{{\mathfrak{u}}, X}(s)) - \lambda {\mathfrak{v}}^{{\mathfrak{u}}, X}(s)\\
		& \quad\quad+ G(u^{\varepsilon,\nu}(s), {\mathfrak{v}}^{{\mathfrak{u}}, X}(s)), {\mathfrak{v}}^{{\mathfrak{u}}, X}(s) \big) ds \bigg] + \mathbb{E} \bigg[\int_0^t 
		\|\Sigma_2(u^{\varepsilon,\nu}(s), {\mathfrak{v}}^{{\mathfrak{u}}, X}(s))\|^2 ds\bigg] \\
		&\quad+ 2\mathbb{E} \bigg[\int_0^t   \left( {\mathfrak{v}}^{{\mathfrak{u}}, X}(s), \Sigma_2({\mathfrak{u}}(s), {\mathfrak{v}}^{{\mathfrak{u}}, X}(s)) \right) d\mathcal{W}_2(s)\bigg],
	\end{align*}
	for all $t\in [0,T]$. Since the martingale term is contributing zero,  differentiating the other terms with respect to time, we arrived at
	\begin{align*}
		\frac{d}{dt} \mathbb{E} \big[\|{\mathfrak{v}}^{{\mathfrak{u}}, X}(t)\|^{2} \big]
		&= 2 \mathbb{E} \big[\left( (1+i)(-\Delta)^{\rho}{\mathfrak{v}}^{{\mathfrak{u}}, X}(t) + \mathcal{N}({\mathfrak{v}}^{{\mathfrak{u}}, X}(t))  - \lambda {\mathfrak{v}}^{{\mathfrak{u}}, X}(t) + G(u^{\varepsilon,\nu}(t), {\mathfrak{v}}^{{\mathfrak{u}}, X}(t)), {\mathfrak{v}}^{{\mathfrak{u}}, X}(t) \right)\big]\\
		&\quad +  \mathbb{E}\big[\|\Sigma_2(u^{\varepsilon,\nu}(t), {\mathfrak{v}}^{{\mathfrak{u}}, X}(t))\|^2\big],
	\end{align*}
	for a.e. $t\in [0,T]$. Here, we apply integration by parts to compute
	\begin{align*}
		\left( (1+i)(-\Delta)^{\rho} {\mathfrak{v}}^{{\mathfrak{u}}, X}, {\mathfrak{v}}^{{\mathfrak{u}}, X} \right)=-\left( (1+i)(-\Delta)^{\frac{\rho}{2}}{\mathfrak{v}}^{{\mathfrak{u}}, X}, (-\Delta)^{\frac{\rho}{2}} {\mathfrak{v}}^{{\mathfrak{u}}, X} \right)=- \|(-\Delta)^{\frac{\rho}{2}} {\mathfrak{v}}^{{\mathfrak{u}}, X}\|^2,
	\end{align*}
	and 
	\begin{align*}
		\left(  \mathcal{N}({\mathfrak{v}}^{{\mathfrak{u}}, X}), {\mathfrak{v}}^{{\mathfrak{u}}, X} \right)= 	\left(  -(1+i \gamma) | {\mathfrak{v}}^{{\mathfrak{u}}, X}|^{\beta-1} {\mathfrak{v}}^{{\mathfrak{u}}, X}, {\mathfrak{v}}^{{\mathfrak{u}}, X} \right)= - \int_{\mathbb{T}} | {\mathfrak{v}}^{{\mathfrak{u}}, X}|^{\beta+1}dx= -\| {\mathfrak{v}}^{{\mathfrak{u}}, X}\|^{\beta+1}_{L^{\beta+1}}.
	\end{align*}
	Thus, we have
	\begin{align*}
		\frac{d}{dt} \mathbb{E} \big[\|{\mathfrak{v}}^{{\mathfrak{u}}, X}(t)\|^{2}\big] &=  2\mathbb{E}  \big[\big(- \lambda {\mathfrak{v}}^{{\mathfrak{u}}, X}(t) + G({\mathfrak{u}}, {\mathfrak{v}}^{{\mathfrak{u}}, X}(t)), {\mathfrak{v}}^{{\mathfrak{u}}, X}(t)\big)\big]+  \mathbb{E} \big[ \|\Sigma_2({\mathfrak{u}}, {\mathfrak{v}}^{{\mathfrak{u}}, X}(t))\|^2\big] \\
		&\quad -2 \mathbb{E} \bigg[\|(-\Delta)^{\frac{\rho}{2}}{\mathfrak{v}}^{\varepsilon,\nu}(t)\|^2+ \int_{\mathbb{T}}|{\mathfrak{v}}^{{\mathfrak{u}}, X}(t)|^{\beta+1} dx\bigg] \\
		&\leq -2\lambda \mathbb{E}\big[ \|{\mathfrak{v}}^{{\mathfrak{u}}, X}(t)\|^{2}\big]+2 \mathbb{E}  \big[\big( G({\mathfrak{u}}, {\mathfrak{v}}^{{\mathfrak{u}}, X}(t)), {\mathfrak{v}}^{{\mathfrak{u}}, X}(t)\big)\big] +\mathbb{E} \big[\|\Sigma_2({\mathfrak{u}}, {\mathfrak{v}}^{{\mathfrak{u}}, X}(t))\|^2\big],
	\end{align*}
	for a.e. $t\in [0,T]$.
	In view of the assumptions on $G$ in \Cref{main assumptions} and the Young inequality, we have
	\begin{align*}
		2( G({\mathfrak{u}}, {\mathfrak{v}}^{{\mathfrak{u}}, X}), {\mathfrak{v}}^{{\mathfrak{u}}, X} )&\leq 2\|{\mathfrak{v}}^{{\mathfrak{u}}, X}\|\|G({\mathfrak{u}}, {\mathfrak{v}}^{{\mathfrak{u}}, X})\|\leq 2L_G\|{\mathfrak{v}}^{{\mathfrak{u}}, X}\|(1+\|{\mathfrak{u}}\|+\|{\mathfrak{v}}^{{\mathfrak{u}}, X}\|) \\
		&= 2L_G\|{\mathfrak{v}}^{{\mathfrak{u}}, X}\|
		+2L_G\|{\mathfrak{v}}^{{\mathfrak{u}}, X}\|\|{\mathfrak{u}}\|
		+2L_G\|{\mathfrak{v}}^{{\mathfrak{u}}, X}\|^{2}\\
		&\leq 2L_G \eta\|{\mathfrak{v}}^{{\mathfrak{u}}, X}\|^{2} + C(\eta,L_G)+2L_G \eta\|{\mathfrak{v}}^{{\mathfrak{u}}, X}\|^{2} + C(\eta,L_G)\|{\mathfrak{u}}\|^{2}+2L_G\|{\mathfrak{v}}^{{\mathfrak{u}}, X}\|^{2}.
	\end{align*}
		Choosing $\eta=\frac{1}{4}$ and using the property of $\Sigma_2$ from \Cref{main assumptions}, we derive
	\begin{align*}
		\frac{d}{dt}\mathbb{E}\big[\|{\mathfrak{v}}^{{\mathfrak{u}}, X}(t)\|^{2}\big]
		&\leq -2\lambda\mathbb{E}\big[\|{\mathfrak{v}}^{{\mathfrak{u}}, X}(t)\|^{2}\big]
		+ 3L_G \mathbb{E}\big[\|{\mathfrak{v}}^{{\mathfrak{u}}, X}(t)\|^{2}\big] + C\mathbb{E}\big[\|{\mathfrak{u}}\|^{2}\big]
		+C \\
		&\leq -\lambda\mathbb{E}\big[\|{\mathfrak{v}}^{{\mathfrak{u}}, X}(t)\|^{2}\big]
		+ C\|{\mathfrak{u}}\|^{2}
		+ C.
	\end{align*}
	Using the variation of constants formula, we deduce
	\begin{align*}
		\mathbb{E}\big[\|{\mathfrak{v}}^{{\mathfrak{u}}, X}(t)\|^{2} \big]
		&\leq \|X\|^{2} e^{-\lambda t}
		+ C(\|{\mathfrak{u}}\|^{2}+1),
	\end{align*}
	for all $t\in [0,T]$. This proves \eqref{ia}.\\
	
	\noindent
	Let ${\mathfrak{v}}^{{\mathfrak{u}}, X}, {\mathfrak{v}}^{{\mathfrak{u}}, Y}$ be the solution corresponding to the initial data $X, Y$ respectively. Then ${\mathfrak{v}}^{{\mathfrak{u}}, X}-{\mathfrak{v}}^{{\mathfrak{u}}, Y}$ satisfies the following equation:
		\begin{equation*}
		\begin{cases}
			&d{({\mathfrak{v}}^{{\mathfrak{u}}, X}-{\mathfrak{v}}^{{\mathfrak{u}}, Y})}= \Big[(1+i)(-\Delta)^{\rho} {({\mathfrak{v}}^{{\mathfrak{u}}, X}-{\mathfrak{v}}^{{\mathfrak{u}}, Y})} + \mathcal{N}({\mathfrak{v}}^{{\mathfrak{u}}, X})-\mathcal{N}({\mathfrak{v}}^{{\mathfrak{u}}, Y}) - \lambda {({\mathfrak{v}}^{{\mathfrak{u}}, X}-{\mathfrak{v}}^{{\mathfrak{u}}, Y})}\\
		    &\qquad\qquad\qquad\qquad+G(\mathfrak{u},{\mathfrak{v}}^{{\mathfrak{u}}, X})-G(\mathfrak{u},{\mathfrak{v}}^{{\mathfrak{u}}, Y})\Big]\,dt+ \Big[\Sigma_{2}({\mathfrak{u}},{\mathfrak{v}}^{{\mathfrak{u}}, X})-\Sigma_{2}({\mathfrak{u}},{\mathfrak{v}}^{{\mathfrak{u}}, Y})\Big]\,d\mathcal{W}_{2} \,\,\,\,\text{in } \mathbb{T}\times(0,\infty),\\
			&{({\mathfrak{v}}^{{\mathfrak{u}}, X}-{\mathfrak{v}}^{{\mathfrak{u}}, Y})}(x,0)=X(x)-Y(x)  \,\,\,\,\,\,\,\,\,\qquad\qquad\qquad\qquad\quad\quad\,\,\,\,\quad\qquad\qquad\qquad\qquad\text{in } \mathbb{T}.
		\end{cases}
	\end{equation*}
Proceeding exactly as in the derivation of \eqref{ia}, we arrive at
	\begin{align*}
		\mathbb{E}\big[\|{\mathfrak{v}}^{{\mathfrak{u}},X}(t)-{\mathfrak{v}}^{{\mathfrak{u}},Y}(t)\|^{2}\big]
		\le \|X-Y\|^{2}e^{-\lambda t},
		\qquad t\ge 0.
	\end{align*}
	\textbf{(ii)} 
	For any $\mathfrak{u}\in L^2(\mathbb{T})$, \eqref{ia}, \eqref{ib} imply an existence of a unique invariant measure $\mu^{\mathfrak{u}}$ for the Markov semigroup $\{P_t^{\mathfrak{u}}\}_{t\geq 0}$ associated with the system \eqref{initial data problem} in $L^2(\mathbb{T})$ such that 
	\begin{align*}
		\int_{L^{2}(\mathbb{T})} P_t^{\mathfrak{u}} \phi\,d\mu^{{\mathfrak{u}}}=\int_{L^{2}(\mathbb{T})} \phi\,d\mu^{{\mathfrak{u}}} \quad t\geq 0,
	\end{align*}
	for any $\phi \in \mathcal{B}_b(L^{2}(\mathbb{T}))$, the space of all bounded Borel measurable functions on $L^{2}(\mathbb{T})$.\\
	\noindent Using \eqref{ia}, for any $t\geq 0$, we have 
	\begin{align*}
		\int_{L^{2}(\mathbb{T})}\|z\|^2 \mu^{{\mathfrak{u}}}(dz)=	\int_{L^{2}(\mathbb{T})}P_t^{\mathfrak{u}}\|z\|^2 \mu^{{\mathfrak{u}}}(dz)=&	\int_{L^{2}(\mathbb{T})}\mathbb{E}\big[\|{\mathfrak{v}}^{{\mathfrak{u}},z}(t)\|^2\big] \mu^{{\mathfrak{u}}}(dz)\\
		&\leq  C\big(\|{\mathfrak{u}}\|^{2}+1\big)+ e^{-\lambda t}\int_{L^{2}(\mathbb{T})}\|z\|^{2} \mu^{{\mathfrak{u}}}(dz).
	\end{align*}
	Therefore, for $\tilde{t}\geq 0$ such that $e^{-\lambda \tilde{t}}\leq\frac{1}{2}$, we have 
	\begin{align*}
		\int_{L^{2}(\mathbb{T})}\|z\|^2 \mu^{{\mathfrak{u}}}(dz)\leq  C\big(\|{\mathfrak{u}}\|^{2}+1\big).
	\end{align*}
	This proves \eqref{ii}.\\
	
	\noindent\textbf{(iii)} According to the invariant property of $\mu^{{\mathfrak{u}}}$, Lipschitz continuity of $F$ in \Cref{main assumptions}, \eqref{ia}, \eqref{ib} and \eqref{averagedequation}, we have 
	\begin{align*}
		\|\mathbb{E}\big[ F({\mathfrak{u}}, {\mathfrak{v}}^{{\mathfrak{u}},X})\big] - \bar{F}({\mathfrak{u}})\|^2
		&= \Big\|\mathbb{E} \big[F({\mathfrak{u}}, {\mathfrak{v}}^{{\mathfrak{u}},X})\big]
		- \int_{L^2(\mathbb{T})} F({\mathfrak{u}}, Y)\,\mu^{\mathfrak{u}}(dY)\Big\|^2 \\
		&= \Big\|\mathbb{E} \big[F({\mathfrak{u}}, {\mathfrak{v}}^{{\mathfrak{u}},X})\big]
		- \mathbb{E}\Big[ \int_{L^2(\mathbb{T})} F({\mathfrak{u}}, {\mathfrak{v}}^{{\mathfrak{u}},Y})\,\mu^{\mathfrak{u}}(dY)\Big]\Big\|^2 \\
		&= \Big\| \int_{L^2(\mathbb{T})}
		\mathbb{E}\big[F({\mathfrak{u}}, {\mathfrak{v}}^{{\mathfrak{u}},X}) - F({\mathfrak{u}}, {\mathfrak{v}}^{{\mathfrak{u}},Y})\big]
		\,\mu^{\mathfrak{u}}(dY)\Big\|^2 \\
		&\le L_F \int_{L^2(\mathbb{T})}
		\mathbb{E}\big[\|{\mathfrak{v}}^{{\mathfrak{u}},X} - {\mathfrak{v}}^{{\mathfrak{u}},Y}\|^2\big]
		\,\mu^{\mathfrak{u}}(dY) \\
		&\le C \int_{L^2(\mathbb{T})}
		\|X - Y\|^2 e^{-\lambda t}
		\,\mu^{\mathfrak{u}}(dY) \\
		&\le C \int_{L^2(\mathbb{T})}
		\big(\|X\|^2 + \|Y\|^2\big) e^{-\lambda t}
		\,\mu^{\mathfrak{u}}(dY) \\
		&\le C\big(1 + \|X\|^2 + \|{\mathfrak{u}}\|^2\big)e^{-\lambda t}.
	\end{align*}
	This completes the proof of \Cref{propertiesofivpsolution}.
\end{proof}
\begin{lem}\label{lemu1u2diffestimate}
	There exists a constant $C>0$ such that for any ${\mathfrak{u}_1}, {\mathfrak{u}_2}, X \in H$, we have 
	\begin{align}\label{u1u2diffestimate}
		\mathbb{E} \big[\|{\mathfrak{v}}^{{\mathfrak{u}_1},X}(t)-{\mathfrak{v}}^{{\mathfrak{u}_2},X}(t)\|^2 \big]\leq {C(L_G,L_{\Sigma_2},\lambda)} \|{\mathfrak{u}_1}-{\mathfrak{u}_2}\|^2.
	\end{align}
\end{lem}
\begin{proof}
	Let $z(\cdot)={\mathfrak{v}}^{{\mathfrak{u}_1},X}(\cdot)-{\mathfrak{v}}^{{\mathfrak{u}_2},X}(\cdot)$. Then $z(\cdot)$ satisfies 
	\begin{align*}
		\begin{cases}
			d{z} = \big((1+i)(-\Delta)^{\rho} {z} + \Big[\mathcal{N}({\mathfrak{v}}^{{\mathfrak{u}_1},X})- \mathcal{N}({\mathfrak{v}}^{{\mathfrak{u}_2},X})\Big]- \lambda z + \Big[G({\mathfrak{u}}_1,{\mathfrak{v}}^{{\mathfrak{u}_1},X})- G({\mathfrak{u}}_2,{\mathfrak{v}}^{{\mathfrak{u}_2},X})\Big]\big)\,dt\\
			\qquad\quad+\Big[\Sigma_{2}({\mathfrak{u}}_1,{\mathfrak{v}}^{{\mathfrak{u}_1},X})-\Sigma_{2}({\mathfrak{u}}_2,{\mathfrak{v}}^{{\mathfrak{u}_2},X})\Big]\,d\mathcal{W}_{2}, & \text{in } \mathbb{T}\times(0,\infty),\\[2mm]
			z(x,0)=0, & \text{in } \mathbb{T}.
		\end{cases}
	\end{align*}
Applying the infinite-dimensional infinite-dimensional It\^o formula to the process $\|z(\cdot)\|^2$, we get 
\begin{align*}
	 \frac{1}{2}d\|z\|^2 
	&=\Big\langle (1+i)(-\Delta)^{\rho} z + \mathcal{N}({\mathfrak{v}}^{{\mathfrak{u}_1},X}) - \mathcal{N}({\mathfrak{v}}^{{\mathfrak{u}_2},X}) 
	- \lambda z 
	+ G({\mathfrak{u}}_1, {\mathfrak{v}}^{{\mathfrak{u}_1},X}) - G({\mathfrak{u}}_2, {\mathfrak{v}}^{{\mathfrak{u}_2},X}),
	\, z \Big\rangle dt \\
	&\quad + \frac{1}{2} 
	\|\Sigma_2({\mathfrak{u}}_1, {\mathfrak{v}}^{{\mathfrak{u}_1},X}) 
	- \Sigma_2({\mathfrak{u}}_2, {\mathfrak{v}}^{{\mathfrak{u}_2},X})\|^2 dt  + 
	\Big\langle z, 
	\Sigma_2({\mathfrak{u}}_1, {\mathfrak{v}}^{{\mathfrak{u}_1},X}) 
	- \Sigma_2({\mathfrak{u}}_2, {\mathfrak{v}}^{{\mathfrak{u}_2},X}) \Big\rangle d\mathcal{W}_2, \quad \mathbb{P} \text{-a.s.},
\end{align*}
for a.e. $t\in [0,T]$.
Taking the integration over time and mathematical expectation on both sides of the above equation and using the fact that $ \int_{0}^{t}\langle z, 
\Sigma_2({\mathfrak{u}}_1, {\mathfrak{v}}^{{\mathfrak{u}_1},X}) 
- \Sigma_2({\mathfrak{u}}_2, {\mathfrak{v}}^{{\mathfrak{u}_2},X})\rangle d\mathcal{W}_2(s)$ is a martingale, it follows that
\begin{align*}
	\mathbb{E}\big[\|z(t)\|^2 \big]
	&= 2\mathbb{E} \bigg[
	\int_0^t 
	\Big\langle (1+i)(-\Delta)^{\rho} z(s)
	+ \mathcal{N}({\mathfrak{v}}^{{\mathfrak{u}_1},X}(s)) - \mathcal{N}({\mathfrak{v}}^{{\mathfrak{u}_2},X}(s)) 
	- \lambda z(s) \\
	&\quad \quad+ G({\mathfrak{u}}_1, {\mathfrak{v}}^{{\mathfrak{u}_1},X}(s)) 
	- G({\mathfrak{u}}_2, {\mathfrak{v}}^{{\mathfrak{u}_2},X}(s)),
	z(s) \Big\rangle ds\bigg]\\&\quad 
	+  \mathbb{E} \bigg[
	\int_0^t 
	\|\Sigma_2({\mathfrak{u}}_1, {\mathfrak{v}}^{{\mathfrak{u}_1},X}(s)) 
	- \Sigma_2({\mathfrak{u}}_2, {\mathfrak{v}}^{{\mathfrak{u}_2},X}(s))\|^2 ds\bigg],
\end{align*}
for all $t\in [0,T]$.
Differentiation with respect to time leads to
\begin{align*}
	\frac{d}{dt}\mathbb{E}\big[\|z(t)\|^2 \big]
	&= 2 \mathbb{E} \Big[
	\Big\langle (1+i)(-\Delta)^{\rho} z (t)
	+ \mathcal{N}({\mathfrak{v}}^{{\mathfrak{u}_1},X}(t)) - \mathcal{N}({\mathfrak{v}}^{{\mathfrak{u}_2},X}(t)) 
	- \lambda z(t) \\
	&\quad + G({\mathfrak{u}}_1, {\mathfrak{v}}^{{\mathfrak{u}_1},X}(t)) 
	- G({\mathfrak{u}}_2, {\mathfrak{v}}^{{\mathfrak{u}_2},X}(t)),
	z(t) \Big\rangle \Big]
	+  \mathbb{E} \big[
	\|\Sigma_2({\mathfrak{u}}_1, {\mathfrak{v}}^{{\mathfrak{u}_1},X}(t)) 
	- \Sigma_2({\mathfrak{u}}_2, {\mathfrak{v}}^{{\mathfrak{u}_2},X}(t))\|^2\big],
\end{align*}
for a.e. $t\in [0,T]$.
Here, we apply integration by parts to compute
\begin{align*}
	\Big\langle (1+i)(-\Delta)^{\rho} z , z  \Big\rangle=-\left( (1+i)(-\Delta)^{\frac{\rho}{2}}z , (-\Delta)^{\frac{\rho}{2}} z  \right)=- \|(-\Delta)^{\frac{\rho}{2}} z \|^2\leq 0.
\end{align*}
As a consequence of \Cref{diffnonrenega}, one can easily check
\begin{align*}
	(\mathcal{N}({\mathfrak{v}}^{{\mathfrak{u}_1},X}) - \mathcal{N}({\mathfrak{v}}^{{\mathfrak{u}_2},X}), z) \le 0.
\end{align*}
As a consequence of the \Cref{main assumptions} and the Young inequality, it follows that
\begin{align*}
	\frac{d}{dt} \mathbb{E}\big[\|z(t)\|^2 \big]
	&\le -2\lambda \mathbb{E}\big[\|z(t)\|^2\big] 
	+ 2 \mathbb{E}\big[ \langle G({\mathfrak{u}}_1, {\mathfrak{v}}^{{\mathfrak{u}_1},X}(t)) - G({\mathfrak{u}}_2, {\mathfrak{v}}^{{\mathfrak{u}_2},X}(t)), 
	z(t) \rangle\big]\\
	&\quad +  \mathbb{E}\big[\|\Sigma_2({\mathfrak{u}}_1, {\mathfrak{v}}^{{\mathfrak{u}_1},X}(t)) 
	- \Sigma_2({\mathfrak{u}}_2, {\mathfrak{v}}^{{\mathfrak{u}_2},X}(t))\|^2\big] \\
	&\le -2\lambda \mathbb{E}\big[\|z(t)\|^2 \big]
	+ 2L_G \mathbb{E}\big[\|z(t)\|(\|{\mathfrak{u}_1}-{\mathfrak{u}_2}\| + \|z(t)\|)\big] + L_{\Sigma_2}^2 \mathbb{E}\big[\|{\mathfrak{u}_1}-{\mathfrak{u}_2}\|^2\big]\\
	& \leq  -2\lambda \mathbb{E}\big[\|z(t)\|^2\big] +2 L_G \mathbb{E}\big[\|z(t)\|^2\big] +L_G \mathbb{E}\big[ \|{\mathfrak{u}_1}-{\mathfrak{u}_2}\|^2 + \|z(t)\|^2\big]+ L_{\Sigma_2}^2 \|{\mathfrak{u}_1}-{\mathfrak{u}_2}\|^2\\
	&\leq \big[-2\lambda+3L_G\big] \mathbb{E}\big[\|z(t)\|^2\big]+\Big[L_G+L_{\Sigma_2}^2\Big]\|{\mathfrak{u}_1}-{\mathfrak{u}_2}\|^2\\
	&\leq -\lambda\mathbb{E}\big[\|z(t)\|^2\big]+C(L_G,L_{\Sigma_2})\|{\mathfrak{u}_1}-{\mathfrak{u}_2}\|^2,
\end{align*}
for a.e. $t\in [0,T]$.
Hence, an application of the variation of constants formula yields
\begin{align*}
	\mathbb{E}\big[\|z(t)\|^2 \big]\le {C(L_G,L_{\Sigma_2})}\int_0^t e^{-\lambda(t-s)} \|({\mathfrak{u}_1}-{\mathfrak{u}_2})(s)\|^2 ds
	&\leq  {C(L_G,L_{\Sigma_2})}\|{\mathfrak{u}_1}-{\mathfrak{u}_2}\|^2\int_0^t e^{-\lambda(t-s)}ds\\
	&\leq  {C(L_G,L_{\Sigma_2},\lambda)}\|{\mathfrak{u}_1}-{\mathfrak{u}_2}\|^2,
\end{align*}
for all $t\in [0,T]$. This completes the proof.
\end{proof}
\begin{lem}[Lipschitz continuity of $\bar F$]\label{lipforfbar}
	For the averaged drift $\bar F$ defined in \eqref{fbar}, we have the existence of a constant $C>0$ such that
	\begin{align}\label{lipschitzoffbar}
		\|\bar F({\mathfrak{u}_1})-\bar F({\mathfrak{u}_2})\|\leq {C(L_G,L_{\Sigma_2},\lambda)} \|{\mathfrak{u}_1}-{\mathfrak{u}_2}\|.
	\end{align}
\end{lem}
\begin{proof}
	We have
	\begin{align*}
		\|\bar F({\mathfrak{u}_1})-\bar F({\mathfrak{u}_2})\|&\leq\Big\|\bar F({\mathfrak{u}_1})-  \mathbb{E}\big[F({\mathfrak{u}_1},{\mathfrak{v}}^{{\mathfrak{u}_1},X}(t))\big]\Big\|+\Big\|\mathbb{E}\big[F({\mathfrak{u}_2},{\mathfrak{v}}^{{\mathfrak{u}_2},X}(t))\big]-\bar F({\mathfrak{u}_2})\Big\|\\
		&\quad+\Big\|\mathbb{E}\big[F({\mathfrak{u}_1},{\mathfrak{v}}^{{\mathfrak{u}_1},X}(t))\big]-\mathbb{E}\big[F({\mathfrak{u}_2},{\mathfrak{v}}^{{\mathfrak{u}_2},X}(t))\big]\Big\|.
	\end{align*}
	Using \eqref{iii} and \eqref{u1u2diffestimate}, we have 
	\begin{align*}
		\|\bar F({\mathfrak{u}_1})-\bar F({\mathfrak{u}_2})\|&\leq C\big(1+\|X\|+\|{\mathfrak{u}_1}\|\big)e^{-\frac{\lambda t}{2}} +C\big(1+\|X\|+\|{\mathfrak{u}_2}\|\big)e^{-\frac{\lambda t}{2}} +{C(L_G,L_{\Sigma_2},\lambda)}\|{\mathfrak{u}_1}-{\mathfrak{u}_2}\|\\
		&\leq C\big(1+\|X\|+\|{\mathfrak{u}_1}\|+\|{\mathfrak{u}_2}\|\big)e^{-\frac{\lambda t}{2}} +{C(L_G,L_{\Sigma_2},\lambda)}\|{\mathfrak{u}_1}-{\mathfrak{u}_2}\|.
	\end{align*}
	Taking $t\to \infty$, we conclude
	\begin{align*}
		\|\bar F({\mathfrak{u}_1})-\bar F({\mathfrak{u}_2})\|&\leq {C(L_G,L_{\Sigma_2},\lambda)}\|{\mathfrak{u}_1}-{\mathfrak{u}_2}\|.
	\end{align*}
	This completes the proof.
\end{proof}
\section{Uniform estimates for the approximated solutions}\label{secuniformes}
In this section, we establish some uniform estimates for the approximated solution $({\mathfrak{u}}^{\varepsilon, \nu}, {\mathfrak{v}}^{\varepsilon, \nu})$ and some error estimates, which play a crucial role in the proof of the main theorem. In particular, this section is devoted to proving the \textbf{Step II} of \Cref{keystrategy}.
\begin{prop}\label{uniformestimates}
	Let $({\mathfrak{u}}_0, {\mathfrak{v}}_0) \in H^1(\mathbb{T}) \times H^1(\mathbb{T})$ and $\varepsilon, \nu \in (0,1)$. Suppose that $({\mathfrak{u}}^{\varepsilon, \nu}, {\mathfrak{v}}^{\varepsilon, \nu})$ is the unique solution to \eqref{viscousavg}, then for any $p > 0$, there exists a constant $C=C(p, T, {\mathfrak{u}}_0, {\mathfrak{v}}_0)>0$ independent of $\varepsilon, \nu$ such that
	\begin{align*}
		&\sup_{\varepsilon \in (0,1), \nu \in (0,1)} \sup_{0 \leq t \leq T} \mathbb{E} \big[\|{\mathfrak{u}}^{\varepsilon, \nu}(t)\|^{2p} \big]\leq C, \\
		&\sup_{\varepsilon \in (0,1), \nu \in (0,1)} \sup_{0 \leq t \leq T} \mathbb{E} \big[\|{\mathfrak{v}}^{\varepsilon, \nu}(t)\|^{2p}\big] \leq C, \\
		&\sup_{\varepsilon \in (0,1)} \mathbb{E} \Big[\sup_{0 \leq t \leq T} \|{\mathfrak{u}}^{\varepsilon, \nu}(t)\|_{H^1}^{2p}\Big] \leq C, \\
		&\sup_{\varepsilon \in (0,1), \nu \in (0,1)} \mathbb{E} \bigg[\bigg( \nu \int_0^T \|\Delta {\mathfrak{u}}^{\varepsilon, \nu}(t)\|^2 dt \bigg)^p\bigg] \leq C.
	\end{align*}
\end{prop}
\begin{proof}
	%
	An application of the infinite-dimensional It\^o formula to the process $\|{\mathfrak{v}}^{\varepsilon, \nu}\|^{2p}$ yields $\mathbb{P}$-a.s
	\begin{align}
		d \|{\mathfrak{v}}^{\varepsilon, \nu}\|^{2p}
		&= \frac{2p}{\varepsilon} \|{\mathfrak{v}}^{\varepsilon, \nu}\|^{2p-2} \left( (1+i) (-\Delta)^{\rho} {\mathfrak{v}}^{\varepsilon, \nu} + \mathcal{N}({\mathfrak{v}}^{\varepsilon, \nu}) - \lambda {\mathfrak{v}}^{\varepsilon, \nu} + G({\mathfrak{u}}^{\varepsilon, \nu}, {\mathfrak{v}}^{\varepsilon, \nu}), {\mathfrak{v}}^{\varepsilon, \nu} \right) dt \notag \\
		&\quad + \frac{p}{\varepsilon} \|{\mathfrak{v}}^{\varepsilon, \nu}\|^{2p-2} \|\Sigma_2({\mathfrak{u}}^{\varepsilon, \nu}, {\mathfrak{v}}^{\varepsilon, \nu})\|^2 dt + \frac{2p(p-1)}{\varepsilon} \|{\mathfrak{v}}^{\varepsilon, \nu}\|^{2p-4} ( {\mathfrak{v}}^{\varepsilon, \nu}, \Sigma_2({\mathfrak{u}}^{\varepsilon, \nu}, {\mathfrak{v}}^{\varepsilon, \nu}) )^2 dt \notag \\
		&\quad + \frac{2p}{\sqrt{\varepsilon}} \|{\mathfrak{v}}^{\varepsilon, \nu}\|^{2p-2} \left( {\mathfrak{v}}^{\varepsilon, \nu}, \Sigma_2({\mathfrak{u}}^{\varepsilon, \nu}, {\mathfrak{v}}^{\varepsilon, \nu}) \right) d\mathcal{W}_2,
	\end{align}
	for a.e. $t\in [0,T]$.
	Integrating both sides over time and taking mathematical expectation, we obtain
	\begin{align*}
		&\mathbb{E} \big[\|{\mathfrak{v}}^{\varepsilon, \nu}(t)\|^{2p}\big] \\
		&= \|{\mathfrak{v}}_0\|^{2p} 
		+ \frac{2p}{\varepsilon} \mathbb{E} \bigg[\int_0^t \|{\mathfrak{v}}^{\varepsilon, \nu}(s)\|^{2p-2} 
		\big( (1+i)(-\Delta)^{\rho}{\mathfrak{v}}^{\varepsilon, \nu}(s) + \mathcal{N}({\mathfrak{v}}^{\varepsilon, \nu}(s)) - \lambda {\mathfrak{v}}^{\varepsilon, \nu}(s) \\
		&\quad\quad + G(u^{\varepsilon,\nu}(s), {\mathfrak{v}}^{\varepsilon, \nu}(s)), {\mathfrak{v}}^{\varepsilon, \nu} (s)\big) ds \bigg]      + \frac{2p}{\sqrt{\varepsilon}}\mathbb{E} \bigg[\int_0^t  \|{\mathfrak{v}}^{\varepsilon, \nu}\|^{2p-2} \left( {\mathfrak{v}}^{\varepsilon, \nu}, \Sigma_2({\mathfrak{u}}^{\varepsilon, \nu}, {\mathfrak{v}}^{\varepsilon, \nu}) \right) d\mathcal{W}_2 \bigg]\\
		&\quad + \frac{p}{\varepsilon} \mathbb{E}\bigg[ \int_0^t \|{\mathfrak{v}}^{\varepsilon, \nu}(s)\|^{2p-2} 
		\|\Sigma_2(u^{\varepsilon,\nu}, {\mathfrak{v}}^{\varepsilon, \nu})\|^2 ds\bigg]+ \frac{2p(p-1)}{\varepsilon} \mathbb{E}\bigg[ \int_0^t \|{\mathfrak{v}}^{\varepsilon, \nu}(s)\|^{2p-4} 
		( {\mathfrak{v}}^{\varepsilon, \nu}, \Sigma_2(u^{\varepsilon,\nu}, {\mathfrak{v}}^{\varepsilon, \nu}) )^2 ds\bigg],
	\end{align*}
	for all $t\in [0,T]$. Now, differentiating with respect to time, we arrived at
	\begin{align*}
		&\frac{d}{dt} \mathbb{E} \big[\|{\mathfrak{v}}^{\varepsilon, \nu}(t)\|^{2p} \big]\\
		&= \frac{2p}{\varepsilon} \mathbb{E}\big[ \|{\mathfrak{v}}^{\varepsilon, \nu}(t)\|^{2p-2} 
		\left( (1+i)(-\Delta)^{\rho}{\mathfrak{v}}^{\varepsilon, \nu}(t) + \mathcal{N}({\mathfrak{v}}^{\varepsilon, \nu}(t))  - \lambda {\mathfrak{v}}^{\varepsilon, \nu}(t) + G(u^{\varepsilon,\nu}(t), {\mathfrak{v}}^{\varepsilon, \nu}(t)), {\mathfrak{v}}^{\varepsilon, \nu} (t)\right) \big]\\
		&\quad + \frac{p}{\varepsilon} \mathbb{E}\big[ \|{\mathfrak{v}}^{\varepsilon, \nu}(t)\|^{2p-2} \|\Sigma_2(u^{\varepsilon,\nu}(t), {\mathfrak{v}}^{\varepsilon, \nu}(t))\|^2\big] + \frac{2p(p-1)}{\varepsilon} \mathbb{E}\big[ \|{\mathfrak{v}}^{\varepsilon, \nu}(t)\|^{2p-4}
		( {\mathfrak{v}}^{\varepsilon, \nu}(t), \Sigma_2(u^{\varepsilon,\nu}(t), {\mathfrak{v}}^{\varepsilon, \nu}(t)) )^2\big],
	\end{align*}
	for a.e. $t\in [0,T]$. Here, we apply integration by parts to compute,
	\begin{align*}
		\left( (1+i)(-\Delta)^{\rho} {\mathfrak{v}}^{\varepsilon, \nu}, {\mathfrak{v}}^{\varepsilon, \nu} \right)=-\left( (1+i)(-\Delta)^{\frac{\rho}{2}}{\mathfrak{v}}^{\varepsilon, \nu}, (-\Delta)^{\frac{\rho}{2}} {\mathfrak{v}}^{\varepsilon, \nu} \right)=- \|(-\Delta)^{\frac{\rho}{2}} {\mathfrak{v}}^{\varepsilon, \nu}\|^2,
	\end{align*}
	and 
	\begin{align*}
		\left(  \mathcal{N}({\mathfrak{v}}^{\varepsilon, \nu}), {\mathfrak{v}}^{\varepsilon, \nu} \right)= 	\left(  -(1+i \gamma) | {\mathfrak{v}}^{\varepsilon, \nu}|^{\beta-1} {\mathfrak{v}}^{\varepsilon, \nu}, {\mathfrak{v}}^{\varepsilon, \nu} \right)= - \int_{\mathbb{T}} | {\mathfrak{v}}^{\varepsilon, \nu}|^{\beta+1}dx= -\| {\mathfrak{v}}^{\varepsilon, \nu}\|^{\beta+1}_{L^{\beta+1}}.
	\end{align*}
	Thus, we have
	\begin{align*}
	&	\frac{d}{dt} \mathbb{E} \big[\|{\mathfrak{v}}^{\varepsilon, \nu}(t)\|^{2p}\big] \\
	&=  \frac{2p}{\varepsilon} \mathbb{E} \big[\|{\mathfrak{v}}^{\varepsilon, \nu}(t)\|^{2p-2} \big(- \lambda {\mathfrak{v}}^{\varepsilon, \nu}(t) + G({\mathfrak{u}}^{\varepsilon, \nu}(t), {\mathfrak{v}}^{\varepsilon, \nu}(t)), {\mathfrak{v}}^{\varepsilon, \nu}(t)\big)\big] + \frac{p}{\varepsilon} \mathbb{E} \big[\|{\mathfrak{v}}^{\varepsilon, \nu}(t)\|^{2p-2} \|\Sigma_2({\mathfrak{u}}^{\varepsilon, \nu}(t), {\mathfrak{v}}^{\varepsilon, \nu}(t))\|^2 \big]\\
		&\quad - \frac{2p}{\varepsilon} \mathbb{E} \bigg[\|{\mathfrak{v}}^{\varepsilon, \nu}(t)\|^{2p-2}\Big(\|(-\Delta)^{\frac{\rho}{2}}{\mathfrak{v}}^{\varepsilon,\nu}(t)\|^2+ \int_{\mathbb{T}}|{\mathfrak{v}}^{\varepsilon, \nu}(t)|^{\beta+1} dx\Big)\bigg] \\
		&\quad+ \frac{2p(p-1)}{\varepsilon} \mathbb{E} \big[\|{\mathfrak{v}}^{\varepsilon, \nu}(t)\|^{2p-4} 
		( {\mathfrak{v}}^{\varepsilon, \nu}(t), \Sigma_2({\mathfrak{u}}^{\varepsilon, \nu}(t), {\mathfrak{v}}^{\varepsilon, \nu}(t)) )^2\big]\\
		&\leq -\frac{2p\lambda}{\varepsilon} \mathbb{E} \big[\|{\mathfrak{v}}^{\varepsilon, \nu}(t)\|^{2p}\big]+\frac{2p}{\varepsilon} \mathbb{E} \big[\|{\mathfrak{v}}^{\varepsilon, \nu}(t)\|^{2p-2} \big( G({\mathfrak{u}}^{\varepsilon, \nu}(t), {\mathfrak{v}}^{\varepsilon, \nu}(t)), {\mathfrak{v}}^{\varepsilon, \nu}(t)\big)\big] \\
		&\quad + \frac{p}{\varepsilon} \mathbb{E}\big[ \|{\mathfrak{v}}^{\varepsilon, \nu}(t)\|^{2p-2} \|\Sigma_2({\mathfrak{u}}^{\varepsilon, \nu}(t), {\mathfrak{v}}^{\varepsilon, \nu}(t))\|^2 \big]+ \frac{2p(p-1)}{\varepsilon} \mathbb{E}\big[ \|{\mathfrak{v}}^{\varepsilon, \nu}(t)\|^{2p-4} 
		( {\mathfrak{v}}^{\varepsilon, \nu}(t), \Sigma_2({\mathfrak{u}}^{\varepsilon, \nu}(t), {\mathfrak{v}}^{\varepsilon, \nu}(t)) )^2\big],
	\end{align*}
	for a.e. $t\in [0,T]$. Using the property of $G$ in \Cref{main assumptions}, we have
	\begin{align*}
		&2\|{\mathfrak{v}}^{\varepsilon, \nu}(t)\|^{2p-2}( G({\mathfrak{u}}^{\varepsilon, \nu}(t), {\mathfrak{v}}^{\varepsilon, \nu}(t)), {\mathfrak{v}}^{\varepsilon, \nu} (t)) \leq 2\|{\mathfrak{v}}^{\varepsilon, \nu}(t)\|^{2p-1}\|G({\mathfrak{u}}^{\varepsilon, \nu}(t), {\mathfrak{v}}^{\varepsilon, \nu}(t))\| \\
		&\leq 2L_G\|{\mathfrak{v}}^{\varepsilon, \nu}(t)\|^{2p-1}(1+\|{\mathfrak{u}}^{\varepsilon, \nu}(t)\|+\|{\mathfrak{v}}^{\varepsilon, \nu}(t)\|) \\
		&= 2L_G\|{\mathfrak{v}}^{\varepsilon, \nu}(t)\|^{2p-1}
		+2L_G\|{\mathfrak{v}}^{\varepsilon, \nu}(t)\|^{2p-1}\|{\mathfrak{u}}^{\varepsilon, \nu}(t)\|
		+2L_G\|{\mathfrak{v}}^{\varepsilon, \nu}(t)\|^{2p}.
	\end{align*}
	It follows from the Young inequality in \Cref{Youngs inequality} with $p'=\frac{2p}{2p-1}$ and $q'=2p$ that
	\begin{align*}
		2L_G\|{\mathfrak{v}}^{\varepsilon, \nu}\|^{2p-1}
		&\leq 2L_G \eta\|{\mathfrak{v}}^{\varepsilon, \nu}\|^{2p} + C(\eta,L_G,p).
	\end{align*}
	For particular for $\eta=\frac{1}{6}$, we get
	\begin{align*}
		2L_G\|{\mathfrak{v}}^{\varepsilon, \nu}\|^{2p-1}
		&\leq \tfrac{L_G}{3}\|{\mathfrak{v}}^{\varepsilon, \nu}\|^{2p} + C(L_G,p).
	\end{align*}
	Similarly, we have
	\begin{align*}
		2L_G\|{\mathfrak{v}}^{\varepsilon, \nu}\|^{2p-1}\|{\mathfrak{u}}^{\varepsilon, \nu}\|
		&\leq \tfrac{L_G}{3}\|{\mathfrak{v}}^{\varepsilon, \nu}\|^{2p} + C(L_G,p)\|{\mathfrak{u}}^{\varepsilon, \nu}\|^{2p}.
	\end{align*}
    Using the property of $\Sigma_2$ from \Cref{main assumptions} and the Young inequality with $p'=\frac{2p}{2p-2}$ and $q'=p$ in \Cref{Youngs inequality}, we derive
	\begin{align*}
		&\|{\mathfrak{v}}^{\varepsilon, \nu}\|^{2p-2}\|\Sigma_2({\mathfrak{u}}^{\varepsilon, \nu}, {\mathfrak{v}}^{\varepsilon, \nu})\|^2 
		+ 2(p-1)\|{\mathfrak{v}}^{\varepsilon, \nu}\|^{2p-4}( {\mathfrak{v}}^{\varepsilon, \nu}, \Sigma_2({\mathfrak{u}}^{\varepsilon, \nu}, {\mathfrak{v}}^{\varepsilon, \nu}) )^2 \\
		&\leq \tfrac{L_G}{6}\|{\mathfrak{v}}^{\varepsilon, \nu}\|^{2p} + C_1(M,p)+\tfrac{L_G}{6}\|{\mathfrak{v}}^{\varepsilon, \nu}\|^{2p} + C_2(M,p)\\
		&\leq \tfrac{L_G}{3}\|{\mathfrak{v}}^{\varepsilon, \nu}\|^{2p} + C(M,p).
	\end{align*}
	Therefore, we have
	\begin{align*}
		\frac{d}{dt}\mathbb{E}\big[\|{\mathfrak{v}}^{\varepsilon, \nu}(t)\|^{2p}\big]
		&\leq -\frac{2p\lambda}{\varepsilon}\mathbb{E}\big[\|{\mathfrak{v}}^{\varepsilon, \nu}(t)\|^{2p}\big]
		+ \frac{3L_G p}{\varepsilon}\mathbb{E}\big[\|{\mathfrak{v}}^{\varepsilon, \nu}(t)\|^{2p}\big]  + \frac{C}{\varepsilon}\mathbb{E}\big[\|{\mathfrak{u}}^{\varepsilon, \nu}(t)\|^{2p}\big]
		+ \frac{C}{\varepsilon}, \\
		&\leq -\frac{p\lambda}{\varepsilon}\mathbb{E}\big[\|{\mathfrak{v}}^{\varepsilon, \nu}(t)\|^{2p}\big]
		+ \frac{C}{\varepsilon}\mathbb{E}\big[\|{\mathfrak{u}}^{\varepsilon, \nu}(t)\|^{2p}\big]
		+ \frac{C}{\varepsilon},
	\end{align*}
	for a.e. $t\in [0,T]$. Using the variation of constants formula, we deduce
	\begin{align}\label{comparisonv}
		\mathbb{E}\big[\|{\mathfrak{v}}^{\varepsilon, \nu}(t)\|^{2p}\big] 
		&\leq \|{\mathfrak{v}}_0\|^{2p} e^{-\tfrac{p\lambda}{\varepsilon}t}
		+ \frac{C}{\varepsilon}\int_0^t e^{-\tfrac{p\lambda}{\varepsilon}(t-s)}
		\big(\mathbb{E}\big[\|{\mathfrak{u}}^{\varepsilon, \nu}(s)\|^{2p}\big]+1\big)\,ds,
	\end{align}
	for all $t\in [0,T]$. We employ the infinite-dimensional It\^o formula to the process $\|{\mathfrak{u}}^{\varepsilon, \nu}\|^{2p}$ to obtain $\mathbb{P}$-a.s
	\begin{align}\label{itou}
		d\|{\mathfrak{u}}^{\varepsilon, \nu}(t)\|^{2p} 
		&= 2p\|{\mathfrak{u}}^{\varepsilon, \nu}(t)\|^{2p-2}\big(A_{\nu}{\mathfrak{u}}^{\varepsilon, \nu}(t) + \mathcal{N}({\mathfrak{u}}^{\varepsilon, \nu}(t)) + F({\mathfrak{u}}^{\varepsilon, \nu}(t), {\mathfrak{v}}^{\varepsilon, \nu}(t)), {\mathfrak{u}}^{\varepsilon, \nu}(t)\big)dt \nonumber\\
		&\quad + 2p\|{\mathfrak{u}}^{\varepsilon, \nu}(t)\|^{p-2}\|\Sigma_1({\mathfrak{u}}^{\varepsilon, \nu}(t))\|^2 dt  + 2p(p-1)\|{\mathfrak{u}}^{\varepsilon, \nu}(t)\|^{2p-4}( {\mathfrak{u}}^{\varepsilon, \nu}(t), \Sigma_1({\mathfrak{u}}^{\varepsilon, \nu}(t)))^2 dt \nonumber\\
		&\quad + 2p\|{\mathfrak{u}}^{\varepsilon, \nu}(t)\|^{2p-2}( {\mathfrak{u}}^{\varepsilon, \nu}(t), \Sigma_1({\mathfrak{u}}^{\varepsilon, \nu}(t))) d\mathcal{W}_1(t),
	\end{align}
	for a.e. $t\in [0,T]$. Integrating over time and taking mathematical expectation on the both sides of the above equation, we derive
	\begin{align*}
		\mathbb{E}\big[\|{\mathfrak{u}}^{\varepsilon, \nu}(t)\|^{2p }\big]
		&= \|{\mathfrak{u}}_0\|^{2p} + \mathbb{E}\bigg[\int_0^t \Big[2p\|{\mathfrak{u}}^{\varepsilon, \nu}(s)\|^{2p-2}\big(A_{\nu}{\mathfrak{u}}^{\varepsilon, \nu}(s) + \mathcal{N}({\mathfrak{u}}^{\varepsilon, \nu}(s)) + F({\mathfrak{u}}^{\varepsilon, \nu}(s), {\mathfrak{v}}^{\varepsilon, \nu}(s)), {\mathfrak{u}}^{\varepsilon, \nu}(s)\big)\\
		&\qquad + p\|{\mathfrak{u}}^{\varepsilon, \nu}(s)\|^{2p-2}\|\Sigma_1({\mathfrak{u}}^{\varepsilon, \nu}(s))\|^2 + 2p(p-1)\|{\mathfrak{u}}^{\varepsilon, \nu}(s)\|^{2p-4}( {\mathfrak{u}}^{\varepsilon, \nu}(s), \Sigma_1({\mathfrak{u}}^{\varepsilon, \nu}(s)))^2 \Big] ds\bigg]\\
		&\quad +2p\mathbb{E} \bigg[\int_0^t \|{\mathfrak{u}}^{\varepsilon, \nu}(s)\|^{2p-2}( {\mathfrak{u}}^{\varepsilon, \nu}(s), \Sigma_1({\mathfrak{u}}^{\varepsilon, \nu}(s))) d\mathcal{W}_1(s)\bigg],
	\end{align*}
	for all $t\in [0,T]$. Differentiating both sides of the equation with respect to time yields
	\begin{align*}
		\frac{d}{dt}\mathbb{E}\big[\|{\mathfrak{u}}^{\varepsilon, \nu}(t)\|^{2p} \big]
		&= 2p\mathbb{E}\big[\|{\mathfrak{u}}^{\varepsilon, \nu}(t)\|^{2p-2}\big(A_{\nu}{\mathfrak{u}}^{\varepsilon, \nu}(t) + \mathcal{N}({\mathfrak{u}}^{\varepsilon, \nu}(t)) + F({\mathfrak{u}}^{\varepsilon, \nu}(t), {\mathfrak{v}}^{\varepsilon, \nu}(t)), {\mathfrak{u}}^{\varepsilon, \nu}(t)\big)\big]\\
		&\quad + p\mathbb{E}\big[\|{\mathfrak{u}}^{\varepsilon, \nu}(t)\|^{2p-2}\|\Sigma_1({\mathfrak{u}}^{\varepsilon, \nu}(t))\|^2\big] + 2p(p-1)\mathbb{E}\big[\|{\mathfrak{u}}^{\varepsilon, \nu}(t)\|^{2p-4}( {\mathfrak{u}}^{\varepsilon, \nu}(t), \Sigma_1({\mathfrak{u}}^{\varepsilon, \nu}(t)))^2\big],
	\end{align*}
	for a.e. $t\in [0,T]$. Integration by parts gives
	\begin{align*}
		\big(A_{\nu}{\mathfrak{u}}^{\varepsilon, \nu}, {\mathfrak{u}}^{\varepsilon, \nu}\big)=\big(-i(-\Delta)^{\alpha}{\mathfrak{u}}^{\varepsilon, \nu}+\nu \Delta {\mathfrak{u}}^{\varepsilon, \nu}, {\mathfrak{u}}^{\varepsilon, \nu}\big)&=\big(i(-\Delta)^{\frac{\alpha}{2}}{\mathfrak{u}}^{\varepsilon, \nu}, (-\Delta)^{\frac{\alpha}{2}}{\mathfrak{u}}^{\varepsilon, \nu}\big)-\big(\nu \grad {\mathfrak{u}}^{\varepsilon, \nu}, \grad {\mathfrak{u}}^{\varepsilon, \nu}\big)\\
		&= -\nu \|{\mathfrak{u}}_x^{\varepsilon, \nu}\|^2.
	\end{align*}
	We also calculate
	\begin{align*}
		\left(  \mathcal{N}({\mathfrak{u}}^{\varepsilon, \nu}), {\mathfrak{u}}^{\varepsilon, \nu} \right)= 	\left(  -(1+i \gamma) | {\mathfrak{u}}^{\varepsilon, \nu}|^{\beta-1} {\mathfrak{u}}^{\varepsilon, \nu}, {\mathfrak{u}}^{\varepsilon, \nu} \right)= - \int_{\mathbb{T}} | {\mathfrak{u}}^{\varepsilon, \nu}|^{\beta+1}dx= -\| {\mathfrak{u}}^{\varepsilon, \nu}\|^{\beta+1}_{L^{\beta+1}},
	\end{align*}
	and
	\begin{align*}
		( F({\mathfrak{u}}^{\varepsilon, \nu}, {\mathfrak{v}}^{\varepsilon, \nu}), {\mathfrak{u}}^{\varepsilon, \nu})&\leq \|F({\mathfrak{u}}^{\varepsilon, \nu}, {\mathfrak{v}}^{\varepsilon, \nu})\|\|{\mathfrak{u}}^{\varepsilon, \nu}\|\\
		&\leq L_F (1+\|{\mathfrak{u}}^{\varepsilon, \nu}\|+\|{\mathfrak{v}}^{\varepsilon, \nu}\|) \|{\mathfrak{u}}^{\varepsilon, \nu}\|\\
		&\leq C(1+\|{\mathfrak{u}}^{\varepsilon, \nu}\|^2+\|{\mathfrak{v}}^{\varepsilon, \nu}\|^2).
	\end{align*}
	\noindent
	Consequently, we obtain
	\begin{align*}
		&\frac{d}{dt}\mathbb{E}\big[\|{\mathfrak{u}}^{\varepsilon, \nu}(t)\|^{2p}\big]\\
		&\leq 2p\mathbb{E}\big[\|{\mathfrak{u}}^{\varepsilon, \nu}(t)\|^{2p-2}\big[-\nu \|{\mathfrak{u}}_x^{\varepsilon, \nu}(t)\|^2 -\| {\mathfrak{u}}^{\varepsilon, \nu}(t)\|^{\beta+1}_{L^{\beta+1}} + C(1+\|{\mathfrak{u}}^{\varepsilon, \nu}(t)\|^2+\|{\mathfrak{v}}^{\varepsilon, \nu}(t)\|^2)\big]\big]\\
		&\quad + p\mathbb{E}\big[\|{\mathfrak{u}}^{\varepsilon, \nu}(t)\|^{2p-2}\|\Sigma_1({\mathfrak{u}}^{\varepsilon, \nu}(t))\|^2\big]+ 2p(p-1)\mathbb{E}\big[\|{\mathfrak{u}}^{\varepsilon, \nu}(t)\|^{2p-4}\|{\mathfrak{u}}^{\varepsilon, \nu}(t)\|^2 \|\Sigma_1({\mathfrak{u}}^{\varepsilon, \nu}(t))\|^2\big]\\
		&\le 2p\mathbb{E}\big[\|{\mathfrak{u}}^{\varepsilon, \nu}(t)\|^{2p-2}\big[( C(1+\|{\mathfrak{u}}^{\varepsilon, \nu}(t)\|^2+\|{\mathfrak{v}}^{\varepsilon, \nu}(t)\|^2)\big)\big] + C\mathbb{E}\big[\|{\mathfrak{u}}^{\varepsilon, \nu}(t)\|^{2p-2}\|\Sigma_1({\mathfrak{u}}^{\varepsilon, \nu}(t))\|^2\big].
	\end{align*}
	for a.e. $t\in [0,T]$.
	Employing the Young inequality with the condition $\frac{2p-2}{2p}+\frac{1}{p}=1$,
	\begin{align*}
		&\|{\mathfrak{u}}^{\varepsilon, \nu}\|^{2p-2}\|\Sigma_1({\mathfrak{u}}^{\varepsilon, \nu})\|^2\leq \|{\mathfrak{u}}^{\varepsilon, \nu}\|^{2p-2} L_{\Sigma_1}(1+\|{\mathfrak{u}}^{\varepsilon, \nu}\|)^{2} \leq C  \|{\mathfrak{u}}^{\varepsilon, \nu}\|^{2p-2}(1+\|{\mathfrak{u}}^{\varepsilon, \nu}\|^{2})\\
		&\leq C  \|{\mathfrak{u}}^{\varepsilon, \nu}\|^{2p} + C\big(\frac{2p-2}{2p}\|{\mathfrak{u}}^{\varepsilon, \nu}\|^{2p}+\frac{1}{p} \big) \leq C(p) \big(1+\|{\mathfrak{u}}^{\varepsilon, \nu}\|^{2p}\big),
	\end{align*}
	and 
	\begin{align*}
		\|{\mathfrak{u}}^{\varepsilon, \nu}\|^{2p-2}\big[ C(1+\|{\mathfrak{u}}^{\varepsilon, \nu}\|^2+\|{\mathfrak{v}}^{\varepsilon, \nu}\|^2)\big] \leq  C(p) \big(1+\|{\mathfrak{u}}^{\varepsilon, \nu}\|^{2p}+\|{\mathfrak{v}}^{\varepsilon, \nu}\|^{2p}\big).
	\end{align*}
	Therefore, we conclude
	\begin{align*}
		\frac{d}{dt}\mathbb{E}\big[\|{\mathfrak{u}}^{\varepsilon, \nu}(t)\|^{2p}\big]
		\le C(p)\big(\mathbb{E}\big[\|{\mathfrak{u}}^{\varepsilon, \nu}(t)\|^{2p}\big] + \mathbb{E}\big[\|{\mathfrak{u}}^{\varepsilon, \nu}(t)\|^{2p}\big] + 1\big),
	\end{align*}
	for a.e. $t\in [0,T]$. Hence, by the variation of constants formula,
	\begin{align}\label{comparisonu}
		\mathbb{E}\big[\|{\mathfrak{u}}^{\varepsilon, \nu}(t)\|^{2p}\leq e^{Ct}\|{\mathfrak{u}}_0\|^{2p} \big]+ C\int_0^t e^{C(t-\tau)}\big(\mathbb{E}\big[\|{\mathfrak{v}}^{\varepsilon, \nu}(\tau)\|^{2p}\big] + 1\big)d\tau,
	\end{align}
	for all $t\in [0,T]$. Using \eqref{comparisonu} in \eqref{comparisonv} and applying the integration by parts, we get
	\begin{align*}
	&	\mathbb{E}\big[\|{\mathfrak{v}}^{\varepsilon, \nu}(t)\|^{2p}\big] \\
		&\leq \|{\mathfrak{v}}_0\|^{2p} e^{-\tfrac{p\lambda}{\varepsilon}t}
		+ \frac{C}{\varepsilon}\int_0^t e^{-\tfrac{p\lambda}{\varepsilon}(t-s)}
		\left(e^{Ct}\|{\mathfrak{u}}_0\|^{2p} + C\int_0^t e^{C(t-\tau)}\big(\mathbb{E}\big[\|{\mathfrak{v}}^{\varepsilon, \nu}(\tau)\|^{2p}\big] + 1\big)d\tau+1\right)\,ds\\
		&\leq C(p,T,\|{\mathfrak{v}}_0\|,\|{\mathfrak{u}}_0\|)+\frac{C}{\varepsilon}\int_0^t e^{-\tfrac{p\lambda}{\varepsilon}(t-s)}
		\left(\int_0^t e^{C(t-\tau)}\mathbb{E}\big[\|{\mathfrak{v}}^{\varepsilon, \nu}(\tau)\|^{2p}\big]d\tau\right)\,ds\\
		&\leq C(p,T,\|{\mathfrak{v}}_0\|,\|{\mathfrak{u}}_0\|)+\frac{C}{\lambda p}\int_0^t\mathbb{E}\big[\|{\mathfrak{v}}^{\varepsilon, \nu}(\tau)\|^{2p}\big]d\tau - \frac{C}{\lambda p}\int_0^t e^{-\tfrac{p\lambda}{\varepsilon}(t-s)}\mathbb{E}\big[\|{\mathfrak{u}}^{\varepsilon, \nu}(s)\|^{2p}\big]\,ds\\
		&\leq C(p,T,\|{\mathfrak{v}}_0\|,\|{\mathfrak{u}}_0\|)+\frac{C}{\lambda p}\int_0^t \left(1-e^{-\tfrac{p\lambda}{\varepsilon}(t-s)}\right)\mathbb{E}\big[\|{\mathfrak{u}}^{\varepsilon, \nu}(s)\|^{2p}\big]\,ds\\
		&\leq C(p,T,\|{\mathfrak{v}}_0\|,\|{\mathfrak{u}}_0\|)+C(p,T,\lambda)\int_0^t \mathbb{E}\big[\|{\mathfrak{u}}^{\varepsilon, \nu}(s)\|^{2p}\big]\,ds,
	\end{align*}
	for all $t\in [0,T]$. Employing the Gr\"onwall inequality, we conclude
	\begin{equation}\label{v1bound}
		\sup_{0\le t\le T} \mathbb{E}\big[\|{\mathfrak{v}}^{\varepsilon, \nu}(t)\|^{2p}\big] \le C(p,T,{\mathfrak{u}}_0,{\mathfrak{v}}_0).
	\end{equation}
	Consequently, we have
	\begin{equation}\label{u1bound}
		\sup_{0\le t\le T} \mathbb{E}\big[\|{\mathfrak{u}}^{\varepsilon, \nu}(t)\|^{2p}\big] \le C(p,T,{\mathfrak{u}}_0,{\mathfrak{v}}_0).
	\end{equation}
	From \eqref{itou}, it follows that $\mathbb{P}$-a.s,
	\begin{align*}
		\|{\mathfrak{u}}^{\varepsilon, \nu}(t)\|^{2p }
		&= \|{\mathfrak{u}}_0\|^{2p} + \int_0^t \Big[2p\|{\mathfrak{u}}^{\varepsilon, \nu}(s)\|^{2p-2}\big(A_{\nu}{\mathfrak{u}}^{\varepsilon, \nu}(s) + \mathcal{N}({\mathfrak{u}}^{\varepsilon, \nu}(s)) + F({\mathfrak{u}}^{\varepsilon, \nu}(s), {\mathfrak{v}}^{\varepsilon, \nu}(s)), {\mathfrak{u}}^{\varepsilon, \nu}(s)\big)\\
		&\qquad + p\|{\mathfrak{u}}^{\varepsilon, \nu}(s)\|^{2p-2}\|\Sigma_1({\mathfrak{u}}^{\varepsilon, \nu}(s))\|^2 + 2p(p-1)\|{\mathfrak{u}}^{\varepsilon, \nu}(s)\|^{2p-4}( {\mathfrak{u}}^{\varepsilon, \nu}(s), \Sigma_1({\mathfrak{u}}^{\varepsilon, \nu}(s)))^2 \Big] ds\\
		&\quad +2p \int_0^t \|{\mathfrak{u}}^{\varepsilon, \nu}(s)\|^{2p-2}( {\mathfrak{u}}^{\varepsilon, \nu}(s), \Sigma_1({\mathfrak{u}}^{\varepsilon, \nu}(s))) d\mathcal{W}_1(s)\\
		&\leq \|{\mathfrak{u}}_0\|^{2p} + \int_0^t \Big[2p\|{\mathfrak{u}}^{\varepsilon, \nu}(s)\|^{2p-2}\big[-\nu \|{\mathfrak{u}}_x^{\varepsilon, \nu}(s)\|^2 -\| {\mathfrak{u}}^{\varepsilon, \nu}(s)\|^{\beta+1}_{L^{\beta+1}} + C(1+\|{\mathfrak{u}}^{\varepsilon, \nu}(s)\|^2+\|{\mathfrak{v}}^{\varepsilon, \nu}(s)\|^2)\big]\\
		&\qquad + p\|{\mathfrak{u}}^{\varepsilon, \nu}(s)\|^{2p-2}\|\Sigma_1({\mathfrak{u}}^{\varepsilon, \nu}(s))\|^2 + 2p(p-1)\|{\mathfrak{u}}^{\varepsilon, \nu}(s)\|^{2p-4}\|{\mathfrak{u}}^{\varepsilon, \nu}(s)\|^2 \|\Sigma_1({\mathfrak{u}}^{\varepsilon, \nu}(s))\|^2 \Big] ds\\
		&\quad +2p \int_0^t \|{\mathfrak{u}}^{\varepsilon, \nu}(s)\|^{2p-2}( {\mathfrak{u}}^{\varepsilon, \nu}(s), \Sigma_1({\mathfrak{u}}^{\varepsilon, \nu}(s))) d\mathcal{W}_1(s)\\
		&\leq \|{\mathfrak{u}}_0\|^{2p} + \int_0^t \Big[2p\|{\mathfrak{u}}^{\varepsilon, \nu}(s)\|^{2p-2}\big[ C(1+\|{\mathfrak{u}}^{\varepsilon, \nu}(s)\|^2+\|{\mathfrak{v}}^{\varepsilon, \nu}(s)\|^2)\big]ds\\
		&\quad + C(p)\int_0^t \|{\mathfrak{u}}^{\varepsilon, \nu}(s)\|^{2p-2}(1+\|{\mathfrak{u}}^{\varepsilon, \nu}(s)\|^2)ds   +2p \int_0^t \|{\mathfrak{u}}^{\varepsilon, \nu}(s)\|^{2p-2}( {\mathfrak{u}}^{\varepsilon, \nu}(s), \Sigma_1({\mathfrak{u}}^{\varepsilon, \nu}(s))) d\mathcal{W}_1(s),
	\end{align*}
	for all $t\in [0,T]$. Employing the Young inequality with the condition $\frac{2p-2}{2p}+\frac{1}{p}=1$, we get
	\begin{align}\label{usebdg1}
		\|{\mathfrak{u}}^{\varepsilon, \nu}(t)\|^{2p }
		&\leq \|{\mathfrak{u}}_0\|^{2p} + \int_0^t C(p) \big(1+\|{\mathfrak{u}}^{\varepsilon, \nu}(s)\|^{2p}+\|{\mathfrak{v}}^{\varepsilon, \nu}(s)\|^{2p}\big)ds   \nonumber\\
		&\quad +2p \int_0^t \|{\mathfrak{u}}^{\varepsilon, \nu}(s)\|^{2p-2}( {\mathfrak{u}}^{\varepsilon, \nu}(s), \Sigma_1({\mathfrak{u}}^{\varepsilon, \nu}(s))) d\mathcal{W}_1(s),
	\end{align}
		for all $t\in [0,T]$. Using \eqref{v1bound} and \eqref{u1bound}, we establish the estimate
	\begin{align}
		&\mathbb{E} \bigg[\sup_{0 \leq t \leq T} \int_0^t C(p) \big(1+\|{\mathfrak{u}}^{\varepsilon, \nu}(s)\|^{2p}+\|{\mathfrak{v}}^{\varepsilon, \nu}(s)\|^{2p}\big)ds\bigg]\nonumber\\
		&=	\mathbb{E}  \bigg[\int_0^T C(p) \big(1+\|{\mathfrak{u}}^{\varepsilon, \nu}(s)\|^{2p}+\|{\mathfrak{v}}^{\varepsilon, \nu}(s)\|^{2p}\big)ds\bigg]
		\leq C(p, T, {\mathfrak{u}}_0, {\mathfrak{v}}_0).
	\end{align}
	We use the Burkholder-Davis-Gundy inequality and the Young inequality to estimate
	\begin{align*}
		& \mathbb{E} \bigg[\sup_{0 \leq t \leq T}\left|\int_0^t \|{\mathfrak{u}}^{\varepsilon, \nu}(s)\|^{2p-2}( {\mathfrak{u}}^{\varepsilon, \nu}(s), \Sigma_1({\mathfrak{u}}^{\varepsilon, \nu}(s))) d\mathcal{W}_1(s)\right|\bigg]\\
		&\leq C \mathbb{E} \bigg[\int_{0}^{T}\|{\mathfrak{u}}^{\varepsilon, \nu}(s)\|^{2(2p-2)}|( {\mathfrak{u}}^{\varepsilon, \nu}(s), \Sigma_1({\mathfrak{u}}^{\varepsilon, \nu}(s))) |^2 ds\bigg]^{\frac{1}{2}}\\
		&\leq C \mathbb{E} \bigg[\int_{0}^{T}\|{\mathfrak{u}}^{\varepsilon, \nu}(s)\|^{4p-2}\| \Sigma_1({\mathfrak{u}}^{\varepsilon, \nu}(s)) \|^2 ds\bigg]^{\frac{1}{2}}\\
		&\leq C \mathbb{E} \bigg[\sup_{0 \leq t \leq T}\|{\mathfrak{u}}^{\varepsilon, \nu}(s)\|^{2p} \int_{0}^{T}\|{\mathfrak{u}}^{\varepsilon, \nu}(s)\|^{2p-2}L^2_{\Sigma_1}(1+\|{\mathfrak{u}}^{\varepsilon, \nu}(s)\|)^2 ds\bigg]^{\frac{1}{2}}\\
		&\leq \eta \mathbb{E} \Big[\sup_{0 \leq t \leq T}\|{\mathfrak{u}}^{\varepsilon, \nu}(s)\|^{2p}\Big]+ C(\eta,p)\mathbb{E}\bigg[\int_{0}^{T} \big[1+\|{\mathfrak{u}}^{\varepsilon, \nu}(s)\|^{2p}\big]ds\bigg]\\
		&\leq C(\eta,p,T)+ \eta \mathbb{E} \Big[\sup_{0 \leq t \leq T}\|{\mathfrak{u}}^{\varepsilon, \nu}(s)\|^{2p}\Big]+ C(\eta,p,T)\mathbb{E}\bigg[\int_{0}^{T} \|{\mathfrak{u}}^{\varepsilon, \nu}(s)\|^{2p}ds\bigg].
	\end{align*}
	Taking supremum over time and mathematical expectation on \eqref{usebdg1}, one can easily get
	\begin{align*}
		\mathbb{E} \bigg[\sup_{0 \leq t \leq T}\|{\mathfrak{u}}^{\varepsilon, \nu}(t)\|^{2p }\bigg]
		&\leq \|{\mathfrak{u}}_0\|^{2p} + \mathbb{E}\bigg[ \sup_{0 \leq t \leq T} \int_0^t C(p) \big(1+\|{\mathfrak{u}}^{\varepsilon, \nu}(s)\|^{2p}+\|{\mathfrak{v}}^{\varepsilon, \nu}(s)\|^{2p}\big)ds\bigg] \\
		& \quad+2p \mathbb{E} \bigg[\sup_{0 \leq t \leq T} \int_0^t \|{\mathfrak{u}}^{\varepsilon, \nu}(s)\|^{2p-2}( {\mathfrak{u}}^{\varepsilon, \nu}(s), \Sigma_1({\mathfrak{u}}^{\varepsilon, \nu}(s))) d\mathcal{W}_1(s)\bigg]\\
		&\leq C(p, T, {\mathfrak{u}}_0, {\mathfrak{v}}_0)+\eta \mathbb{E} \bigg[\sup_{0 \leq t \leq T}\|{\mathfrak{u}}^{\varepsilon, \nu}(t)\|^{2p}\bigg].
	\end{align*}
	We choose $\eta>0$ small enough such that
	\begin{align*}
		\mathbb{E} \bigg[\sup_{0 \leq t \leq T}\|{\mathfrak{u}}^{\varepsilon, \nu}(t)\|^{2p }\bigg]
		&\leq C(p, T, {\mathfrak{u}}_0, {\mathfrak{v}}_0).
	\end{align*}
	Now, we define ${\mathfrak{w}}^{\varepsilon, \nu} =: {\mathfrak{u}}_x^{\varepsilon, \nu}$. Then ${\mathfrak{w}}^{\varepsilon, \nu}$ satisfies the system $\mathbb{P}$-a.s
	\begin{align}
		&\begin{cases}
			d{\mathfrak{w}}^{\varepsilon, \nu} = \Big[A_\nu {\mathfrak{w}}^{\varepsilon, \nu} + (\mathcal{N}({\mathfrak{u}}^{\varepsilon, \nu}) + F({\mathfrak{u}}^{\varepsilon, \nu}, {\mathfrak{v}}^{\varepsilon, \nu}))_x\Big]dt + (\Sigma_1({\mathfrak{u}}^{\varepsilon, \nu}))_x d\mathcal{W}_1, & \text{in } \mathbb{T} \times [0,T], \\
			{\mathfrak{w}}^{\varepsilon, \nu}(x,0) = {{\mathfrak{u}}_0}_{x}(x).
		\end{cases} 
	\end{align}
	We apply \Cref{Itoapplicationlem} to this equation to get $\mathbb{P}$-a.s
	\begin{align*}
		\|{\mathfrak{w}}^{\varepsilon, \nu}(t)\|^2 + 2\nu \int_0^t \|{\mathfrak{w}}^{\varepsilon, \nu}_x(s)\|^2 ds
		&= \|{\mathfrak{w}}^{\varepsilon, \nu}(0)\|^2 + \int_0^t \Big[2({\mathfrak{w}}^{\varepsilon, \nu}(s), (\mathcal{N}({\mathfrak{u}}^{\varepsilon, \nu}(s)) + F({\mathfrak{u}}^{\varepsilon, \nu}(s), {\mathfrak{v}}^{\varepsilon, \nu}(s)))_x) \\
		&\quad\quad  + \|(\Sigma_1({\mathfrak{u}}^{\varepsilon, \nu}(s)))_x\|^2\Big] ds+ 2\int_0^t ({\mathfrak{w}}^{\varepsilon, \nu}(s), (\Sigma_1({\mathfrak{u}}^{\varepsilon, \nu}(s)))_x)d\mathcal{W}_1,
	\end{align*}
	for all $t\in [0,T]$.
	From \Cref{propinnernegative}, we have
	\begin{align*}
		&\Big({\mathfrak{w}}^{\varepsilon, \nu}, (\mathcal{N}({\mathfrak{u}}^{\varepsilon, \nu}))_x\Big) = \Big((\mathcal{N}({\mathfrak{u}}^{\varepsilon, \nu}))_x, {\mathfrak{w}}^{\varepsilon, \nu}\Big) = 
		\operatorname{Re} \int_{\mathbb{T}} (\mathcal{N}({\mathfrak{u}}^{\varepsilon, \nu}))_x \overline{{\mathfrak{w}}^{\varepsilon, \nu}_x} \, dx \le 0.
	\end{align*}
	We use the assumptions in \Cref{main assumptions} and the Young inequality to get
	\begin{align*}
		&|({\mathfrak{w}}^{\varepsilon, \nu}, (F({\mathfrak{u}}^{\varepsilon, \nu}, {\mathfrak{v}}^{\varepsilon, \nu}))_x)| 
		\le \|{\mathfrak{w}}^{\varepsilon, \nu}\| \, \|(F({\mathfrak{u}}^{\varepsilon, \nu}, {\mathfrak{v}}^{\varepsilon, \nu}))_x\| \le \|{\mathfrak{w}}^{\varepsilon, \nu}\| \, \|(F({\mathfrak{u}}^{\varepsilon, \nu}, {\mathfrak{v}}^{\varepsilon, \nu}))\|_{H^1} \\
		&\leq \|{\mathfrak{w}}^{\varepsilon, \nu}\| L_F\Big(\|{\mathfrak{u}}^{\varepsilon, \nu}\|_{H^1} + \|{\mathfrak{v}}^{\varepsilon, \nu}\| + 1\Big)\le C(\|{\mathfrak{w}}^{\varepsilon, \nu}\|^2 + \|{\mathfrak{u}}^{\varepsilon, \nu}\|^2 + \|{\mathfrak{v}}^{\varepsilon, \nu}\|^2 + 1), 
	\end{align*}
	and 
	\begin{align*}
		&\|(\Sigma_1({\mathfrak{u}}^{\varepsilon, \nu}))_x\|^2 \leq L_{\Sigma_1}\Big( \|{\mathfrak{u}}^{\varepsilon, \nu}\|_{H^1}+1\Big)\le C(\|{\mathfrak{w}}^{\varepsilon, \nu}\|^2 + \|{\mathfrak{u}}^{\varepsilon, \nu}\|^2 + 1).
	\end{align*}
	The Burkholder-Davis-Gundy inequality together with Young’s inequality yield the estimate
	\begin{align*}
		&\mathbb{E}\bigg[\sup_{0 \le t \le T} \left|\int_0^t ({\mathfrak{w}}^{\varepsilon, \nu}(s), (\Sigma_1({\mathfrak{u}}^{\varepsilon, \nu}(s)))_x)d\mathcal{W}_1(s) \right|^p \bigg]\\
		&\le C(p)\mathbb{E}\bigg[\int_0^T \|{\mathfrak{w}}^{\varepsilon, \nu}(t)\|^2 \|(\Sigma_1({\mathfrak{u}}^{\varepsilon, \nu}(t)))_x\|^2 dt\bigg]^{\frac{p}{2}} \\
		&\le C(p)\mathbb{E}\bigg[\sup_{0 \le t \le T} \|{\mathfrak{w}}^{\varepsilon, \nu}(t)\|^2 
		\int_0^T \|(\Sigma_1({\mathfrak{u}}^{\varepsilon, \nu}(t)))_x\|^2 dt\bigg]^{\frac{p}{2}} \\
		&\le \eta \mathbb{E}\bigg[\sup_{0 \le t \le T} \|{\mathfrak{w}}^{\varepsilon, \nu}(t)\|^{2p} \bigg]
		+ C(p,\eta)\mathbb{E}\bigg[\int_0^T \|(\Sigma_1({\mathfrak{u}}^{\varepsilon, \nu}(t)))_x\|^2 dt\bigg]^p \\
		&\le \eta \mathbb{E}\bigg[\sup_{0 \le t \le T} \|{\mathfrak{w}}^{\varepsilon, \nu}(t)\|^{2p}\bigg]
		+ C(p,\eta,T)\left(\mathbb{E}\bigg[\int_0^T \|{\mathfrak{w}}^{\varepsilon, \nu}(t)\|^{2p} dt\bigg] 
		+ \mathbb{E}\bigg[\int_0^T \|{\mathfrak{u}}^{\varepsilon, \nu}(t)\|^{2p} dt\bigg] + 1 \right).
	\end{align*}
	Thus, we have
	\begin{align*}
		&\mathbb{E}\bigg[ \sup_{0 \le t \le T} \|{\mathfrak{w}}^{\varepsilon, \nu}(t)\|^2 
		+ 2\nu \int_0^T \|{\mathfrak{w}}^{\varepsilon, \nu}_x(t)\|^2 dt \bigg]^p\\
		&\le C\|{{\mathfrak{u}}_0}_x\|^{2p} + \eta \mathbb{E}\bigg[\sup_{0 \le t \le T} \|{\mathfrak{w}}^{\varepsilon, \nu}(t)\|^{2p}\bigg] + C(p,\eta,T)\bigg( \mathbb{E}\bigg[\int_0^T \Big(\|{\mathfrak{w}}^{\varepsilon, \nu}(t)\|^{2p}+\\
		&\quad\quad  \|{\mathfrak{u}}^{\varepsilon, \nu}(t)\|^{2p}+ \|{\mathfrak{v}}^{\varepsilon, \nu}(t)\|^{2p}\Big)dt\bigg] + 1 \bigg).
	\end{align*}
	From the above estimates \eqref{v1bound} and \eqref{u1bound}, it follows that
	\begin{align*}
		\mathbb{E}\left[\sup_{0 \le t \le T} \|{\mathfrak{w}}^{\varepsilon, \nu}(t)\|^{2p}\right]
		&\le C(p,\eta,T,{\mathfrak{u}}_0,{\mathfrak{v}}_0) + \eta \mathbb{E}\left[\sup_{0 \le t \le T} \|{\mathfrak{w}}^{\varepsilon, \nu}(t)\|^{2p}\right] + C(p,\eta,T)\mathbb{E}\bigg[\int_0^T \|{\mathfrak{w}}^{\varepsilon, \nu}(t)\|^{2p} dt\bigg].
	\end{align*}
	We choose $\eta$ very small such that
	\begin{align*}
		\mathbb{E}\left[\sup_{0 \le t \le T} \|{\mathfrak{w}}^{\varepsilon, \nu}(t)\|^{2p}\right]
		\le C(p,T,{\mathfrak{u}}_0,{\mathfrak{v}}_0) + C(p,T)\mathbb{E}\bigg[\int_0^T \|{\mathfrak{w}}^{\varepsilon, \nu}(t)\|^{2p} dt\bigg].
	\end{align*}
	It follows from the Gr\"onwall inequality that
	\begin{align*}
		\mathbb{E}\left[\sup_{0 \le t \le T} \|{\mathfrak{w}}^{\varepsilon, \nu}(t)\|^{2p} \right]
		\le C(p,T,{\mathfrak{u}}_0,{\mathfrak{v}}_0).
	\end{align*}
	Hence, we obtain
	\begin{align*}
		\mathbb{E}\bigg[ \sup_{0 \le t \le T} \|{\mathfrak{w}}^{\varepsilon, \nu}(t)\|^2 
		+ \nu \int_0^T \|{\mathfrak{w}}^{\varepsilon, \nu}_x(t)\|^2 dt \bigg]^p
		\le C(p,T,{\mathfrak{u}}_0,{\mathfrak{v}}_0),
	\end{align*}
	which implies 
	\begin{align*}
		\mathbb{E}\bigg[ \sup_{0 \le t \le T} \|{\mathfrak{w}}^{\varepsilon, \nu}(t)\|^{2p} \bigg]
		+ \mathbb{E}\bigg[\nu \int_0^T \|{\mathfrak{w}}^{\varepsilon, \nu}_x(t)\|^2 dt \bigg]^p
		\le C(p,T,{\mathfrak{u}}_0,{\mathfrak{v}}_0).
	\end{align*}
	Since ${\mathfrak{w}}^{\varepsilon, \nu} = {\mathfrak{u}}_x^{\varepsilon, \nu}$, consequently, we get the estimate
	\begin{align*}
		\sup_{\varepsilon \in (0,1), \nu \in (0,1)}\mathbb{E}\bigg[\nu \int_0^T \|\Delta {\mathfrak{u}}^{\varepsilon, \nu}(t)\|^2 dt \bigg]^p
		\le C(p,T,{\mathfrak{u}}_0,{\mathfrak{v}}_0).
	\end{align*}
	With the help of the estimate of 
	$\mathbb{E}\left[\sup_{0 \le t \le T}\|{\mathfrak{u}}^{\varepsilon, \nu}(t)\|^{2p}\right]$ 
	and 
	$\mathbb{E}\left[\sup_{0 \le t \le T}\|{\mathfrak{w}}^{\varepsilon, \nu}(t)\|^{2p}\right]$, 
	we conclude
	\begin{align*}
		\mathbb{E}\left[\sup_{0 \le t \le T}\|{\mathfrak{u}}^{\varepsilon, \nu}(t)\|_{H^1}^{2p}\right]\leq C(p,T,{\mathfrak{u}}_0,{\mathfrak{v}}_0),
	\end{align*}
	which completes the proof.
\end{proof}
By applying the same method as in the preceding proposition and invoking similar estimates, the following propositions can be established. Since the arguments are essentially similar, we state them below without repeating the details of the proof.
\begin{prop}\label{viscousaveragingsolution}
	Let ${\mathfrak{u}}_0 \in H^1(\mathbb{T})$. Then the viscous averaged equation \eqref{viscousavg} admits a unique solution
	\begin{align*}
		\bar{\mathfrak{u}}^{\nu} \in L^2(\Omega, C([0,T];H^1(\mathbb{T}))).
	\end{align*}
	Furthermore, for every $p > 0$, there exists a positive constant $C=(p, T, {\mathfrak{u}}_0)$ such that the solution 
	$\bar{\mathfrak{u}}^{\nu}$ satisfies
	\begin{align*}
		\mathbb{E}\left[\sup_{0 \le t \le T}\|\bar{\mathfrak{u}}^{\nu}(t)\|_{H^1}^{2p}\right] \le C.
	\end{align*}
\end{prop}
\begin{prop}\label{averagingsolution}
	Let ${\mathfrak{u}}_0 \in H^1(\mathbb{T})$. Then the averaged equation \eqref{averagedequation}admits a unique solution 
	\begin{align*}
		\bar{\mathfrak{u}} \in L^2(\Omega, C([0,T];H^1(\mathbb{T}))).
	\end{align*}
	Furthermore, for every $p > 0$, there exists a positive constant $C=(p, T, {\mathfrak{u}}_0)$ such that the solution $\bar{\mathfrak{u}}$ satisfies
	\begin{align*}
		\mathbb{E}\left[\sup_{0 \le t \le T}\|\bar{\mathfrak{u}}(t)\|_{H^1}^{2p}\right] \le C.
	\end{align*}
\end{prop}

\subsection{The error estimates}\label{subsecerrorestimate}
In this subsection, we establish the error estimates for ${\mathfrak{u}}^{\varepsilon, \nu}-{\mathfrak{u}}^{\varepsilon}, \,\, {\mathfrak{v}}^{\varepsilon, \nu}-{\mathfrak{v}}^{\varepsilon}$ and $\bar{\mathfrak{u}}^{\nu}-\bar{\mathfrak{u}}$.
\begin{prop}\label{uepsnuueps}
	Let $({\mathfrak{u}}_0, {\mathfrak{v}}_0) \in H^1(\mathbb{T}) \times H^1(\mathbb{T})$ and $\varepsilon, \nu \in (0,1)$. Let $({\mathfrak{u}}^{\varepsilon, \nu}, {\mathfrak{v}}^{\varepsilon, \nu})$ be the solution to \eqref{viscous} and $({\mathfrak{u}}^{\varepsilon}, {\mathfrak{v}}^{\varepsilon})$ be the solution to \eqref{multiscaleeq}. Then there exists a positive constant $C=C(T, {\mathfrak{u}}_0, {\mathfrak{v}}_0)$ independent of $\varepsilon, \nu$ such that the following holds:
	\begin{align*}
		\sup_{\varepsilon \in (0,1)} \bigg[ 
		\sup_{0 \le t \le T} \mathbb{E}\big[\|{\mathfrak{u}}^{\varepsilon, \nu}(t) - {\mathfrak{u}}^{\varepsilon}(t)\|^2\big] 
		+ 
		\sup_{0 \le t \le T} \mathbb{E}\big[\|{\mathfrak{v}}^{\varepsilon, \nu}(t) - {\mathfrak{v}}^{\varepsilon}(t)\|^2\big]
		\bigg] \le C \nu.
	\end{align*}
	More precisely,
	\begin{align*}
		\lim_{\nu \to 0} \sup_{\varepsilon \in (0,1)} \sup_{0 \le t \le T} 
		\mathbb{E}\big[\|{\mathfrak{u}}^{\varepsilon, \nu}(t) - {\mathfrak{u}}^{\varepsilon}(t)\|^2\big] = 0.
	\end{align*}
\end{prop}
\begin{proof}
	Let $X^{\varepsilon,\nu} = {\mathfrak{u}}^{\varepsilon, \nu} - {\mathfrak{u}}^{\varepsilon}$, 
	$Y^{\varepsilon,\nu} = {\mathfrak{v}}^{\varepsilon, \nu} - {\mathfrak{v}}^{\varepsilon}$. From \eqref{multiscaleeq} and \eqref{viscous}, we have
	\begin{align*}
		\left\{
		\begin{aligned}
			&dX^{\varepsilon,\nu} = \big[-i(-\Delta)^{\alpha}X^{\varepsilon,\nu} + \mathcal{N}({\mathfrak{u}}^{\varepsilon, \nu}) - \mathcal{N}({\mathfrak{u}}^{\varepsilon}) 
			+ F({\mathfrak{u}}^{\varepsilon, \nu}, {\mathfrak{v}}^{\varepsilon, \nu}) - F({\mathfrak{u}}^{\varepsilon}, {\mathfrak{v}}^{\varepsilon}) 
			- {\nu}\Delta {\mathfrak{u}}^{\varepsilon, \nu}\big] dt\\
			&\quad \qquad\quad+ \big(\Sigma_1({\mathfrak{u}}^{\varepsilon, \nu}) - \Sigma_1({\mathfrak{u}}^{\varepsilon})\big)d\mathcal{W}_1, \\
			&dY^{\varepsilon,\nu} =\tfrac{1}{\varepsilon}\big[(1+i)(-\Delta)^{\rho} Y^{\varepsilon,\nu} 
			+ \mathcal{N}({\mathfrak{v}}^{\varepsilon, \nu}) - \mathcal{N}({\mathfrak{v}}^{\varepsilon}) - \lambda Y^{\varepsilon,\nu} 
			+ G({\mathfrak{u}}^{\varepsilon, \nu}, {\mathfrak{v}}^{\varepsilon, \nu}) - G({\mathfrak{u}}^{\varepsilon}, {\mathfrak{v}}^{\varepsilon})\big] dt\\
			&\quad \qquad\quad+ \tfrac{1}{\sqrt{\varepsilon}}\big[\Sigma_2({\mathfrak{u}}^{\varepsilon, \nu}, {\mathfrak{v}}^{\varepsilon, \nu}) 
			- \Sigma_2({\mathfrak{u}}^{\varepsilon}, {\mathfrak{v}}^{\varepsilon})\big] d\mathcal{W}_2, \\
			&X^{\varepsilon,\nu}(x,0) = 0, \\
			&Y^{\varepsilon,\nu}(x,0) = 0.
		\end{aligned}
		\right.
	\end{align*}
	Applying the infinite-dimensional It\^o formula to the process $\|Y^{\varepsilon,\nu}\|^2$ yields $\mathbb{P}$-a.s
	\begin{align*}
		d\|Y^{\varepsilon,\nu}\|^2 
		&= \frac{2}{\varepsilon} 
		\Big\langle (1+i)(-\Delta)^{\rho} Y^{\varepsilon,\nu} + \mathcal{N}({\mathfrak{v}}^{\varepsilon, \nu}) - \mathcal{N}({\mathfrak{v}}^{\varepsilon}) 
		- \lambda Y^{\varepsilon,\nu} 
		+ G({\mathfrak{u}}^{\varepsilon, \nu}, {\mathfrak{v}}^{\varepsilon, \nu}) - G({\mathfrak{u}}^{\varepsilon}, {\mathfrak{v}}^{\varepsilon}),
		\, Y^{\varepsilon,\nu} \Big\rangle dt \\
		&\quad + \frac{1}{\varepsilon} 
		\|\Sigma_2({\mathfrak{u}}^{\varepsilon, \nu}, {\mathfrak{v}}^{\varepsilon, \nu}) 
		- \Sigma_2({\mathfrak{u}}^{\varepsilon}, {\mathfrak{v}}^{\varepsilon})\|^2 dt  + \frac{2}{\sqrt{\varepsilon}} 
		\Big\langle Y^{\varepsilon,\nu}, 
		\Sigma_2({\mathfrak{u}}^{\varepsilon, \nu}, {\mathfrak{v}}^{\varepsilon, \nu}) 
		- \Sigma_2({\mathfrak{u}}^{\varepsilon}, {\mathfrak{v}}^{\varepsilon}) \Big\rangle d\mathcal{W}_2,
	\end{align*}
	for a.e. $t\in [0,T]$. Taking the mathematical expectation on both sides of the above equation, it follows that
	\begin{align*}
		\mathbb{E}\big[\|Y^{\varepsilon,\nu}(t)\|^2 \big]
		&= \frac{2}{\varepsilon} \mathbb{E} \bigg[
		\int_0^t 
		\Big\langle (1+i)(-\Delta)^{\rho} Y^{\varepsilon,\nu}(s) 
		+ \mathcal{N}({\mathfrak{v}}^{\varepsilon, \nu}(s)) - \mathcal{N}({\mathfrak{v}}^{\varepsilon}(s)) 
		- \lambda Y^{\varepsilon,\nu}(s) \\
		&\quad \quad+ G({\mathfrak{u}}^{\varepsilon, \nu}(s), {\mathfrak{v}}^{\varepsilon, \nu}(s)) 
		- G({\mathfrak{u}}^{\varepsilon}(s), {\mathfrak{v}}^{\varepsilon}(s)),
		Y^{\varepsilon,\nu}(s) \Big\rangle ds \bigg]\\
		&\quad
		+ \frac{1}{\varepsilon} \mathbb{E} \bigg[
		\int_0^t 
		\|\Sigma_2({\mathfrak{u}}^{\varepsilon, \nu}(s), {\mathfrak{v}}^{\varepsilon, \nu}(s)) 
		- \Sigma_2({\mathfrak{u}}^{\varepsilon}(s), {\mathfrak{v}}^{\varepsilon}(s))\|^2 ds\bigg],
	\end{align*}
	for all $t\in [0,T]$. Differentiation with respect to time leads to
	\begin{align*}
		&\frac{d}{dt}\mathbb{E}\big[\|Y^{\varepsilon,\nu}(t)\|^2 \big]
		= \frac{2}{\varepsilon} \mathbb{E} \Big[
		\Big\langle (1+i)(-\Delta)^{\rho} Y^{\varepsilon,\nu}(t) 
		+ \mathcal{N}({\mathfrak{v}}^{\varepsilon, \nu}(t)) - \mathcal{N}({\mathfrak{v}}^{\varepsilon}(t)) 
		- \lambda Y^{\varepsilon,\nu}(t) \\
		&\quad + G({\mathfrak{u}}^{\varepsilon, \nu}(t), {\mathfrak{v}}^{\varepsilon, \nu}(t)) 
		- G({\mathfrak{u}}^{\varepsilon}(t), {\mathfrak{v}}^{\varepsilon}(t)),
		Y^{\varepsilon,\nu}(t) \Big\rangle \Big]
		+ \frac{1}{\varepsilon} \mathbb{E} 
		\Big[\|\Sigma_2({\mathfrak{u}}^{\varepsilon, \nu}(t), {\mathfrak{v}}^{\varepsilon, \nu}(t)) 
		- \Sigma_2({\mathfrak{u}}^{\varepsilon}(t), {\mathfrak{v}}^{\varepsilon}(t))\|^2\Big],
	\end{align*}
	for a.e. $t\in [0,T]$. Here, we apply integration by parts to compute
	\begin{align*}
		\Big\langle (1+i)(-\Delta)^{\rho} Y^{\varepsilon,\nu} , Y^{\varepsilon,\nu}  \Big\rangle=-\left( (1+i)(-\Delta)^{\frac{\rho}{2}}Y^{\varepsilon,\nu} , (-\Delta)^{\frac{\rho}{2}} Y^{\varepsilon,\nu}  \right)=- \|(-\Delta)^{\frac{\rho}{2}} Y^{\varepsilon,\nu} \|^2\leq 0.
	\end{align*}
	As a consequence of \Cref{diffnonrenega}, one can easily check
	\begin{align*}
		(\mathcal{N}({\mathfrak{v}}^{\varepsilon, \nu}) - \mathcal{N}({\mathfrak{v}}^{\varepsilon}), Y^{\varepsilon,\nu}) \le 0.
	\end{align*}
	Therefore, we obtain
	\begin{align*}
		\frac{d}{dt} \mathbb{E}\big[\|Y^{\varepsilon,\nu}(t)\|^2 \big]
		&\le -\frac{2\lambda}{\varepsilon} \mathbb{E}\big[\|Y^{\varepsilon,\nu}(t)\|^2 \big]
		+ \frac{2}{\varepsilon} \mathbb{E} \big[\langle 
		G({\mathfrak{u}}^{\varepsilon, \nu}(t), {\mathfrak{v}}^{\varepsilon, \nu}(t)) - G({\mathfrak{u}}^{\varepsilon}(t), {\mathfrak{v}}^{\varepsilon}(t)), 
		Y^{\varepsilon,\nu}(t) \rangle \big]\\
		&\quad + \frac{1}{\varepsilon} \mathbb{E}\big[\|\Sigma_2({\mathfrak{u}}^{\varepsilon, \nu}(t), {\mathfrak{v}}^{\varepsilon, \nu}(t)) 
		- \Sigma_2({\mathfrak{u}}^{\varepsilon}(t), {\mathfrak{v}}^{\varepsilon}(t))\|^2\big] \\
		&\le -\frac{2\lambda}{\varepsilon} \mathbb{E}\big[\|Y^{\varepsilon,\nu}(t)\|^2\big] 
		+ \frac{2L_G}{\varepsilon} \mathbb{E}\big[\|Y^{\varepsilon,\nu}(t)\|(\|X^{\varepsilon,\nu}(t)\| + \|Y^{\varepsilon,\nu}(t)\|)\big] \\
		&\quad + \frac{L_{\Sigma_2}^2}{\varepsilon} \mathbb{E}\big[\|X^{\varepsilon,\nu}(t)\| + \|Y^{\varepsilon,\nu}(t)\|\big]^2,
	\end{align*}
	for a.e. $t\in [0,T]$. An application of the Young inequality leads us to the following estimate:
	\begin{align*}
		&2L_G \mathbb{E}\big[\|Y^{\varepsilon,\nu}(t)\|(\|X^{\varepsilon,\nu}(t)\| + \|Y^{\varepsilon,\nu}(t)\|) \big]
		+ L_{\Sigma_2}^2 \mathbb{E}\big[\|X^{\varepsilon,\nu}(t)\| + \|Y^{\varepsilon,\nu}(t)\|\big]^2 \\
		&\le (3L_G + 2L_{\Sigma_2}^2)\mathbb{E}\big[\|Y^{\varepsilon,\nu}(t)\|^2 \big]
		+ (L_G + 2L_{\Sigma_2}^2)\mathbb{E}\big[\|X^{\varepsilon,\nu}(t)\|^2\big] \\
		&\le (3L_G + 2L_{\Sigma_2}^2)\Big(\mathbb{E}\big[\|Y^{\varepsilon,\nu}(t)\|^2\big]
		+ \mathbb{E}\big[\|X^{\varepsilon,\nu}(t)\|^2\big]\Big).
	\end{align*}
	This implies
	\begin{align*}
		\frac{d}{dt}\mathbb{E}\big[\|Y^{\varepsilon,\nu}(t)\|^2 \big]
		\le -\frac{2\lambda}{\varepsilon} \mathbb{E}\big[\|Y^{\varepsilon,\nu}(t)\|^2 \big]
		+ \frac{(3L_G + 2L_{\Sigma_2}^2)}{\varepsilon} 
		\Big( \mathbb{E}\big[\|Y^{\varepsilon,\nu}(t)\|^2\big] + \mathbb{E}\big[\|X^{\varepsilon,\nu}(t)\|^2 \big]\Big),
	\end{align*}
	for a.e. $t\in [0,T]$. As a consequence of the \Cref{main assumptions}, it follows that
	\begin{align*}
		\frac{d}{dt}\mathbb{E}\big[\|Y^{\varepsilon,\nu}(t)\|^2 \big]
		\le -\frac{\lambda}{\varepsilon} \mathbb{E}\big[\|Y^{\varepsilon,\nu}(t)\|^2 \big]
		+ \frac{C(g,\Sigma_2)}{\varepsilon} \mathbb{E}\big[\|X^{\varepsilon,\nu}(t)\|^2\big],
	\end{align*}
	for a.e. $t\in [0,T]$. Hence, an application of the variation of constants formula yields
	\begin{align}\label{comparisony}
		\mathbb{E}\big[\|Y^{\varepsilon,\nu}(t)\|^2 \big]
		\le \frac{C(g,\Sigma_2)}{\varepsilon} 
		\int_0^t e^{-\frac{\lambda}{\varepsilon}(t-s)} 
		\mathbb{E}\big[\|X^{\varepsilon,\nu}(s)\|^2\big] ds,
	\end{align}
	for all $t\in [0,T]$. Applying the infinite-dimensional It\^o formula to the process $\|X^{\varepsilon,\nu}\|^2$, we obtain $\mathbb{P}$-a.s
	\begin{align*}
		d\|X^{\varepsilon,\nu}\|^2 
		&= 2\langle -i(-\Delta)^{\alpha}X^{\varepsilon,\nu} 
		+ \mathcal{N}({\mathfrak{u}}^{\varepsilon, \nu}) - \mathcal{N}({\mathfrak{u}}^{\varepsilon}) 
		+ F({\mathfrak{u}}^{\varepsilon, \nu}, {\mathfrak{v}}^{\varepsilon, \nu}) 
		- F({\mathfrak{u}}^{\varepsilon}, {\mathfrak{v}}^{\varepsilon}) \\
		&\quad + \nu \Delta {\mathfrak{u}}^{\varepsilon, \nu}, 
		X^{\varepsilon,\nu} \rangle dt 
		+ \|\Sigma_1({\mathfrak{u}}^{\varepsilon, \nu}) - \Sigma_1({\mathfrak{u}}^{\varepsilon})\|^2 dt 
		+ 2\langle X^{\varepsilon,\nu}, 
		\Sigma_1({\mathfrak{u}}^{\varepsilon, \nu}) - \Sigma_1({\mathfrak{u}}^{\varepsilon}) \rangle d\mathcal{W}_1,
	\end{align*}
	for a.e. $t\in [0,T]$. Taking the mathematical expectation on both sides of the above equation leads to
	\begin{align*}
		\mathbb{E}\big[\|X^{\varepsilon,\nu}(t)\|^2 \big]
		&= \mathbb{E}\bigg[\int_0^t 
		\Big\{
		2\langle -i(-\Delta)^{\alpha}X^{\varepsilon,\nu}(s) 
		+ \mathcal{N}({\mathfrak{u}}^{\varepsilon, \nu}(s)) - \mathcal{N}({\mathfrak{u}}^{\varepsilon}(s)) 
		+ F({\mathfrak{u}}^{\varepsilon, \nu}(s), {\mathfrak{v}}^{\varepsilon, \nu}(s))\\
		&\quad\quad - F({\mathfrak{u}}^{\varepsilon}(s), {\mathfrak{v}}^{\varepsilon}(s))  + \nu \Delta {\mathfrak{u}}^{\varepsilon, \nu}(s), X^{\varepsilon,\nu}(s) \rangle 
		+ \|\Sigma_1({\mathfrak{u}}^{\varepsilon, \nu}(s)) - \Sigma_1({\mathfrak{u}}^{\varepsilon}(s))\|^2 
		\Big\} ds\bigg],
	\end{align*}
	for all $t\in [0,T]$. Differentiating with respect to time, we get
	\begin{align*}
		\frac{d}{dt}\mathbb{E}\big[\|X^{\varepsilon,\nu}(t)\|^2 \big]
		&= \mathbb{E}\Big[ 
		2\langle -i(-\Delta)^{\alpha}X^{\varepsilon,\nu}(t) 
		+ \mathcal{N}({\mathfrak{u}}^{\varepsilon, \nu}(t)) - \mathcal{N}({\mathfrak{u}}^{\varepsilon}(t)) 
		+ F({\mathfrak{u}}^{\varepsilon, \nu}(t), {\mathfrak{v}}^{\varepsilon, \nu}(t)) - F({\mathfrak{u}}^{\varepsilon}(t), {\mathfrak{v}}^{\varepsilon}(t)) \\
		&\quad + \nu \Delta {\mathfrak{u}}^{\varepsilon, \nu}(t), X^{\varepsilon,\nu}(t) \rangle 
		+ \|\Sigma_1({\mathfrak{u}}^{\varepsilon, \nu}(t)) - \Sigma_1({\mathfrak{u}}^{\varepsilon}(t))\|^2 \Big] \\
		&= \mathbb{E}\Big[ 
		2\langle \mathcal{N}({\mathfrak{u}}^{\varepsilon, \nu}(t)) - \mathcal{N}({\mathfrak{u}}^{\varepsilon}(t)) 
		+ F({\mathfrak{u}}^{\varepsilon, \nu}(t), {\mathfrak{v}}^{\varepsilon, \nu}(t)) - F({\mathfrak{u}}^{\varepsilon}(t), {\mathfrak{v}}^{\varepsilon}(t)) 
		+ \nu \Delta {\mathfrak{u}}^{\varepsilon, \nu}(t), X^{\varepsilon,\nu}(t) \rangle \\
		&\quad+ \|\Sigma_1({\mathfrak{u}}^{\varepsilon, \nu}(t)) - \Sigma_1({\mathfrak{u}}^{\varepsilon}(t))\|^2 
		\Big],
	\end{align*}
	for a.e. $t\in [0,T]$. As a consequence of \Cref{diffnonrenega}, one can easily check
	\begin{align*}
		(\mathcal{N}({\mathfrak{u}}^{\varepsilon, \nu}) - \mathcal{N}({\mathfrak{u}}^{\varepsilon}), X^{\varepsilon,\nu}) \le 0.
	\end{align*}
	Therefore, we obtain
	\begin{align}
		\frac{d}{dt}\mathbb{E}\big[\|X^{\varepsilon,\nu}(t)\|^2 \big]
		&\le \mathbb{E}\Big[ 2(F({\mathfrak{u}}^{\varepsilon, \nu}(t), {\mathfrak{v}}^{\varepsilon, \nu}(t)) - F({\mathfrak{u}}^{\varepsilon}(t), {\mathfrak{v}}^{\varepsilon}(t)) + \nu \Delta {\mathfrak{u}}^{\varepsilon, \nu}(t), X^{\varepsilon,\nu}(t))\nonumber \\
		&\quad\quad+ |\Sigma_1({\mathfrak{u}}^{\varepsilon, \nu}(t)) - \Sigma_1({\mathfrak{u}}^{\varepsilon}(t))|^2 \Big] \nonumber\\
		&\le C \Big(\mathbb{E}\big[\|X^{\varepsilon,\nu}(t)\|^2\big] + \mathbb{E}\big[\|Y^{\varepsilon,\nu}(t)\|^2\big] + \nu^2 \mathbb{E}\big[\|\Delta {\mathfrak{u}}^{\varepsilon, \nu}(t)\|^2\big]\Big),
	\end{align}
	for a.e. $t\in [0,T]$. Applying the variation of constants formula, we have
	\begin{align}\label{comparisonx}
		\mathbb{E}\big[\|X^{\varepsilon,\nu}(t)\|^2 \big]
		&\le C \int_0^t e^{C(t-s)} \mathbb{E}\Big[\|Y^{\varepsilon,\nu}(s)\|^2 + \nu^2 \|\Delta {\mathfrak{u}}^{\varepsilon, \nu}(s)\|^2\Big]ds\nonumber\\
		&\le C \int_0^t \mathbb{E}\Big[\|Y^{\varepsilon,\nu}(s)\|^2 + \nu^2 \|\Delta {\mathfrak{u}}^{\varepsilon, \nu}(s)\|^2\Big]ds,
	\end{align}
	for all $t\in [0,T]$. The estimates \eqref{comparisonx} and \eqref{comparisony} together imply
	\begin{align*}
		\mathbb{E}\big[\|Y^{\varepsilon,\nu}(t)\|^2 \big]
		&\le \frac{C}{\varepsilon} \int_0^t e^{-\frac{\lambda}{\varepsilon}(t-s)} \int_0^s \mathbb{E}\Big[\|Y^{\varepsilon,\nu}(\tau)\|^2 + \nu^2 \|\Delta {\mathfrak{u}}^{\varepsilon, \nu}(\tau)\|^2\Big]d\tau ds \\
		&\le \frac{C}{\lambda} \int_0^t (1 - e^{-\frac{\lambda}{\varepsilon}(t-\tau)}) \mathbb{E}\Big[\|Y^{\varepsilon,\nu}(\tau)\|^2 + \nu^2 \|\Delta {\mathfrak{u}}^{\varepsilon, \nu}(\tau)\|^2\Big]d\tau \\
		&\le C \int_0^t \mathbb{E}\Big[\|Y^{\varepsilon,\nu}(\tau)\|^2 + \nu^2 \|\Delta {\mathfrak{u}}^{\varepsilon, \nu}(\tau)\|^2\Big]d\tau,
	\end{align*}
	for all $t\in [0,T]$. An application of the Gr\"onwall inequality yields
	\begin{align*}
		\sup_{0 \le t \le T} \mathbb{E}\big[\|Y^{\varepsilon,\nu}(t)\|^2 \big]
		&\le C \nu^2 \int_0^T \mathbb{E}\big[\|\Delta {\mathfrak{u}}^{\varepsilon, \nu}(t)\|^2\big] dt 
		\le C \nu \mathbb{E}\bigg[\int_0^T \nu \|\Delta {\mathfrak{u}}^{\varepsilon, \nu}(t)\|^2 dt\bigg].
	\end{align*}
	Thus, \eqref{comparisonx} implies
	\begin{align*}
		\mathbb{E}\big[\|X^{\varepsilon,\nu}(t)\|^2 \big]
		&\le C \nu \mathbb{E}\bigg[ \int_0^T \nu \|\Delta {\mathfrak{u}}^{\varepsilon, \nu}(t)\|^2 dt\bigg],
	\end{align*}
	for all $t\in [0,T]$. Using the estimates established in \Cref{uniformestimates}, we obtain
	\begin{align*}
		\sup_{0 \le t \le T} \mathbb{E}\big[\|X^{\varepsilon,\nu}(t)\|^2 \big]
		+ \sup_{0 \le t \le T} \mathbb{E}\big[\|Y^{\varepsilon,\nu}(t)\|^2 \big]
		\le C \nu.
	\end{align*}
	Since $C$ is independent of $\varepsilon$, this implies that
	\begin{align*}
		\sup_{\varepsilon \in (0,1)} 
		\Big( \sup_{0 \le t \le T} \mathbb{E}\big[\|X^{\varepsilon,\nu}(t)\|^2 \big]
		+ \sup_{0 \le t \le T} \mathbb{E}\big[\|Y^{\varepsilon,\nu}(t)\|^2\big] \Big)
		\le C \nu.
	\end{align*}
	Therefore, we conclude
	\begin{align*}
		\sup_{\varepsilon \in (0,1)} \left( 
		\sup_{0 \le t \le T} \mathbb{E}\big[\|{\mathfrak{u}}^{\varepsilon, \nu}(t) - {\mathfrak{u}}^{\varepsilon}(t)\|^2 \big]
		+ 
		\sup_{0 \le t \le T} \mathbb{E}\big[\|{\mathfrak{v}}^{\varepsilon, \nu}(t) - {\mathfrak{v}}^{\varepsilon}(t)\|^2\big]
		\right) \le C \nu.
	\end{align*}
	In particular, we get 
	\begin{align*}
		\lim_{\nu \to 0} \sup_{\varepsilon \in (0,1)} \sup_{0 \le t \le T} 
		\mathbb{E}\big[\|{\mathfrak{u}}^{\varepsilon, \nu}(t) - {\mathfrak{u}}^{\varepsilon}(t)\|^2\big] = 0.
	\end{align*}
	Using the same argument as in the preceding proposition, and relying on analogous estimates, the following proposition can be established. Since the arguments are essentially similar, the proof is omitted.
\end{proof}
\begin{prop}\label{ubarnuandubar}
	Let ${\mathfrak{u}}_0 \in H^1(\mathbb{T})$ and $\nu \in (0,1)$. Let $\bar{\mathfrak{u}}^\nu$ be the solution to \eqref{viscousavg} and $\bar{\mathfrak{u}}$ be the solution to \eqref{averagedequation}, then there exists a positive constant $C=C(T, {\mathfrak{u}}_0)$ independent of $\nu$ such that
	\begin{align*}
		\sup_{0 \le t \le T} \mathbb{E}\big[\|\bar{\mathfrak{u}}^\nu(t) - \bar{\mathfrak{u}}(t)\|^2\big] \le C \nu.
	\end{align*}
	More precisely,
	\begin{align*}
		\lim_{\nu \to 0} \sup_{0 \le t \le T} \mathbb{E}\big[\|\bar{\mathfrak{u}}^\nu(t) - \bar{\mathfrak{u}}(t)\|^2\big] = 0.
	\end{align*}
\end{prop} 

\section{Proof of Step III}\label{secstepthree}
In this section, we prove the \textbf{Step III} of \Cref{keystrategy}, that is the averaging principle for the viscous system \eqref{viscous}
\begin{prop}\label{apofviscoused}
	Fix $\nu \in (0,1)$, and let $p \ge 1$. Under the assumptions in \Cref{main assumptions} with
	$({\mathfrak{u}}_0,{\mathfrak{v}}_0)\in H^1(\mathbb{T})\times H^1(\mathbb{T})$, let
	$({\mathfrak{u}}^{\varepsilon, \nu},{\mathfrak{v}}^{\varepsilon, \nu})$ be the solution of \eqref{viscous} and
	$\bar{\mathfrak{u}}^{\nu}$ be the solution of \eqref{viscousavg}. Then, it holds that
	\begin{align*}
		\lim_{\varepsilon\to 0}
		\mathbb{E}\!\bigg[
		\sup_{0\le t\le T}
		\|{\mathfrak{u}}^{\varepsilon, \nu}(t)-\bar{\mathfrak{u}}^{\nu}(t)\|^{2p}
		\bigg]=0.
	\end{align*}
\end{prop}
\noindent To prove \Cref{apofviscoused}, we need several intermediate results. First, the H\"older continuity of ${\mathfrak{u}}^{\varepsilon, \nu}$ in time variable is very crucial. Then, we define the auxiliary processes and establish the error estimate of them from the solution to \eqref{viscous} and \eqref{viscousavg}. Finally using these results, we prove \Cref{apofviscoused} in \Cref{avgprnprf}.
\subsection{H\"older continuity in time for ${\mathfrak{u}}^{\varepsilon, \nu}$}\label{subsecholdercts}
 Here, we establish the H\"older continuity of ${\mathfrak{u}}^{\varepsilon, \nu}$ in time variable in the following proposition:
\begin{prop}\label{holderctsintime}
	Let $p\ge \frac{1}{2}$ and $({\mathfrak{u}}_0,{\mathfrak{v}}_0)\in H^1(\mathbb{T})\times H^1(\mathbb{T})$. Then
	there exists a positive constant $C$ depending on $p,\nu,\alpha,T$ such that
	\begin{equation}\label{holderintime}
		\mathbb{E}\Big[\bigl\|{\mathfrak{u}}^{\varepsilon, \nu}(t+h)-{\mathfrak{u}}^{\varepsilon, \nu}(t)\bigr\|^{2p}\Big]
		\le C h^{p},
	\end{equation}
	for any $0\le t\le t+h\le T$.
\end{prop}
\begin{proof}
	We have
	\begin{align*}
		{\mathfrak{u}}^{\varepsilon, \nu}(t+h)
		&=S_\nu(h){\mathfrak{u}}^{\varepsilon, \nu}(t)
		+\int_t^{t+h} S_\nu(t+h-s)\big[\mathcal{N}({\mathfrak{u}}^{\varepsilon, \nu}(s))+F({\mathfrak{u}}^{\varepsilon, \nu}(s),{\mathfrak{v}}^{\varepsilon, \nu}(s))\big]\,ds \\
		&\quad +\int_t^{t+h} S_\nu(t+h-s)\Sigma_1({\mathfrak{u}}^{\varepsilon, \nu}(s))\,d\mathcal{W}_1(s),\quad \mathbb{P} \text{-a.s.}
	\end{align*}
  Therefore,
	\begin{align*}
		{\mathfrak{u}}^{\varepsilon, \nu}(t+h)-{\mathfrak{u}}^{\varepsilon, \nu}(t)
		&=(S_\nu(h)-I_d) {\mathfrak{u}}^{\varepsilon, \nu}(t)
		+\int_t^{t+h} S_\nu(t+h-s)\big[\mathcal{N}({\mathfrak{u}}^{\varepsilon, \nu}(s))+F({\mathfrak{u}}^{\varepsilon, \nu}(s),{\mathfrak{v}}^{\varepsilon, \nu}(s))\big]\,ds \\
		&\quad +\int_t^{t+h} S_\nu(t+h-s)\Sigma_1({\mathfrak{u}}^{\varepsilon, \nu}(s))\,d\mathcal{W}_1(s),
	\end{align*}
	where $I_d$ is the identity operator. \Cref{property of semigroup} implies the existence of a constant $C$ such that for all $x\in H^1$, 
	\begin{align}\label{smoothness}
		\|(S_\nu(h)-I_d)x\| \leq C h^{\frac{1}{2}} \|x\|_{H^1}.
	\end{align}
	This implies, 
	\begin{align*}
		\mathbb{E}\big[ \|(S_\nu(h)-I_d){\mathfrak{u}}^{\varepsilon, \nu}(t)\|^{2p} \big]\leq C h^p \mathbb{E}\big[\|{\mathfrak{u}}^{\varepsilon, \nu}(t)\|_{H^1}^{2p}\big]\leq C h^p.
	\end{align*}
	Using the boundedness of the semigroup, the H\"older inequality and Sobolev inequality, we estimate 
	\begin{align*}
		&\mathbb{E}\bigg[\Big\|
		\int_t^{t+h} S_\nu(t+h-s)\mathcal{N}({\mathfrak{u}}^{\varepsilon, \nu}(s))\,ds
		\Big\|^{2p}\bigg]
		\le \mathbb{E}\bigg[
		\int_t^{t+h}\|S_\nu(t+h-s)\mathcal{N}({\mathfrak{u}}^{\varepsilon, \nu}(s))\|\,ds
		\bigg]^{2p} \\
		&\le \mathbb{E}\bigg[
		\int_t^{t+h}\|\mathcal{N}({\mathfrak{u}}^{\varepsilon, \nu}(s))\|\,ds
	\bigg]^{2p}\le \mathbb{E}\bigg[
		\bigg(\int_t^{t+h}1\,ds\bigg)^{2p-1}
		\int_t^{t+h}\|\mathcal{N}({\mathfrak{u}}^{\varepsilon, \nu}(s))\|^{2p}\,ds
		\bigg] \\
		&\le C h^{2p-1}\,
		\mathbb{E}\int_t^{t+h}\|{\mathfrak{u}}^{\varepsilon, \nu}(s)\|_{L^{2\beta}}^{2\beta p}\,ds\le C h^{2p-1}\,
		\mathbb{E}\int_t^{t+h}\|{\mathfrak{u}}^{\varepsilon, \nu}(s)\|_{H^1}^{2\beta p}\,ds=C h^{2p},
	\end{align*}
	and 
	\begin{align*}
		&\mathbb{E}\bigg[\Big\|
		\int_t^{t+h} S_\nu(t+h-s)F({\mathfrak{u}}^{\varepsilon, \nu}(s),{\mathfrak{v}}^{\varepsilon, \nu}(s))\,ds
		\Big\|^{2p}\bigg]
		\le
		\mathbb{E}\bigg[
		\int_t^{t+h}
		\|S_\nu(t+h-s)F({\mathfrak{u}}^{\varepsilon, \nu}(s),{\mathfrak{v}}^{\varepsilon, \nu}(s))\|\,ds
		\bigg]^{2p} \\
		&\le
		\mathbb{E}\bigg[
		\int_t^{t+h}
		\|F({\mathfrak{u}}^{\varepsilon, \nu}(s),{\mathfrak{v}}^{\varepsilon, \nu}(s))\|\,ds
		\bigg]^{2p} \le
		\mathbb{E}\bigg[
		\bigg(\int_t^{t+h} 1\,ds\bigg)^{2p-1}
		\int_t^{t+h}
		\|F({\mathfrak{u}}^{\varepsilon, \nu}(s),{\mathfrak{v}}^{\varepsilon, \nu}(s))\|^{2p}\,ds
	\bigg] \\
		&\le
		C h^{2p-1}
		\mathbb{E}\bigg[\int_t^{t+h}
		\bigl(1+\|{\mathfrak{u}}^{\varepsilon, \nu}(s)\|^{2p}
		+\|{\mathfrak{v}}^{\varepsilon, \nu}(s)\|^{2p}\bigr)\,ds\bigg] \le
		C h^{2p}.
	\end{align*}
	Using the Burkholder-Davis-Gundy inequality and the H\"older inequality, we obtain
	\begin{align*}
		&\mathbb{E}\bigg[\Big\|
		\int_t^{t+h} S_\nu(t+h-s)\Sigma_1({\mathfrak{u}}^{\varepsilon, \nu}(s))\,d\mathcal{W}_1(s)
		\Big\|^{2p}\bigg]\leq	\mathbb{E} \bigg[\sup_{0 \le t \le T}\Big\|
		\int_t^{t+h} S_\nu(t+h-s)\Sigma_1({\mathfrak{u}}^{\varepsilon, \nu}(s))\,d\mathcal{W}_1(s)
		\Big\|^{2p}\bigg]\\
		&\le	C\,\mathbb{E}\bigg[
		\int_t^{t+h}\|\Sigma_1({\mathfrak{u}}^{\varepsilon, \nu}(s))\|^2\,ds
		\bigg]^p\le	C\,\mathbb{E}\bigg[
		\int_t^{t+h}L_{\Sigma_1}^2 (1+\|{\mathfrak{u}}^{\varepsilon, \nu}(s)\|^2)\,ds
	\bigg]^p\\
		&\le C h^p +C\,\mathbb{E}\bigg[\int_t^{t+h}\|{\mathfrak{u}}^{\varepsilon, \nu}(s)\|^2\,ds
		\bigg]^p\le C h^p +C \mathbb{E} \bigg[\int_t^{t+h} h^{p-1} \|{\mathfrak{u}}^{\varepsilon, \nu}(s)\|^{2p}ds\bigg]\\
		&\le C h^p +C  h^{p-1}  \int_t^{t+h}\mathbb{E}\big[\|{\mathfrak{u}}^{\varepsilon, \nu}(s)\|^{2p}ds\big]\le 	C h^{p}.
	\end{align*}
	Therefore, we have 
	\begin{align*}
		\mathbb{E}\Big[\bigl\|{\mathfrak{u}}^{\varepsilon, \nu}(t+h)-{\mathfrak{u}}^{\varepsilon, \nu}(t)\bigr\|^{2p}\Big]
		\le C h^{p}.
	\end{align*}
	This completes the proof of \Cref{holderctsintime}.
\end{proof}
\subsection{Definition of the auxiliary process $({\hat{\mathfrak{u}}}^{\varepsilon,\nu}, {\hat{\mathfrak{v}}}^{\varepsilon,\nu})$}\label{subsecdefofaux}
We use the Khasminskii's time discretization approach proposed in \cite{MR260052} to estimate an auxiliary process. We fix a positive number $\delta$ and do a partition of time interval $[0,T]$ of size $\delta$. Next, we define an auxiliary process $ {\hat{\mathfrak{v}}}^{\varepsilon,\nu}$ as
\begin{align*}
	\hat {\mathfrak{v}}^{\varepsilon, \nu}(t)
	&=: {\mathfrak{v}}^{\varepsilon, \nu}(k\delta)
	+ \frac{1}{\varepsilon}\int_{k\delta}^t
	\bigl((1+i)(-\Delta)^{\rho}\hat {\mathfrak{v}}^{\varepsilon, \nu}(s)
	+\mathcal{N}(\hat {\mathfrak{v}}^{\varepsilon, \nu}(s))
	-\lambda\hat {\mathfrak{v}}^{\varepsilon, \nu}(s)
	+G({\mathfrak{u}}^{\varepsilon, \nu}(k\delta),
	\hat {\mathfrak{v}}^{\varepsilon, \nu}(s))\bigr)\,ds \\
	&\quad
	+ \frac{1}{\sqrt{\varepsilon}}
	\int_{k\delta}^t
	\Sigma_2({\mathfrak{u}}^{\varepsilon, \nu}(k\delta),
	\hat {\mathfrak{v}}^{\varepsilon, \nu}(s))\,d\mathcal{W}_2(s),\quad \mathbb{P} \text{-a.s.},
\end{align*}
for $t\in[k\delta,\min\{(k+1)\delta,T\})$, $k\ge 0$.
We additionally define the process $\hat {\mathfrak{u}}^{\varepsilon, \nu}$ as
\begin{align*}
	\hat {\mathfrak{u}}^{\varepsilon, \nu}(t)
	=: {\mathfrak{u}}_0
	+ \int_0^t
	\Bigl[A_\nu\hat {\mathfrak{u}}^{\varepsilon, \nu}(s)
	+\mathcal{N}(\hat {\mathfrak{u}}^{\varepsilon, \nu}(t_s))
	+F(\hat {\mathfrak{u}}^{\varepsilon, \nu}(t_s),
	\hat {\mathfrak{v}}^{\varepsilon, \nu}(s))\Bigr]\,ds
	+ \int_0^t \Sigma_1(\hat {\mathfrak{u}}^{\varepsilon, \nu}(s))\,d\mathcal{W}_1(s),
\end{align*}
$\mathbb{P}$-a.s., for $t\in [0,T]$, where $t_s:= [\frac{s}{\delta}]\delta$ denotes the nearest break point preceding 
$s$. Therefore, $({\hat{\mathfrak{u}}}^{\varepsilon,\nu}, {\hat{\mathfrak{v}}}^{\varepsilon,\nu})$ satisfies the system
\begin{align}\label{auxiliary}
	\left\{
	\begin{aligned}
		&d\hat {\mathfrak{u}}^{\varepsilon, \nu}
		=\bigl[A_\nu \hat {\mathfrak{u}}^{\varepsilon, \nu}
		+\mathcal{N}(\hat {\mathfrak{u}}^{\varepsilon, \nu}(t_s))
		+F(\hat {\mathfrak{u}}^{\varepsilon, \nu}(t_s),\hat {\mathfrak{v}}^{\varepsilon, \nu})\bigr]dt
		+\Sigma_1(\hat {\mathfrak{u}}^{\varepsilon, \nu})\,d\mathcal{W}_1,
		&& \text{in }\mathbb{T}\times[0,T],\\[2mm]
		&d\hat {\mathfrak{v}}^{\varepsilon, \nu}
		=\frac{1}{\varepsilon}\bigl[(1+i)(-\Delta)^{\rho}\hat {\mathfrak{v}}^{\varepsilon, \nu}
		+ \mathcal{N}(\hat {\mathfrak{v}}^{\varepsilon, \nu})
		-\lambda \hat {\mathfrak{v}}^{\varepsilon, \nu}
		+G({\mathfrak{u}}^{\varepsilon, \nu}(k\delta),\hat {\mathfrak{v}}^{\varepsilon, \nu})\bigr]dt \\
		&\qquad
		+\frac{1}{\sqrt{\varepsilon}}
		\Sigma_2({\mathfrak{u}}^{\varepsilon, \nu}(k\delta),\hat {\mathfrak{v}}^{\varepsilon, \nu})\,d\mathcal{W}_2,
		&& \text{in }\mathbb{T}\times[k\delta,\min\{(k+1)\delta,T\}),\\[2mm]
		&\hat {\mathfrak{u}}^{\varepsilon, \nu}(x,0)={\mathfrak{u}}_0(x), && \text{in }\mathbb{T},\\
		&\hat {\mathfrak{v}}^{\varepsilon, \nu}(x,k\delta)={\mathfrak{v}}^{\varepsilon, \nu}(x,k\delta),
		&& \text{in }\mathbb{T}.
	\end{aligned}
	\right.
\end{align}
In similar way as in the proof of \Cref{uniformestimates}, one can easily get the following proposition. Since the arguments are essentially similar, we only state the proposition and omit the proof.
\begin{prop}
	Let $({\mathfrak{u}}_0,{\mathfrak{v}}_0)\in H^1(\mathbb{T})\times H^1(\mathbb{T})$ and $\varepsilon\in(0,1)$. Let $(\hat {\mathfrak{u}}^{\varepsilon, \nu},\hat {\mathfrak{v}}^{\varepsilon, \nu})$ be the solution of \eqref{auxiliary}, then for any $p>0$, there exists a positive constant $C=C(p,T,{\mathfrak{u}}_0,{\mathfrak{v}}_0)$ independent of $\varepsilon$ such that the solution $(\hat {\mathfrak{u}}^{\varepsilon, \nu},\hat {\mathfrak{v}}^{\varepsilon, \nu})$ satisfies the estimates
	\begin{align*}
		&\sup_{\varepsilon\in(0,1)}\sup_{0\le t\le T}
		\mathbb{E}\Big[\|\hat {\mathfrak{u}}^{\varepsilon, \nu}(t)\|^{2p}\Big]\le C,\\
		&	\sup_{\varepsilon\in(0,1)}\sup_{0\le t\le T}
		\mathbb{E}\Big[\|\hat {\mathfrak{v}}^{\varepsilon, \nu}(t)\|^{2p}\Big]\le C,\\
		&\sup_{\varepsilon\in(0,1)}
		\mathbb{E}\bigg[\sup_{0\le t\le T}
		\|\hat {\mathfrak{v}}^{\varepsilon, \nu}(t)\|^{2p}\bigg]\le C.
	\end{align*}
\end{prop}
\subsection{Error estimates}\label{subsecerroraux}
In this subsection, we establish the estimates for the error of auxiliary processes and the solution of \eqref{viscous}, namely ${\mathfrak{u}}^{\varepsilon, \nu}-\hat {\mathfrak{u}}^{\varepsilon, \nu}$ and ${\mathfrak{v}}^{\varepsilon, \nu}-\hat {\mathfrak{v}}^{\varepsilon, \nu}$.
\begin{prop}\label{uuhatvvhaterror}
	There exists a positive constant $C$ such that the following estimates hold:
	\begin{align*}
		&\sup_{0\le t\le T}
		\mathbb{E}\Big[\|{\mathfrak{v}}^{\varepsilon, \nu}(t)-\hat {\mathfrak{v}}^{\varepsilon, \nu}(t)\|^{2p}\Big]
		\le C\,\delta^{p+1}\,\varepsilon\,e^{C\frac{\delta}{\varepsilon}},\\
		&\mathbb{E}\bigg[\sup_{0\le t\le T}
		\|{\mathfrak{u}}^{\varepsilon, \nu}(t)-\hat {\mathfrak{u}}^{\varepsilon, \nu}(t)\|^{2p}\bigg]
		\le C\left(\delta^{p}
		+\frac{\delta^{p+1}}{\varepsilon}\,e^{C\frac{\delta}{\varepsilon}}\right).
	\end{align*}
\end{prop}
\begin{proof}
	From \eqref{viscous} and \eqref{auxiliary}, we have
	\begin{align*}
		\left\{
		\begin{aligned}
			d{\mathfrak{v}}^{\varepsilon, \nu}(t)
			&=\frac{1}{\varepsilon}\bigl[(1+i)(-\Delta)^{\rho} {\mathfrak{v}}^{\varepsilon, \nu}
			+\mathcal{N}({\mathfrak{v}}^{\varepsilon, \nu})
			-\lambda {\mathfrak{v}}^{\varepsilon, \nu}
			+G({\mathfrak{u}}^{\varepsilon, \nu},{\mathfrak{v}}^{\varepsilon, \nu})\bigr]dt
			+\frac{1}{\sqrt{\varepsilon}}
			\Sigma_2({\mathfrak{u}}^{\varepsilon, \nu},{\mathfrak{v}}^{\varepsilon, \nu})\,d\mathcal{W}_2,\\
			d\hat {\mathfrak{v}}^{\varepsilon, \nu}(t)
			&=\frac{1}{\varepsilon}\bigl[(1+i)(-\Delta)^{\rho} \hat {\mathfrak{v}}^{\varepsilon, \nu}
			+\mathcal{N}(\hat {\mathfrak{v}}^{\varepsilon, \nu})
			-\lambda \hat {\mathfrak{v}}^{\varepsilon, \nu}
			+G({\mathfrak{u}}^{\varepsilon, \nu}(k\delta),\hat {\mathfrak{v}}^{\varepsilon, \nu})\bigr]dt
			+\frac{1}{\sqrt{\varepsilon}}
			\Sigma_2({\mathfrak{u}}^{\varepsilon, \nu}(k\delta),\hat {\mathfrak{v}}^{\varepsilon, \nu})\,d\mathcal{W}_2,
		\end{aligned}
		\right.
	\end{align*}
	for a.e. $t\in [k\delta, (k+1)\delta)$.
	Therefore, ${\mathfrak{v}}^{\varepsilon, \nu}-\hat {\mathfrak{v}}^{\varepsilon, \nu}$ satisfies the following equation
	\begin{align}\label{spdeofvdifference}
		d({\mathfrak{v}}^{\varepsilon, \nu}-\hat {\mathfrak{v}}^{\varepsilon, \nu})
		&=\frac{1}{\varepsilon}\Bigl[(1+i)(-\Delta)^{\rho}
		({\mathfrak{v}}^{\varepsilon, \nu}-\hat {\mathfrak{v}}^{\varepsilon, \nu})
		+\mathcal{N}({\mathfrak{v}}^{\varepsilon, \nu})-\mathcal{N}(\hat {\mathfrak{v}}^{\varepsilon, \nu})
		-\lambda({\mathfrak{v}}^{\varepsilon, \nu}-\hat {\mathfrak{v}}^{\varepsilon, \nu})+G({\mathfrak{u}}^{\varepsilon, \nu},{\mathfrak{v}}^{\varepsilon, \nu}) \nonumber\\
		&\qquad			-G({\mathfrak{u}}^{\varepsilon, \nu}(k\delta),\hat {\mathfrak{v}}^{\varepsilon, \nu})\Bigr]dt 
		+\frac{1}{\sqrt{\varepsilon}}
		\bigl[\Sigma_2({\mathfrak{u}}^{\varepsilon, \nu},{\mathfrak{v}}^{\varepsilon, \nu})
		-\Sigma_2({\mathfrak{u}}^{\varepsilon, \nu}(k\delta),\hat {\mathfrak{v}}^{\varepsilon, \nu})\bigr]d\mathcal{W}_2,    \quad \mathbb{P} \text{-a.s.}
	\end{align}
	For $t\in [0,T]$ with $t\in [k\delta, (k+1)\delta)$, an application of the infinite-dimensional It\^o formula to the process $\|({\mathfrak{v}}^{\varepsilon, \nu}-\hat {\mathfrak{v}}^{\varepsilon, \nu})(t)\|^{2p}$ yields
	\begin{align*}
		&\|({\mathfrak{v}}^{\varepsilon, \nu}-\hat {\mathfrak{v}}^{\varepsilon, \nu})(t)\|^{2p}\\
		&=	\|{\mathfrak{v}}^{\varepsilon, \nu}(k\delta)-\hat {\mathfrak{v}}^{\varepsilon, \nu}(k\delta)\|^{2p}+\frac{2p}{\varepsilon}\int_{k\delta}^{t}
		\|({\mathfrak{v}}^{\varepsilon, \nu}-\hat {\mathfrak{v}}^{\varepsilon, \nu})(s)\|^{2p-2}
		\Big( ({\mathfrak{v}}^{\varepsilon, \nu}-\hat {\mathfrak{v}}^{\varepsilon, \nu})(s),
		(1+i)(-\Delta)^{\rho}({\mathfrak{v}}^{\varepsilon, \nu}-\hat {\mathfrak{v}}^{\varepsilon, \nu})(s) \\
		&\qquad +\mathcal{N}({\mathfrak{v}}^{\varepsilon, \nu}(s))-\mathcal{N}(\hat {\mathfrak{v}}^{\varepsilon, \nu}(s))
		-\lambda\big({\mathfrak{v}}^{\varepsilon, \nu}-\hat {\mathfrak{v}}^{\varepsilon, \nu}\big)(s)
		+G({\mathfrak{u}}^{\varepsilon, \nu}(s),{\mathfrak{v}}^{\varepsilon, \nu}(s))
		-G({\mathfrak{u}}^{\varepsilon, \nu}(k\delta),\hat {\mathfrak{v}}^{\varepsilon, \nu}(s))
		\Big) ds \\
		&\quad +\frac{p}{\varepsilon}\int_{k\delta}^{t}
		\|({\mathfrak{v}}^{\varepsilon, \nu}-\hat {\mathfrak{v}}^{\varepsilon, \nu})(s)\|^{2p-2}
		\|\Sigma_2({\mathfrak{u}}^{\varepsilon, \nu}(s),{\mathfrak{v}}^{\varepsilon, \nu}(s))
		-\Sigma_2({\mathfrak{u}}^{\varepsilon, \nu}(k\delta),\hat {\mathfrak{v}}^{\varepsilon, \nu}(s))\|^2 ds \\
		&\quad +\frac{2p(p-1)}{\varepsilon}\int_{k\delta}^{t}
		\|({\mathfrak{v}}^{\varepsilon, \nu}-\hat {\mathfrak{v}}^{\varepsilon, \nu})(s)\|^{2p-4}
		\big( \Sigma_2({\mathfrak{u}}^{\varepsilon, \nu}(s),{\mathfrak{v}}^{\varepsilon, \nu}(s))
		-\Sigma_2({\mathfrak{u}}^{\varepsilon, \nu}(k\delta),\hat {\mathfrak{v}}^{\varepsilon, \nu}(s)),
		({\mathfrak{v}}^{\varepsilon, \nu}-\hat {\mathfrak{v}}^{\varepsilon, \nu})(s)\big)^2 ds\\
		& \quad +\frac{2p}{\sqrt{\varepsilon}}\int_{k\delta}^{t}
		\|({\mathfrak{v}}^{\varepsilon, \nu}-\hat {\mathfrak{v}}^{\varepsilon, \nu})(s)\|^{2p-2}
		\big( ({\mathfrak{v}}^{\varepsilon, \nu}-\hat {\mathfrak{v}}^{\varepsilon, \nu})(s),\Sigma_2({\mathfrak{u}}^{\varepsilon, \nu}(s),{\mathfrak{v}}^{\varepsilon, \nu}(s))
		-\Sigma_2({\mathfrak{u}}^{\varepsilon, \nu}(k\delta),\hat {\mathfrak{v}}^{\varepsilon, \nu}(s))
		\big) d\mathcal{W}_2(s).
	\end{align*}
	We take mathematical expectation on both sides and perform an integration by parts to get
	\begin{align*}
		&\mathbb{E}\big[\|({\mathfrak{v}}^{\varepsilon, \nu}-\hat {\mathfrak{v}}^{\varepsilon, \nu})(t)\|^{2p}\big]\\
		&=\frac{2p}{\varepsilon}\mathbb{E}\bigg[\int_{k\delta}^{t}
		\|({\mathfrak{v}}^{\varepsilon, \nu}-\hat {\mathfrak{v}}^{\varepsilon, \nu})(s)\|^{2p-2}
		\Big( ({\mathfrak{v}}^{\varepsilon, \nu}-\hat {\mathfrak{v}}^{\varepsilon, \nu})(s),
		(1+i)(-\Delta)^{\rho}({\mathfrak{v}}^{\varepsilon, \nu}-\hat {\mathfrak{v}}^{\varepsilon, \nu})(s) \\
		&\qquad+\mathcal{N}({\mathfrak{v}}^{\varepsilon, \nu}(s))-\mathcal{N}(\hat {\mathfrak{v}}^{\varepsilon, \nu}(s))
		-\lambda\big({\mathfrak{v}}^{\varepsilon, \nu}-\hat {\mathfrak{v}}^{\varepsilon, \nu}\big)(s)
		+G({\mathfrak{u}}^{\varepsilon, \nu}(s),{\mathfrak{v}}^{\varepsilon, \nu}(s))
		-G({\mathfrak{u}}^{\varepsilon, \nu}(k\delta),\hat {\mathfrak{v}}^{\varepsilon, \nu}(s))
		\Big)ds\bigg] \\
		&\quad +\frac{p}{\varepsilon}\mathbb{E}\bigg[\int_{k\delta}^{t}
		\|({\mathfrak{v}}^{\varepsilon, \nu}-\hat {\mathfrak{v}}^{\varepsilon, \nu})(s)\|^{2p-2}
		\|\Sigma_2({\mathfrak{u}}^{\varepsilon, \nu}(s),{\mathfrak{v}}^{\varepsilon, \nu}(s))
		-\Sigma_2({\mathfrak{u}}^{\varepsilon, \nu}(k\delta),\hat {\mathfrak{v}}^{\varepsilon, \nu}(s))\|^2 ds \bigg]\\
		&\quad +\frac{2p(p-1)}{\varepsilon}\mathbb{E}\bigg[\int_{k\delta}^{t}
		\|({\mathfrak{v}}^{\varepsilon, \nu}-\hat {\mathfrak{v}}^{\varepsilon, \nu})(s)\|^{2p-4}
		\big( \Sigma_2({\mathfrak{u}}^{\varepsilon, \nu}(s),{\mathfrak{v}}^{\varepsilon, \nu}(s))\\
		&\quad\quad-\Sigma_2({\mathfrak{u}}^{\varepsilon, \nu}(k\delta),\hat {\mathfrak{v}}^{\varepsilon, \nu}(s)),
		({\mathfrak{v}}^{\varepsilon, \nu}-\hat {\mathfrak{v}}^{\varepsilon, \nu})(s)\big)^2 ds\bigg]\\
		&   =\frac{2p}{\varepsilon}\mathbb{E}\bigg[\int_{k\delta}^{t}
		\|({\mathfrak{v}}^{\varepsilon, \nu}-\hat {\mathfrak{v}}^{\varepsilon, \nu})(s)\|^{2p-2}\big(-\|(-\Delta)^{\frac{\rho}{2}}({\mathfrak{v}}^{\varepsilon, \nu}-\hat {\mathfrak{v}}^{\varepsilon, \nu})(s)\|^2-\lambda\|({\mathfrak{v}}^{\varepsilon, \nu}-\hat {\mathfrak{v}}^{\varepsilon, \nu})(s)\|^2\big)\bigg]\\
		&\quad+\frac{2p}{\varepsilon}\mathbb{E}\bigg[\int_{k\delta}^{t}
		\|({\mathfrak{v}}^{\varepsilon, \nu}-\hat {\mathfrak{v}}^{\varepsilon, \nu})(s)\|^{2p-2}
		\Big( ({\mathfrak{v}}^{\varepsilon, \nu}-\hat {\mathfrak{v}}^{\varepsilon, \nu})(s),\mathcal{N}({\mathfrak{v}}^{\varepsilon, \nu}(s))-\mathcal{N}(\hat {\mathfrak{v}}^{\varepsilon, \nu}(s))
		+G({\mathfrak{u}}^{\varepsilon, \nu}(s),{\mathfrak{v}}^{\varepsilon, \nu}(s))\\
		&\quad\quad -G({\mathfrak{u}}^{\varepsilon, \nu}(k\delta),\hat {\mathfrak{v}}^{\varepsilon, \nu}(s))
		\Big)ds \bigg]\\
		&\quad +\frac{p}{\varepsilon}\mathbb{E}\bigg[\int_{k\delta}^{t}
		\|({\mathfrak{v}}^{\varepsilon, \nu}-\hat {\mathfrak{v}}^{\varepsilon, \nu})(s)\|^{2p-2}
		\|\Sigma_2({\mathfrak{u}}^{\varepsilon, \nu}(s),{\mathfrak{v}}^{\varepsilon, \nu}(s))
		-\Sigma_2({\mathfrak{u}}^{\varepsilon, \nu}(k\delta),\hat {\mathfrak{v}}^{\varepsilon, \nu}(s))\|^2 ds \bigg]\\
		&\quad +\frac{2p(p-1)}{\varepsilon}\mathbb{E}\bigg[\int_{k\delta}^{t}
		\|({\mathfrak{v}}^{\varepsilon, \nu}-\hat {\mathfrak{v}}^{\varepsilon, \nu})(s)\|^{2p-4}
		\big( \Sigma_2({\mathfrak{u}}^{\varepsilon, \nu}(s),{\mathfrak{v}}^{\varepsilon, \nu}(s))\\
		&\quad\quad-\Sigma_2({\mathfrak{u}}^{\varepsilon, \nu}(k\delta),\hat {\mathfrak{v}}^{\varepsilon, \nu}(s)),
		({\mathfrak{v}}^{\varepsilon, \nu}-\hat {\mathfrak{v}}^{\varepsilon, \nu})(s)\big)^2 ds\bigg],
	\end{align*}
	for all $t\in [0,T]$ with $t\in [k\delta, (k+1)\delta)$. Differentiating with respect to time `$t$', we obtain
	\begin{align*}
		&\frac{d}{dt}\mathbb{E}\Big[\|({\mathfrak{v}}^{\varepsilon, \nu}-\hat {\mathfrak{v}}^{\varepsilon, \nu})(t)\|^{2p}\Big]\\
		&=-\frac{2p}{\varepsilon}\mathbb{E}\Big[
		\|({\mathfrak{v}}^{\varepsilon, \nu}-\hat {\mathfrak{v}}^{\varepsilon, \nu})(t)\|^{2p-2}
		\|(-\Delta)^{\frac{\rho}{2}}({\mathfrak{v}}^{\varepsilon, \nu}-\hat {\mathfrak{v}}^{\varepsilon, \nu})(t)\|^2\Big] -\frac{2p\lambda}{\varepsilon}\mathbb{E}
	\Big[	\|({\mathfrak{v}}^{\varepsilon, \nu}-\hat {\mathfrak{v}}^{\varepsilon, \nu})(t)\|^{2p}\Big]\\
		&\quad +\frac{2p}{\varepsilon}\mathbb{E}\Big[
		\|({\mathfrak{v}}^{\varepsilon, \nu}-\hat {\mathfrak{v}}^{\varepsilon, \nu})(t)\|^{2p-2}
		\big(\mathcal{N}({\mathfrak{v}}^{\varepsilon, \nu}(t))-\mathcal{N}(\hat {\mathfrak{v}}^{\varepsilon, \nu}(t)),
		({\mathfrak{v}}^{\varepsilon, \nu}-\hat {\mathfrak{v}}^{\varepsilon, \nu})(t)\big)\Big] \\
		&\quad +\frac{2p}{\varepsilon}\mathbb{E}\Big[
		\|({\mathfrak{v}}^{\varepsilon, \nu}-\hat {\mathfrak{v}}^{\varepsilon, \nu})(t)\|^{2p-2}
		\big(({\mathfrak{v}}^{\varepsilon, \nu}-\hat {\mathfrak{v}}^{\varepsilon, \nu})(t),
		G({\mathfrak{u}}^{\varepsilon, \nu}(t),{\mathfrak{v}}^{\varepsilon, \nu}(t))
		-G({\mathfrak{u}}^{\varepsilon, \nu}(k\delta),\hat {\mathfrak{v}}^{\varepsilon, \nu}(t))\big)\Big] \\
		&\quad +\frac{p}{\varepsilon}\mathbb{E}\Big[
		\|({\mathfrak{v}}^{\varepsilon, \nu}-\hat {\mathfrak{v}}^{\varepsilon, \nu})(t)\|^{2p-2}
		\|\Sigma_2({\mathfrak{u}}^{\varepsilon, \nu}(t),{\mathfrak{v}}^{\varepsilon, \nu}(t))
		-\Sigma_2({\mathfrak{u}}^{\varepsilon, \nu}(k\delta),\hat {\mathfrak{v}}^{\varepsilon, \nu}(t))\|^2\Big] \\
		&\quad +\frac{2p(p-1)}{\varepsilon}\mathbb{E}\Big[
		\|({\mathfrak{v}}^{\varepsilon, \nu}-\hat {\mathfrak{v}}^{\varepsilon, \nu})(t)\|^{2p-4}
		\big( \Sigma_2({\mathfrak{u}}^{\varepsilon, \nu}(t),{\mathfrak{v}}^{\varepsilon, \nu}(t))
		-\Sigma_2({\mathfrak{u}}^{\varepsilon, \nu}(k\delta),\hat {\mathfrak{v}}^{\varepsilon, \nu}(t)),
		({\mathfrak{v}}^{\varepsilon, \nu}-\hat {\mathfrak{v}}^{\varepsilon, \nu})(t)\big)^2\Big],
	\end{align*}
	and a consequence of \Cref{diffnonrenega} implies
	\begin{align*}
		&\frac{d}{dt}\mathbb{E}\Big[\|({\mathfrak{v}}^{\varepsilon, \nu}-\hat {\mathfrak{v}}^{\varepsilon, \nu})(t)\|^{2p}\Big]\\
		&      \leq -\frac{2p\lambda}{\varepsilon}\mathbb{E}\Big[
		\|({\mathfrak{v}}^{\varepsilon, \nu}-\hat {\mathfrak{v}}^{\varepsilon, \nu})(t)\|^{2p}\Big]+\frac{2p}{\varepsilon}\mathbb{E}\Big[
		\|({\mathfrak{v}}^{\varepsilon, \nu}-\hat {\mathfrak{v}}^{\varepsilon, \nu})(t)\|^{2p-1} \|G({\mathfrak{u}}^{\varepsilon, \nu}(t),{\mathfrak{v}}^{\varepsilon, \nu}(t))
		-G({\mathfrak{u}}^{\varepsilon, \nu}(k\delta),\hat {\mathfrak{v}}^{\varepsilon, \nu}(t))\| \Big]\\
		&\quad +\frac{p}{\varepsilon}\mathbb{E}\Big[
		\|({\mathfrak{v}}^{\varepsilon, \nu}-\hat {\mathfrak{v}}^{\varepsilon, \nu})(t)\|^{2p-2}
		\|\Sigma_2({\mathfrak{u}}^{\varepsilon, \nu}(t),{\mathfrak{v}}^{\varepsilon, \nu}(t))
		-\Sigma_2({\mathfrak{u}}^{\varepsilon, \nu}(k\delta),\hat {\mathfrak{v}}^{\varepsilon, \nu}(t))\|^2 \Big]\\
		&\quad +\frac{2p(p-1)}{\varepsilon}\mathbb{E}\Big[
		\|({\mathfrak{v}}^{\varepsilon, \nu}-\hat {\mathfrak{v}}^{\varepsilon, \nu})(t)\|^{2p-2}	\|\Sigma_2({\mathfrak{u}}^{\varepsilon, \nu}(t),{\mathfrak{v}}^{\varepsilon, \nu}(t))
		-\Sigma_2({\mathfrak{u}}^{\varepsilon, \nu}(k\delta),\hat {\mathfrak{v}}^{\varepsilon, \nu}(t))\|^2\Big],
	\end{align*}
	for a.e. $t\in [0,T]$ with $t\in [k\delta, (k+1)\delta)$. Now, for each term, we use the assumptions in \Cref{main assumptions} and the Young inequality with the condition $\frac{2p-1}{2p}+\frac{1}{2p}=1$ to estimate
	\begin{align*}
		&	\|{\mathfrak{v}}^{\varepsilon, \nu}-\hat {\mathfrak{v}}^{\varepsilon, \nu}\|^{2p-1} \|G({\mathfrak{u}}^{\varepsilon, \nu},{\mathfrak{v}}^{\varepsilon, \nu})
		-G({\mathfrak{u}}^{\varepsilon, \nu}(k\delta),\hat {\mathfrak{v}}^{\varepsilon, \nu})\| \\
		&\leq \|{\mathfrak{v}}^{\varepsilon, \nu}-\hat {\mathfrak{v}}^{\varepsilon, \nu}\|^{2p-1} L_G \big\{\|{\mathfrak{u}}^{\varepsilon, \nu}-{\mathfrak{u}}^{\varepsilon, \nu}(k\delta)\|+\|{\mathfrak{v}}^{\varepsilon, \nu}-\hat {\mathfrak{v}}^{\varepsilon, \nu}\| \big\}\\
		&\leq C\big\{\|{\mathfrak{u}}^{\varepsilon, \nu}-{\mathfrak{u}}^{\varepsilon, \nu}(k\delta)\|^{2p} +\|{\mathfrak{v}}^{\varepsilon, \nu}-\hat {\mathfrak{v}}^{\varepsilon, \nu}\|^{2p} \big\},
	\end{align*}
	and
	\begin{align*}
		\|{\mathfrak{v}}^{\varepsilon, \nu}-\hat {\mathfrak{v}}^{\varepsilon, \nu}\|^{2p-2}
		\|\Sigma_2({\mathfrak{u}}^{\varepsilon, \nu},{\mathfrak{v}}^{\varepsilon, \nu})
		-\Sigma_2({\mathfrak{u}}^{\varepsilon, \nu}(k\delta),\hat {\mathfrak{v}}^{\varepsilon, \nu})\|^2 &\leq \|{\mathfrak{v}}^{\varepsilon, \nu}-\hat {\mathfrak{v}}^{\varepsilon, \nu}\|^{2p-2}
		L_{\Sigma_2}^2 \|{\mathfrak{u}}^{\varepsilon, \nu}-{\mathfrak{u}}^{\varepsilon, \nu}(k\delta)\|^2\\
		&\leq C\big\{\|{\mathfrak{u}}^{\varepsilon, \nu}-{\mathfrak{u}}^{\varepsilon, \nu}(k\delta)\|^{2p} +\|{\mathfrak{v}}^{\varepsilon, \nu}-\hat {\mathfrak{v}}^{\varepsilon, \nu}\|^{2p} \big\}.
	\end{align*}
	Therefore, using the H\"older continuity in time (\Cref{holderctsintime}), we have 
	\begin{align*}
		\frac{d}{dt}\mathbb{E}\Big[\|({\mathfrak{v}}^{\varepsilon, \nu}-\hat {\mathfrak{v}}^{\varepsilon, \nu})(t)\|^{2p}\Big]
		&\leq \frac{C}{\varepsilon}\Big\{\mathbb{E}\Big[\|{\mathfrak{u}}^{\varepsilon, \nu}-{\mathfrak{u}}^{\varepsilon, \nu}(k\delta)\|^{2p} \Big]+\mathbb{E}\Big[\|({\mathfrak{v}}^{\varepsilon, \nu}-\hat {\mathfrak{v}}^{\varepsilon, \nu})(t)\|^{2p} \Big]\Big\}\\
		&\leq \frac{C}{\varepsilon}\Big\{\mathbb{E}\Big[\|({\mathfrak{v}}^{\varepsilon, \nu}-\hat {\mathfrak{v}}^{\varepsilon, \nu})(t)\|^{2p}\Big] +\delta^p\Big\},
	\end{align*}
	for a.e. $t\in [0,T]$ with $t\in [k\delta, (k+1)\delta)$. Using the variation of constants formula, we derive
	\begin{align*}
		\mathbb{E}\Big[\|({\mathfrak{v}}^{\varepsilon, \nu}-\hat {\mathfrak{v}}^{\varepsilon, \nu})(t)\|^{2p}\Big]\leq \frac{C}{\varepsilon} \delta^p \int_{k\delta}^{t} e^{\frac{C}{\varepsilon}(t-\tau)} d \tau \leq \frac{C}{\varepsilon} \delta^{p+1} e^{\frac{C}{\varepsilon}\delta},
	\end{align*}
	for all $t\in [0,T]$ with $t\in [k\delta, (k+1)\delta)$.
	Therefore, we conclude
\begin{align}\label{errorv}
	\sup_{0 \le t \le T}\mathbb{E}\Big[\|({\mathfrak{v}}^{\varepsilon, \nu}-\hat {\mathfrak{v}}^{\varepsilon, \nu})(t)\|^{2p}\Big]\leq \frac{C}{\varepsilon} \delta^{p+1} e^{\frac{C}{\varepsilon}\delta}.
\end{align}
Now, we estimate the error of ${\mathfrak{u}}^{\varepsilon, \nu}-\hat {\mathfrak{u}}^{\varepsilon, \nu}$. From \eqref{viscous} and \eqref{auxiliary}, we have 
	\begin{align*}
		\begin{cases}
			d {\mathfrak{u}}^{\varepsilon, \nu}
			= \big[A_\nu {\mathfrak{u}}^{\varepsilon, \nu}+\mathcal{N}({\mathfrak{u}}^{\varepsilon, \nu})
			+ F({\mathfrak{u}}^{\varepsilon, \nu},{\mathfrak{v}}^{\varepsilon, \nu})\big]\,dt
			+ \Sigma_1({\mathfrak{u}}^{\varepsilon, \nu})\,d\mathcal{W}_1, \\[1mm]
			d\hat {\mathfrak{u}}^{\varepsilon, \nu}
			= \big[A_\nu \hat {\mathfrak{u}}^{\varepsilon, \nu}
			+ \mathcal{N}({\mathfrak{u}}^{\varepsilon, \nu}(t_s))
			+ F({\mathfrak{u}}^{\varepsilon, \nu}(t_s),\hat {\mathfrak{v}}^{\varepsilon, \nu})\big]\,dt
			+ \Sigma_1({\mathfrak{u}}^{\varepsilon, \nu})\,d\mathcal{W}_1,
		\end{cases}
	\end{align*}
    where $t_s= [\frac{s}{\delta}]\delta$ denotes the nearest break point preceding 
     $s$. Therefore, ${\mathfrak{u}}^{\varepsilon, \nu}-\hat {\mathfrak{u}}^{\varepsilon, \nu}$ satisfies the equation
	\begin{align*}
		{\mathfrak{u}}^{\varepsilon, \nu}(t)-\hat {\mathfrak{u}}^{\varepsilon, \nu}(t)
		= \int_0^t S_\nu(t-s)
		\Big[ \mathcal{N}({\mathfrak{u}}^{\varepsilon, \nu}(s))-\mathcal{N}({\mathfrak{u}}^{\varepsilon, \nu}(t_s))
		+ F({\mathfrak{u}}^{\varepsilon, \nu}(s),{\mathfrak{v}}^{\varepsilon, \nu}(s))
		- F({\mathfrak{u}}^{\varepsilon, \nu}(t_s),\hat {\mathfrak{v}}^{\varepsilon, \nu}(s))
		\Big]\,ds .
	\end{align*}
	Using \Cref{liptypeofnonlinear}, boundedness of the semigroup, Sobolev embedding $H^1(\mathbb{T})\subset L^\infty(\mathbb{T}) $ and the H\"older inequality with $\frac{1}{4p}+\frac{1}{4p}+\frac{2p-1}{2p}=1$, we estimate
	\begin{align*}
		&\left\|
		\int_0^t S_\nu(t-s)
		\big[\mathcal{N}({\mathfrak{u}}^{\varepsilon, \nu}(s))-\mathcal{N}({\mathfrak{u}}^{\varepsilon, \nu}(t_s))\big]\,ds
		\right\|^{2p}\le C\left(\int_0^t
		\big\| S_\nu(t-s)
		\big[\mathcal{N}({\mathfrak{u}}^{\varepsilon, \nu}(s))-\mathcal{N}({\mathfrak{u}}^{\varepsilon, \nu}(t_s))\big]
		\big\|\,ds
		\right)^{2p} \\
		&\le C\left(
		\int_0^t
		\big\|\mathcal{N}({\mathfrak{u}}^{\varepsilon, \nu}(s))-\mathcal{N}({\mathfrak{u}}^{\varepsilon, \nu}(t_s))\big\|\,ds
		\right)^{2p} \\
		&\leq C\left(\int_0^t	\big\|\left(|{\mathfrak{u}}^{\varepsilon, \nu}(s)|^{\beta -1}+|{\mathfrak{u}}^{\varepsilon, \nu}(t_s)|^{\beta -1}\right)|{\mathfrak{u}}^{\varepsilon, \nu}(s)-{\mathfrak{u}}^{\varepsilon, \nu}(t_s)|\big\|\,ds
		\right)^{2p} \\
		&\leq C\left[\int_0^t \left(\|{\mathfrak{u}}^{\varepsilon, \nu}(s)\|_{L^\infty}^{\beta -1}+\|{\mathfrak{u}}^{\varepsilon, \nu}(t_s)\|_{L^\infty}^{\beta -1}\right)
		\big\|{\mathfrak{u}}^{\varepsilon, \nu}(s)-{\mathfrak{u}}^{\varepsilon, \nu}(t_s)\big\|\,ds
		\right]^{2p} \\
		&\leq C\left[\int_0^t \left(\|{\mathfrak{u}}^{\varepsilon, \nu}(s)\|_{H^1}^{\beta -1}+\|{\mathfrak{u}}^{\varepsilon, \nu}(t_s)\|_{H^1}^{\beta -1}\right)    \big\|{\mathfrak{u}}^{\varepsilon, \nu}(s)-{\mathfrak{u}}^{\varepsilon, \nu}(t_s)\big\|\,ds
		\right]^{2p} \\ 
		&\le C\left(\int_0^t 1\,ds\right)^{2p-1}
		\left(\int_0^t	\|{\mathfrak{u}}^{\varepsilon, \nu}(s)-{\mathfrak{u}}^{\varepsilon, \nu}(t_s)\|^{4p}\,ds
		\right)^{\frac12}	\left(\int_0^t
		\big(\|{\mathfrak{u}}^{\varepsilon, \nu}(s)\|_{H^1}^{\beta -1}
		+ \|{\mathfrak{u}}^{\varepsilon, \nu}(t_s)\|_{H^1}^{\beta -1}\big)^{4p}
		\,ds\right)^{\frac12}.
	\end{align*}
	Using the Jensen inequality, \Cref{holderctsintime} and \Cref{uniformestimates}, we obtain
	\begin{align}\label{liptypeoff}
		&\mathbb{E}\left[\sup_{0\le t\le T}
		\left\|
		\int_0^t S_\nu(t-s)
		\big[\mathcal{N}({\mathfrak{u}}^{\varepsilon, \nu}(s))-\mathcal{N}({\mathfrak{u}}^{\varepsilon, \nu}(t_s))\big]\,ds
		\right\|^{2p}\right]\leq C \mathbb{E}
		\left[\int_0^T
		\|\mathcal{N}({\mathfrak{u}}^{\varepsilon, \nu}(s))-\mathcal{N}({\mathfrak{u}}^{\varepsilon, \nu}(t_s))\|\,ds
		\right]^{2p} \nonumber \\
		&\le C \mathbb{E}
		\left[\left(\int_0^T
		\|{\mathfrak{u}}^{\varepsilon, \nu}(s)-{\mathfrak{u}}^{\varepsilon, \nu}(t_s)\|^{4p}\,ds
		\right)^{\frac12} 
		\left(\int_0^T
		\big(\|{\mathfrak{u}}^{\varepsilon, \nu}(s)\|_{H^1}^{\beta -1}
		+ \|{\mathfrak{u}}^{\varepsilon, \nu}(t_s)\|_{H^1}^{\beta -1}\big)^{4p}
		\,ds\right)^{\frac12}\right]\nonumber \\
		&\le C
		\left(\mathbb{E}\bigg[\int_0^T
		\|{\mathfrak{u}}^{\varepsilon, \nu}(s)-{\mathfrak{u}}^{\varepsilon, \nu}(t_s)\|^{4p}\,ds\bigg]
		\right)^{\frac12}
		\left(\mathbb{E}\bigg[\int_0^T
		\big(\|{\mathfrak{u}}^{\varepsilon, \nu}(s)\|_{H^1}^{\beta -1}
		+ \|{\mathfrak{u}}^{\varepsilon, \nu}(t_s)\|_{H^1}^{\beta -1}\big)^{4p}
		\,ds\bigg]\right)^{\frac12}\nonumber \\
		&\le C
		\left(\mathbb{E}\bigg[\int_0^T
		\|{\mathfrak{u}}^{\varepsilon, \nu}(s)-{\mathfrak{u}}^{\varepsilon, \nu}(t_s)\|^{4p}\,ds\bigg]
		\right)^{\frac12}
		\left(\int_0^T \mathbb{E}
			\Big[\|{\mathfrak{u}}^{\varepsilon, \nu}(s)\|_{H^1}^{2p}
			+ \|{\mathfrak{u}}^{\varepsilon, \nu}(t_s)\|_{H^1}^{2p}\Big]^{2(\beta -1)}
			\,ds\right)^{\frac12}\nonumber \\
		&\leq C \left(\int_0^T	\delta^{2p}\,ds	\right)^{\frac12}\leq C \delta^{p}.
	\end{align}
	Boundedness of the semigroup, the H\"older inequality in time with $\frac{1}{2p}+\frac{2p-1}{2p}=1$ and \eqref{errorv} imply
	\begin{align*}
		&\mathbb{E}\left[\sup_{0\le t\le T}
		\left\|
		\int_0^t S_\nu(t-s)
		\big[F({\mathfrak{u}}^{\varepsilon, \nu}(s),{\mathfrak{v}}^{\varepsilon, \nu}(s))
		- F({\mathfrak{u}}^{\varepsilon, \nu}(t_s),\hat {\mathfrak{v}}^{\varepsilon, \nu}(s))\big]\,ds
		\right\|^{2p}\right]\\
		&\le C\mathbb{E}\left[\sup_{0\le t\le T}
		\left(\int_0^t
		\|F({\mathfrak{u}}^{\varepsilon, \nu}(s),{\mathfrak{v}}^{\varepsilon, \nu}(s))
		- F({\mathfrak{u}}^{\varepsilon, \nu}(t_s),\hat {\mathfrak{v}}^{\varepsilon, \nu}(s))\|\,ds\right)^{2p}\right] \\
		&\le C\mathbb{E}\left[\sup_{0\le t\le T}
		\int_0^t
		\|F({\mathfrak{u}}^{\varepsilon, \nu}(s),{\mathfrak{v}}^{\varepsilon, \nu}(s))
		- F({\mathfrak{u}}^{\varepsilon, \nu}(t_s),\hat {\mathfrak{v}}^{\varepsilon, \nu}(s))\|^{2p}\,ds\right] \\
		&\le C\mathbb{E}\bigg[\int_0^T
		\big(\|{\mathfrak{u}}^{\varepsilon, \nu}(s)-{\mathfrak{u}}^{\varepsilon, \nu}(t_s)\|^{2p}
		+ \|{\mathfrak{v}}^{\varepsilon, \nu}(s)-\hat {\mathfrak{v}}^{\varepsilon, \nu}(s)\|^{2p}\big)\,ds\bigg] \\
		&\le C\left(
		\delta^{p}+\frac{\delta^{p+1}}{\varepsilon}e^{\frac{C}{\varepsilon}\delta}
		\right).
	\end{align*}
	Thus, we conclude 
	\begin{align}\label{uhatuerror}
		\mathbb{E}\left[\sup_{0\le t\le T}
		\|{\mathfrak{u}}^{\varepsilon, \nu}(t)-\hat {\mathfrak{u}}^{\varepsilon, \nu}(t)\|^{2p}\right]
		\le C\left(
		\delta^{p}
		+\frac{\delta^{p+1}}{\varepsilon}e^{\frac{C}{\varepsilon}\delta}
		\right),
	\end{align}
	which completes the proof.
\end{proof}
\subsection{The error of $\hat{\mathfrak{u}}^{\varepsilon,\nu}-\bar{\mathfrak{u}}^\nu$}\label{subsecerrorviscous}
We aim to prove the strong convergence of the auxiliary process $\hat{\mathfrak{u}}^{\varepsilon,\nu}$ to the averaged solution process $\bar{\mathfrak{u}}^\nu$. For that we establish the following proposition:
\begin{prop}\label{propuhatuerror}
	For any $p\ge1$, there exists a positive constant $C$ depending on $\nu,p,T$ such that
	\begin{align*}
		\mathbb{E}\left[
		\sup_{0\le t\le T}
		\|\hat {\mathfrak{u}}^{\varepsilon, \nu}(t)-\bar{\mathfrak{u}}^\nu(t)\|^{2p}
		\right]
		&\le C\Bigg(
		\sqrt{\frac{\delta}{\varepsilon}}
		+ \delta^{p}
		+ \delta^{2p-1}
		+ \frac{\delta^{p+1}}{\varepsilon}e^{\frac{C}{\varepsilon}\delta}+\left(\frac{\delta^{2p+1}}{\varepsilon}e^{\frac{C}{\varepsilon}\delta}\right)^{\frac{1}{2}}
		\Bigg) e^{Cn^{2p(\beta-1)}}
		+ \frac{C}{\sqrt{n}} .
	\end{align*}
\end{prop}
\begin{proof}
	The process $\hat{\mathfrak{u}}^{\varepsilon,\nu}-\bar{\mathfrak{u}}^\nu$ satisfies the mild form
	\begin{align}\label{j123}
		\hat {\mathfrak{u}}^{\varepsilon, \nu}(t)-\bar{\mathfrak{u}}^\nu(t)
		&= \int_0^t S_\nu(t-s)\Big[
		F({\mathfrak{u}}^{\varepsilon, \nu}(t_s),\hat {\mathfrak{v}}^{\varepsilon, \nu}(s))
		- \bar F(\bar{\mathfrak{u}}^\nu(s))
		\Big]\,ds \nonumber\\
		&\quad + \int_0^t S_\nu(t-s)\Big[
		\mathcal{N}({\mathfrak{u}}^{\varepsilon, \nu}(t_s)) - \mathcal{N}(\bar{\mathfrak{u}}^\nu(s))
		\Big]\,ds \nonumber\\
		&\quad + \int_0^t S_\nu(t-s)\Big[
		\Sigma_1({\mathfrak{u}}^{\varepsilon, \nu}(s)) - \Sigma_1(\bar{\mathfrak{u}}^\nu(s))
		\Big]\,d\mathcal{W}_1 \nonumber \\
		&=: J_1 + J_2 + J_3 .
	\end{align}
	We define the stopping time 
	\begin{align}\label{stoppingtime}
		\tau_n^{\varepsilon,\nu}
		:= \inf\Big\{ t>0 : \|\hat {\mathfrak{u}}^{\varepsilon, \nu}(t)\|_{H^1}
		+ \|\bar{\mathfrak{u}}^\nu(t)\|_{H^1} > n \Big\}.
	\end{align}
	Let $m_t=[\frac{t}{\delta}]$. We further decompose the term $J_1$ as
	\begin{align*}
		J_1	&= \int_0^t S_\nu(t-s)\Big[
		F({\mathfrak{u}}^{\varepsilon, \nu}(t_s),\hat {\mathfrak{v}}^{\varepsilon, \nu}(s))
		- \bar F(\bar{\mathfrak{u}}^\nu(s))
		\Big]\,ds \\
		&= \int_0^t S_\nu(t-s)\Big[
		F({\mathfrak{u}}^{\varepsilon, \nu}(t_s),\hat {\mathfrak{v}}^{\varepsilon, \nu}(s))
		- \bar F( {\mathfrak{u}}^{\varepsilon, \nu}(s))
		\Big]\,ds  + \int_0^t S_\nu(t-s)\Big[
		\bar F( {\mathfrak{u}}^{\varepsilon, \nu}(s))- \bar F(\hat {\mathfrak{u}}^{\varepsilon, \nu}(s))
		\Big]\,ds \\
		&\quad + \int_0^t S_\nu(t-s)\Big[
		\bar F(\hat {\mathfrak{u}}^{\varepsilon, \nu}(s)) - \bar F(\bar{\mathfrak{u}}^\nu(s))
		\Big]\,ds \\
		&= \sum_{k=0}^{m_t-1}
		\int_{k\delta}^{(k+1)\delta}
		S_\nu(t-s)\Big[
		F({\mathfrak{u}}^{\varepsilon, \nu}(k\delta),\hat {\mathfrak{v}}^{\varepsilon, \nu}(s))
		- \bar F({\mathfrak{u}}^{\varepsilon, \nu}(k\delta))
		\Big]\,ds \\
		&\quad + \sum_{k=0}^{m_t-1}
		\int_{k\delta}^{(k+1)\delta}
		S_\nu(t-s)\Big[
		\bar F({\mathfrak{u}}^{\varepsilon, \nu}(k\delta))
		- \bar F({\mathfrak{u}}^{\varepsilon, \nu}(s))
		\Big]\,ds \\
		&\quad + \int_{m\delta}^t
		S_\nu(t-s)\Big[
		F({\mathfrak{u}}^{\varepsilon, \nu}(m_t\delta),\hat {\mathfrak{v}}^{\varepsilon, \nu}(s))
		- \bar F({\mathfrak{u}}^{\varepsilon, \nu}(s))
		\Big]\,ds \\
		&\quad + \int_0^t S_\nu(t-s)\Big[
		\bar	F({\mathfrak{u}}^{\varepsilon, \nu}(s)) - \bar F(\hat {\mathfrak{u}}^{\varepsilon, \nu}(s))
		\Big]\,ds \\
		&\quad + \int_0^t S_\nu(t-s)\Big[
		\bar F(\hat {\mathfrak{u}}^{\varepsilon, \nu}(s)) - \bar F(\bar{\mathfrak{u}}^\nu(s))
		\Big]\,ds \\
		&=: J_{11} + J_{12} + J_{13} + J_{14} + J_{15}.
	\end{align*}
	Now, to estimate the term $J_{11}$, we first show that, for any fixed $p$ and $t\in [0,\delta)$, the processes  $\hat {\mathfrak{v}}^{\varepsilon, \nu}(t+p\delta)$ and  $\mathfrak{v}^{{\mathfrak{u}}^{\varepsilon, \nu}(p\delta),{\mathfrak{v}}^{\varepsilon, \nu}(p\delta)}(\frac{t}{\varepsilon})$ coincide in the sense of distribution, that is,
	\begin{align*}
		\mathscr{L}\Big(\big\{\hat {\mathfrak{v}}^{\varepsilon, \nu}(t+p\delta)\big\}_{0\leq t\leq\delta}\Big)=\mathscr{L}\left(\Big\{\mathfrak{v}^{{\mathfrak{u}}^{\varepsilon, \nu}(p\delta),{\mathfrak{v}}^{\varepsilon, \nu}(p\delta)}\Big(\frac{t}{\varepsilon}\Big)\Big\}_{0\leq t\leq\delta}\right),
	\end{align*}
	where $\mathscr{L}(\cdot)$ denotes the law of the distribution. Let us define
	\begin{align*}
		K({\mathfrak{u}},{\mathfrak{v}}) := F({\mathfrak{u}},{\mathfrak{v}}) - \lambda {\mathfrak{v}} + G({\mathfrak{u}},{\mathfrak{v}}).
	\end{align*}
	Applying a time shift transformation to $\hat {\mathfrak{v}}^{\varepsilon, \nu}$, we obtain
	\begin{align}\label{shiftedequation}
		\hat {\mathfrak{v}}^{\varepsilon, \nu}(t+p\delta)
		&= {\mathfrak{v}}^{\varepsilon, \nu}(p\delta)
		+ \frac{1}{\varepsilon}
		\int_{p\delta}^{t+p\delta}
		(1+i)(-\Delta)^{\rho}\hat {\mathfrak{v}}^{\varepsilon, \nu}(s)\,ds + \frac{1}{\varepsilon}
		\int_{p\delta}^{t+p\delta}
		K({\mathfrak{u}}^{\varepsilon, \nu}(p\delta),\hat {\mathfrak{v}}^{\varepsilon, \nu}(s))\,ds \nonumber\\
		&\quad + \frac{1}{\sqrt{\varepsilon}}
		\int_{p\delta}^{t+p\delta}
		\Sigma_2\!\left({\mathfrak{u}}^{\varepsilon, \nu}(p\delta),
		\hat {\mathfrak{v}}^{\varepsilon, \nu}(s)\right)\,d\mathcal{W}_2 \nonumber\\
		&= {\mathfrak{v}}^{\varepsilon, \nu}(p\delta)
		+ \frac{1}{\varepsilon}
		\int_0^t (1+i)(-\Delta)^{\rho} \hat {\mathfrak{v}}^{\varepsilon, \nu}(s+p\delta)\,ds  + \frac{1}{\varepsilon}
		\int_0^t
		K\left({\mathfrak{u}}^{\varepsilon, \nu}(p\delta),
		\hat {\mathfrak{v}}^{\varepsilon, \nu}(s+p\delta)\right)\,ds \nonumber\\
		&\quad + \frac{1}{\sqrt{\varepsilon}}
		\int_0^t
		\Sigma_2\!\left({\mathfrak{u}}^{\varepsilon, \nu}(p\delta),
		\hat {\mathfrak{v}}^{\varepsilon, \nu}(s+p\delta)\right)\,d \tilde{\mathcal{W}}_2 ,
	\end{align}
	where $\tilde{\mathcal{W}}_2(\cdot)=\mathcal{W}_2(\cdot+p\delta)$ and hence $\tilde{\mathcal{W}}_2$ and $\mathcal{W}_2$ have the same distribution. On the same stochastic basis, let $\tilde W(\cdot)$ be a Wiener process independent of $\mathcal{W}_1(\cdot)$ and $\mathcal{W}_2(\cdot)$. Now, we define a process  ${\mathfrak{v}}^{{\mathfrak{u}}^{\varepsilon, \nu}(p\delta),{\mathfrak{v}}^{\varepsilon, \nu}(p\delta)}(\cdot)\in L^2(\mathbb{T})$ as 
	\begin{align*}
		&{\mathfrak{v}}^{{\mathfrak{u}}^{\varepsilon, \nu}(p\delta),{\mathfrak{v}}^{\varepsilon, \nu}(p\delta)}\left(\frac{t}{\varepsilon}\right)
		= {\mathfrak{v}}^{\varepsilon, \nu}(p\delta)
		+ \int_0^{t/\varepsilon}
		(1+i)(-\Delta)^{\rho} {\mathfrak{v}}^{{\mathfrak{u}}^{\varepsilon, \nu}(p\delta),{\mathfrak{v}}^{\varepsilon, \nu}(p\delta)}(s)\,ds \\
		&\quad + \int_0^{t/\varepsilon}
		K\left({\mathfrak{u}}^{\varepsilon, \nu}(p\delta),
		{\mathfrak{v}}^{{\mathfrak{u}}^{\varepsilon, \nu}(p\delta),{\mathfrak{v}}^{\varepsilon, \nu}(p\delta)}(s)\right)\,ds  + \int_0^{t/\varepsilon}
		\Sigma_2\!\left({\mathfrak{u}}^{\varepsilon, \nu}(p\delta),
		{\mathfrak{v}}^{{\mathfrak{u}}^{\varepsilon, \nu}(p\delta),{\mathfrak{v}}^{\varepsilon, \nu}(p\delta)}(r)\right)\,d\tilde W.
	\end{align*}
	Using the transformation $s'=s\varepsilon, r'=r\varepsilon$ and again replacing $s', r'$ by $s, r$, we get
	\begin{align}\label{vpoweruv}
		&{\mathfrak{v}}^{{\mathfrak{u}}^{\varepsilon, \nu}(p\delta),{\mathfrak{v}}^{\varepsilon, \nu}(p\delta)}\left(\frac{t}{\varepsilon}\right)
		= {\mathfrak{v}}^{\varepsilon, \nu}(p\delta)
		+ \frac{1}{\varepsilon}
		\int_0^t
		(1+i)(-\Delta)^{\rho} {\mathfrak{v}}^{{\mathfrak{u}}^{\varepsilon, \nu}(p\delta),{\mathfrak{v}}^{\varepsilon, \nu}(p\delta)}\left(\frac{s}{\varepsilon}\right)\,ds\nonumber \\
		&\quad + \frac{1}{\varepsilon}
		\int_0^t
		K\left({\mathfrak{u}}^{\varepsilon, \nu}(p\delta),
		{\mathfrak{v}}^{{\mathfrak{u}}^{\varepsilon, \nu}(p\delta),{\mathfrak{v}}^{\varepsilon, \nu}(p\delta)}\left(\frac{s}{\varepsilon}\right)\right)\,ds  + \frac{1}{\sqrt{\varepsilon}}
		\int_0^t
		\Sigma_2\left({\mathfrak{u}}^{\varepsilon, \nu}(p\delta), {\mathfrak{v}}^{{\mathfrak{u}}^{\varepsilon, \nu}(p\delta),{\mathfrak{v}}^{\varepsilon, \nu}(p\delta)}\left(\frac{r}{\varepsilon}\right)\right)\,d\tilde{\tilde W },
	\end{align}
	where $\tilde{\tilde W }$ is the scaled version of $\tilde W$ defined by $\tilde W(\frac{r}{\varepsilon}):=\frac{1}{\sqrt{\varepsilon}} \tilde{\tilde W }(r)$. Therefore, the uniqueness of strong solution of \eqref{shiftedequation} and \eqref{vpoweruv} imply 
	\begin{align}\label{shiftequalscale}
		\mathscr{L}(\hat {\mathfrak{v}}^{\varepsilon, \nu}(t+p\delta))=\mathscr{L}\left({\mathfrak{v}}^{{\mathfrak{u}}^{\varepsilon, \nu}(p\delta),{\mathfrak{v}}^{\varepsilon, \nu}(p\delta)}\left(\frac{t}{\varepsilon}\right)\right).
	\end{align}
	Now, using the Cauchy-Schwartz inequality, we estimate for $J_{11}$ as 
	\begin{align*}
		&\mathbb{E}\bigg[\sup_{0\le t\le T\wedge\tau_n^{\varepsilon,\nu}}
		\|J_{11}\|^2\bigg]\\
		&\le	\mathbb{E}\bigg[\sup_{0\le t\le T}	\bigg\|\sum_{k=0}^{m_t-1}\int_{k\delta}^{(k+1)\delta}
		S_\nu(t-s)	\bigl[	F({\mathfrak{u}}^{\varepsilon, \nu}(k\delta),\hat {\mathfrak{v}}^{\varepsilon, \nu}(s)) - \bar F({\mathfrak{u}}^{\varepsilon, \nu}(k\delta))	\bigr]ds	\bigg\|^2\bigg]\\
		&=	\mathbb{E}\bigg[\sup_{0\le t\le T}	\bigg\|\sum_{k=0}^{m_t-1}	\int_{k\delta}^{(k+1)\delta}
		S_\nu\bigl(t-(k+1)\delta\bigr)	S_\nu\bigl((k+1)\delta-s\bigr) \bigl[	F({\mathfrak{u}}^{\varepsilon, \nu}(k\delta),\hat {\mathfrak{v}}^{\varepsilon, \nu}(s))- \bar F({\mathfrak{u}}^{\varepsilon, \nu}(k\delta))
		\bigr]ds\bigg\|^2\bigg]\\
		&= \mathbb{E}\bigg[\sup_{0\le t\le T}	\bigg\|	\sum_{k=0}^{m_t-1} 	S_\nu\bigl(t-(k+1)\delta\bigr)	\int_{k\delta}^{(k+1)\delta}S_\nu\bigl((k+1)\delta-s\bigr)
		\bigl[F({\mathfrak{u}}^{\varepsilon, \nu}(k\delta),\hat {\mathfrak{v}}^{\varepsilon, \nu}(s))
		- \bar F({\mathfrak{u}}^{\varepsilon, \nu}(k\delta))\bigr]ds	\bigg\|^2 \bigg]\\[1ex]
		&\le	\mathbb{E}\bigg[\sup_{0\le t\le T}	\bigg\{	m_t \sum_{k=0}^{m_t-1}	\bigg\|
		S_\nu\bigl(t-(k+1)\delta\bigr)	\int_{k\delta}^{(k+1)\delta}S_\nu\bigl((k+1)\delta-s\bigr)
		\bigl[	F({\mathfrak{u}}^{\varepsilon, \nu}(k\delta),\hat {\mathfrak{v}}^{\varepsilon, \nu}(s))	\\
		&\quad\quad- \bar F({\mathfrak{u}}^{\varepsilon, \nu}(k\delta))	\bigr]ds\bigg\|^2	\bigg\}\bigg] \\[1ex]
		&\le\mathbb{E}\bigg[\sup_{0\le t\le T}\bigg\{	m_t \sum_{k=0}^{m_t-1}\bigg\|	\int_{k\delta}^{(k+1)\delta}
		S_\nu\bigl((k+1)\delta-s\bigr)	\bigl[	F({\mathfrak{u}}^{\varepsilon, \nu}(k\delta),\hat {\mathfrak{v}}^{\varepsilon, \nu}(s))	- \bar F({\mathfrak{u}}^{\varepsilon, \nu}(k\delta))	\bigr]ds\bigg\|^2	\bigg\}\bigg] \\[1ex]
		&\le	\bigg[\frac{T}{\delta}\bigg]\mathbb{E}	\bigg[	\sum_{k=0}^{\big[\frac{T}{\delta}\big]-1}\bigg\|	\int_{k\delta}^{(k+1)\delta}
		S_\nu\bigl((k+1)\delta-s\bigr)	\bigl[F({\mathfrak{u}}^{\varepsilon, \nu}(k\delta),\hat {\mathfrak{v}}^{\varepsilon, \nu}(s))
		- \bar F({\mathfrak{u}}^{\varepsilon, \nu}(k\delta))\bigr]ds\bigg\|^2\bigg] \\[1ex]
		&\le	\bigg[\frac{T}{\delta}\bigg]^2 	\max_{0\le k\le \big[\frac{T}{\delta}\big]-1}
		\mathbb{E}\bigg[	\bigg\|	\int_{k\delta}^{(k+1)\delta}
		S_\nu\bigl((k+1)\delta-s\bigr)	\bigl[F({\mathfrak{u}}^{\varepsilon, \nu}(k\delta),\hat {\mathfrak{v}}^{\varepsilon, \nu}(s))
		- \bar F({\mathfrak{u}}^{\varepsilon, \nu}(k\delta))\bigr]ds\bigg\|^2\bigg].
	\end{align*}
	Using the transformation $s'=s\varepsilon+ k \delta$, we deduce
	\begin{align*}
		&=	\bigg[\frac{T}{\delta}\bigg]^2	\varepsilon^2	\max_{0\le k\le \big[\frac{T}{\delta}\big]-1}
		\mathbb{E}	\bigg[\Big\|	\int_0^{\frac{\delta}{\varepsilon}}	S_\nu(\delta-s\varepsilon)\bigl[
		F({\mathfrak{u}}^{\varepsilon, \nu}(k\delta),\hat {\mathfrak{v}}^{\varepsilon, \nu}(s\varepsilon+k\delta))	- \bar F({\mathfrak{u}}^{\varepsilon, \nu}(k\delta))	\bigr]ds	\Big\|^2\bigg] \\[1ex]
		&=	2\bigg[\frac{T}{\delta}\bigg]^2	\varepsilon^2	\max_{0\le k\le \big[\frac{T}{\delta}\big]-1} \int_{\mathbb{T}} \mathbb{E}	\bigg[	\int_0^{\frac{\delta}{\varepsilon}}	\int_r^{\frac{\delta}{\varepsilon}}
		S_\nu(\delta-s\varepsilon)	\bigl[	F({\mathfrak{u}}^{\varepsilon, \nu}(k\delta),\hat {\mathfrak{v}}^{\varepsilon, \nu}(s\varepsilon+k\delta))	- \bar F({\mathfrak{u}}^{\varepsilon, \nu}(k\delta))
		\bigr] \\
		&\qquad\times	S_\nu(\delta-r\varepsilon)
		\bigl[	F({\mathfrak{u}}^{\varepsilon, \nu}(k\delta),\hat {\mathfrak{v}}^{\varepsilon, \nu}(r\varepsilon+k\delta))
		- \bar F({\mathfrak{u}}^{\varepsilon, \nu}(k\delta))	\bigr]	\,ds\,dr	\bigg]dx \\[1ex]
		&\leq	2\bigg[\frac{T}{\delta}\bigg]^2\varepsilon^2\max_{0\le k\le \big[\frac{T}{\delta}\big]-1}\int_0^{\frac{\delta}{\varepsilon}}\int_r^{\frac{\delta}{\varepsilon}}	J_k(s,r)\,ds\,dr ,
	\end{align*}
	where 
	\begin{align*}
		J_k(s,r)
		&:= \mathbb{E}\bigg[\int_{\mathbb{T}} \Big[
		S_\nu(\delta-s\varepsilon)
		\bigl(
		F({\mathfrak{u}}^{\varepsilon, \nu}(k\delta),\hat {\mathfrak{v}}^{\varepsilon, \nu}(s\varepsilon+k\delta))
		- \bar F({\mathfrak{u}}^{\varepsilon, \nu}(k\delta))
		\bigr)\\
		&  \qquad\times
		S_\nu(\delta-r\varepsilon)
		\bigl(
		F({\mathfrak{u}}^{\varepsilon, \nu}(k\delta),\hat {\mathfrak{v}}^{\varepsilon, \nu}(r\varepsilon+k\delta))
		- \bar F({\mathfrak{u}}^{\varepsilon, \nu}(k\delta))
		\bigr)	\Big]dx\bigg].
	\end{align*}
	Moving forward, we use the notation $\mathbb{E}^{x}[\cdot]=\mathbb{E}[\cdot|x]$ for conditional expectation.
	Using \eqref{shiftequalscale}, the properties of the semigroup, the Markov property and the H\"older inequality, we estimate
	\begin{align*}
		J_k(s,r)
		&= \mathbb{E}\bigg[\int_{\mathbb{T}} \Big[
		S_\nu(\delta-s\varepsilon)
		\bigl(
		F({\mathfrak{u}}^{\varepsilon, \nu}(k\delta),\hat {\mathfrak{v}}^{\varepsilon, \nu}(s\varepsilon+k\delta))
		- \bar F({\mathfrak{u}}^{\varepsilon, \nu}(k\delta))
		\bigr)\\
		& \qquad \times
		S_\nu(\delta-r\varepsilon)
		\bigl(
		F({\mathfrak{u}}^{\varepsilon, \nu}(k\delta),\hat {\mathfrak{v}}^{\varepsilon, \nu}(r\varepsilon+k\delta))
		- \bar F({\mathfrak{u}}^{\varepsilon, \nu}(k\delta))
		\bigr)	\Big]dx\bigg]\\
		&= \mathbb{E}\bigg[
		\int_{\mathbb{T}}
		S_\nu(\delta-s\varepsilon)
		\Big[
		F({\mathfrak{u}}^{\varepsilon, \nu}(k\delta),
		{\mathfrak{v}}^{{\mathfrak{u}}^{\varepsilon, \nu}(k\delta),{\mathfrak{v}}^{\varepsilon, \nu}(k\delta)}(s))
		- \bar F({\mathfrak{u}}^{\varepsilon, \nu}(k\delta))
		\Big] \\
		&\qquad \times
		S_\nu(\delta-r\varepsilon)
		\Big[
		F({\mathfrak{u}}^{\varepsilon, \nu}(k\delta),
		{\mathfrak{v}}^{{\mathfrak{u}}^{\varepsilon, \nu}(k\delta),{\mathfrak{v}}^{\varepsilon, \nu}(k\delta)}(r))
		- \bar F({\mathfrak{u}}^{\varepsilon, \nu}(k\delta))
		\Big]\,dx \bigg]\\
		&= \mathbb{E}\bigg[
		\int_{\mathbb{T}}
		S_\nu(\delta-s\varepsilon)
		\Big[
		F({\mathfrak{u}}^{\varepsilon, \nu}(k\delta),
		{\mathfrak{v}}^{{\mathfrak{u}}^{\varepsilon, \nu}(k\delta),{\mathfrak{v}}^{\varepsilon, \nu}(k\delta)}(s))
		- \bar F({\mathfrak{u}}^{\varepsilon, \nu}(k\delta))
		\Big] \\
		&\qquad \times
		\mathbb{E}^{{\mathfrak{v}}^{{\mathfrak{u}}^{\varepsilon, \nu}(k\delta),{\mathfrak{v}}^{\varepsilon, \nu}(k\delta)}(r)}
				\Big(
		S_\nu(\delta-r\varepsilon)
		\Big[
		F({\mathfrak{u}}^{\varepsilon, \nu}(k\delta),
		{\mathfrak{v}}^{{\mathfrak{u}}^{\varepsilon, \nu}(k\delta),{\mathfrak{v}}^{\varepsilon, \nu}(k\delta)}(s-r))
		- \bar F({\mathfrak{u}}^{\varepsilon, \nu}(k\delta))
		\Big]
		\Big)\,dx\bigg] \\[1ex]
		&\le
		\left\{
		\mathbb{E}\bigg[
		\int_{\mathbb{T}}
		\left|
		S_\nu(\delta-s\varepsilon)
		\Big[
		F({\mathfrak{u}}^{\varepsilon, \nu}(k\delta),
		{\mathfrak{v}}^{{\mathfrak{u}}^{\varepsilon, \nu}(k\delta),{\mathfrak{v}}^{\varepsilon, \nu}(k\delta)}(s))
		- \bar F({\mathfrak{u}}^{\varepsilon, \nu}(k\delta))
		\Big]
		\right|^2 dx\bigg]
		\right\}^{1/2} \\
		&\quad \times
		\bigg\{
		\mathbb{E}\bigg[
		\int_{\mathbb{T}}
		\bigg|
		\mathbb{E}^{{\mathfrak{v}}^{{\mathfrak{u}}^{\varepsilon, \nu}(k\delta),{\mathfrak{v}}^{\varepsilon, \nu}(k\delta)}(r)}
		\Big(
		S_\nu(\delta-r\varepsilon)\\
		&\quad\quad
		\Big[
		F({\mathfrak{u}}^{\varepsilon, \nu}(k\delta),
		{\mathfrak{v}}^{{\mathfrak{u}}^{\varepsilon, \nu}(k\delta),{\mathfrak{v}}^{\varepsilon, \nu}(k\delta)}(s-r))
		- \bar F({\mathfrak{u}}^{\varepsilon, \nu}(k\delta))
		\Big]
		\Big)
		\bigg|^2 dx\bigg]
		\bigg\}^{1/2} \\[1ex]
		&=
		\left(
		\mathbb{E}\bigg[
		\left\|
		S_\nu(\delta-s\varepsilon)
	\Big[
		F({\mathfrak{u}}^{\varepsilon, \nu}(k\delta),
		{\mathfrak{v}}^{{\mathfrak{u}}^{\varepsilon, \nu}(k\delta),{\mathfrak{v}}^{\varepsilon, \nu}(k\delta)}(s))
		- \bar F({\mathfrak{u}}^{\varepsilon, \nu}(k\delta))
		\Big]
		\right\|^2\bigg]
		\right)^{1/2} \\
		&\quad \times
		\bigg(
		\mathbb{E}\bigg[
		\bigg\|
		\mathbb{E}^{{\mathfrak{v}}^{{\mathfrak{u}}^{\varepsilon, \nu}(k\delta),{\mathfrak{v}}^{\varepsilon, \nu}(k\delta)}(r)}
		\Big(
		S_\nu(\delta-r\varepsilon)\\
		&\quad\quad
		\Big[
		F({\mathfrak{u}}^{\varepsilon, \nu}(k\delta),
		{\mathfrak{v}}^{{\mathfrak{u}}^{\varepsilon, \nu}(k\delta),{\mathfrak{v}}^{\varepsilon, \nu}(k\delta)}(s-r))
		- \bar F({\mathfrak{u}}^{\varepsilon, \nu}(k\delta))
		\Big]
		\Big)
		\bigg\|^2\bigg]
		\bigg)^{1/2} \\[1ex]
		&\le
		\left(
		\mathbb{E}\Big[
		\bigl\|
		F({\mathfrak{u}}^{\varepsilon, \nu}(k\delta),
		{\mathfrak{v}}^{{\mathfrak{u}}^{\varepsilon, \nu}(k\delta),{\mathfrak{v}}^{\varepsilon, \nu}(k\delta)}(s))
		- \bar F({\mathfrak{u}}^{\varepsilon, \nu}(k\delta))
		\bigr\|^2\Big]
		\right)^{1/2} \\
		&\quad \times
		\left(
		\mathbb{E}\bigg[
		\left\|
		\mathbb{E}^{{\mathfrak{v}}^{{\mathfrak{u}}^{\varepsilon, \nu}(k\delta),{\mathfrak{v}}^{\varepsilon, \nu}(k\delta)}(r)}
		\Big[F({\mathfrak{u}}^{\varepsilon, \nu}(k\delta),	{\mathfrak{v}}^{{\mathfrak{u}}^{\varepsilon, \nu}(k\delta),{\mathfrak{v}}^{\varepsilon, \nu}(k\delta)}(s-r))	- \bar F({\mathfrak{u}}^{\varepsilon, \nu}(k\delta))		\Big]	\right\|^2\bigg]	\right)^{1/2}.
	\end{align*}
	As a consequence of \Cref{propertiesofivpsolution}, we have
	\begin{align}
		J_k(s,r) \le C e^{-\mu (s-r)},
	\end{align}\label{j111}
	for some $\mu>0.$ Therefore, we estimate 
	\begin{align}
		\mathbb{E}
		\bigg[	\sup_{0\le t\le T\wedge\tau_n^{\varepsilon,\nu}}
		\|J_{11}\|^2\bigg]
		&\le
		2\left[\frac{T}{\delta}\right]^2
		\varepsilon^2
		\max_{0\le k\le \left[\frac{T}{\delta}\right]-1}
		\int_0^{\frac{\delta}{\varepsilon}}
		\int_r^{\frac{\delta}{\varepsilon}}
		J_k(s,r)\,ds\,dr \nonumber\\
		&\le
		C\left[\frac{T}{\delta}\right]^2
		\varepsilon^2
		\max_{0\le k\le \left[\frac{T}{\delta}\right]-1}
		\int_0^{\frac{\delta}{\varepsilon}}
		\int_r^{\frac{\delta}{\varepsilon}}
		e^{-\mu (s-r)}\,ds\,dr \nonumber\\
		&\le	2\left[\frac{T}{\delta}\right]^2 \varepsilon^2\left\{\frac{\delta}{\mu \varepsilon}+\frac{1}{\mu^2}\Big(e^{-\frac{\mu \delta}{\varepsilon}}-1\Big)\right\}\nonumber\\
		&\le	C\,\frac{\varepsilon}{\delta}.
	\end{align}
	We use the H\"older inequality in time, boundedness of the semigroup to obtain
	\begin{align}\label{j112}
		\mathbb{E}\bigg[\sup_{0\le t\le T\wedge\tau_n^{\varepsilon,\nu}}
		\|J_{11}\|^{2p}\bigg]
		&\le
		\mathbb{E}\bigg[\sup_{0\le t\le T}
		\bigg\|
		\sum_{k=0}^{m_t-1}
		\int_{k\delta}^{(k+1)\delta}
		S_\nu(t-s)
		\bigl[
		F({\mathfrak{u}}^{\varepsilon, \nu}(k\delta),\hat {\mathfrak{v}}^{\varepsilon, \nu}(s))
		- \bar F({\mathfrak{u}}^{\varepsilon, \nu}(k\delta))
		\bigr]ds
		\bigg\|^{2p}\bigg] \nonumber\\[1ex]
		&=
		\mathbb{E}\bigg[\sup_{0\le t\le T}
		\bigg\|
		\int_0^{m_t\delta}
		S_\nu(t-s)
		\bigl[
		F({\mathfrak{u}}^{\varepsilon, \nu}(t_s),\hat {\mathfrak{v}}^{\varepsilon, \nu}(s))
		- \bar F({\mathfrak{u}}^{\varepsilon, \nu}(t_s))
		\bigr]ds
		\bigg\|^{2p}\bigg] \nonumber\\[1ex]
		&\le
		\mathbb{E}\bigg[\sup_{0\le t\le T}
		\bigg(
		\int_0^{m_t\delta}
		\bigl\|
		S_\nu(t-s)
		\bigl[
		F({\mathfrak{u}}^{\varepsilon, \nu}(t_s),\hat {\mathfrak{v}}^{\varepsilon, \nu}(s))
		- \bar F({\mathfrak{u}}^{\varepsilon, \nu}(t_s))
		\bigr]
		\bigr\|\,ds
		\bigg)^{2p}\bigg] \nonumber\\[1ex]
		&\le
		\mathbb{E}
		\bigg[
		\int_0^{T}
		\bigl\|
		F({\mathfrak{u}}^{\varepsilon, \nu}(t_s),\hat {\mathfrak{v}}^{\varepsilon, \nu}(s))
		- \bar F({\mathfrak{u}}^{\varepsilon, \nu}(t_s))
		\bigr\|\,ds
		\bigg]^{2p}\nonumber \\[1ex]
		&\le C
		\mathbb{E}
		\bigg[
		\int_0^{T}
		\bigl\|
		F({\mathfrak{u}}^{\varepsilon, \nu}(t_s),\hat {\mathfrak{v}}^{\varepsilon, \nu}(s))
		- \bar F({\mathfrak{u}}^{\varepsilon, \nu}(t_s))
		\bigr\|^{2p}\,ds
		\bigg] \nonumber\\[1ex]
		&\le
		C\,\mathbb{E}\bigg[
		\int_0^{T}
		\Bigl[
		1+\|{\mathfrak{u}}^{\varepsilon, \nu}(t_s)\|^{2p}
		+ \|\hat {\mathfrak{v}}^{\varepsilon, \nu}(s)\|^{2p}
		\Bigr]\,ds\bigg] \nonumber\\
		&\le C(p,T).
	\end{align}
	Thus, the H\"older inequality with \eqref{j111} and \eqref{j112} imply
	\begin{align*}
		\mathbb{E}
		\bigg[\sup_{0\le t\le T\wedge\tau_n^{\varepsilon,\nu}}
		\|J_{11}\|^{2p}\bigg]
		\le
		\Bigl(
		\mathbb{E}\Big[\sup_{0\le t\le T}\|J_{11}\|^{2(p-1)}\Big]
		\Bigr)^{1/2}
		\Bigl(
		\mathbb{E}\Big[\sup_{0\le t\le T}\|J_{11}\|^{2}\Big]
		\Bigr)^{1/2}
		\le
		C(p,T)\sqrt{\frac{\varepsilon}{\delta}}.
	\end{align*}
	For $J_{12}$, using the Lipschitz continuity of $\bar F$ proved in \Cref{lipforfbar} and the Jensen continuity in time for ${\mathfrak{u}}^{\varepsilon, \nu}$, we estimate 
	\begin{align*}
		\mathbb{E}\bigg[\sup_{0\le t\le T\wedge\tau_n^{\varepsilon,\nu}}
		\|J_{12}\|^{2p}\bigg]
		&\le
		\mathbb{E}\bigg[\sup_{0\le t\le T}
		\bigg\|
		\sum_{k=0}^{m_t-1}
		\int_{k\delta}^{(k+1)\delta}
		S_\nu(t-s)
		\bigl[
		\bar F({\mathfrak{u}}^{\varepsilon, \nu}(k\delta))
		- \bar F({\mathfrak{u}}^{\varepsilon, \nu}(s))
		\bigr]ds
		\bigg\|^{2p}\bigg] \\[1ex]
		&\le
		\mathbb{E}\bigg[\sup_{0\le t\le T}
		\bigg\|
		\sum_{k=0}^{m_t-1}
		\int_{k\delta}^{(k+1)\delta}
		S_\nu(t-s)
		\bigl[
		\bar F({\mathfrak{u}}^{\varepsilon, \nu}(t_s))
		- \bar F({\mathfrak{u}}^{\varepsilon, \nu}(s))
		\bigr]ds
		\bigg\|^{2p} \bigg]\\[1ex]
		&=
		\mathbb{E}\bigg[\sup_{0\le t\le T}
		\bigg\|
		\int_0^{m_t\delta}
		S_\nu(t-s)
		\bigl[
		\bar F({\mathfrak{u}}^{\varepsilon, \nu}(t_s))
		- \bar F({\mathfrak{u}}^{\varepsilon, \nu}(s))
		\bigr]ds
		\bigg\|^{2p} \bigg]\\[1ex]
		&\le
		\mathbb{E}\bigg[\sup_{0\le t\le T}
		\bigg(
		\int_0^{m_t\delta}
		\bigl\|
		S_\nu(t-s)
		\bigl[
		\bar F({\mathfrak{u}}^{\varepsilon, \nu}(t_s))
		- \bar F({\mathfrak{u}}^{\varepsilon, \nu}(s))
		\bigr]
		\bigr\|\,ds
		\bigg)^{2p} \bigg]\\[1ex]
		&\le C
		\mathbb{E}\bigg[\sup_{0\le t\le T}
		\bigg(
		\int_0^{m_t\delta}
		\bigl\|
		\bigl[
		\bar F({\mathfrak{u}}^{\varepsilon, \nu}(t_s))
		- \bar F({\mathfrak{u}}^{\varepsilon, \nu}(s))
		\bigr]
		\bigr\|\,ds
		\bigg)^{2p}\bigg] \\[1ex]
		&\le
		C\,\mathbb{E}\bigg[\sup_{0\le t\le T}
		\bigg(
		\int_0^{m_t\delta}
		\|{\mathfrak{u}}^{\varepsilon, \nu}(t_s)-{\mathfrak{u}}^{\varepsilon, \nu}(s)\|\,ds
		\bigg)^{2p}\bigg]\\
		&\le
		C\,\mathbb{E}
		\bigg[
		\int_0^T
		\|{\mathfrak{u}}^{\varepsilon, \nu}(t_s)-{\mathfrak{u}}^{\varepsilon, \nu}(s)\|\,ds
		\bigg]^{2p} \\
		&\le
		C\,\mathbb{E}
		\bigg[
		\bigg(\int_0^T ds\bigg)^{2p-1}
		\int_0^T
		\|{\mathfrak{u}}^{\varepsilon, \nu}(t_s)-{\mathfrak{u}}^{\varepsilon, \nu}(s)\|^{2p}\,ds
		\bigg]\\
		&\le
		C\,\delta^{p}.
	\end{align*}
	We use \Cref{difffbarf}, Lipschitz property of $F$ (in \Cref{main assumptions}) and $\bar F$ (proved in \Cref{lipforfbar}) to estimate $J_{13}$ as follows:
	\begin{align*}
		&\mathbb{E}\bigg[\sup_{0\le t\le T\wedge\tau_n^{\varepsilon,\nu}}
		\|J_{13}\|^{2p}\bigg]\\
		&=
		\mathbb{E}\bigg[\sup_{0\le t\le T}
		\bigg\|
		\int_{m_t\delta}^t
		S_\nu(t-s)
		\bigl[
		F({\mathfrak{u}}^{\varepsilon, \nu}(m_t\delta),\hat {\mathfrak{v}}^{\varepsilon, \nu}(s))
		- F({\mathfrak{u}}^{\varepsilon, \nu}(s))
		\bigr]ds
		\bigg\|^{2p}\bigg] \\[1ex]
		&\le
		\mathbb{E}\bigg[\sup_{0\le t\le T}
		\bigg(
		\int_{m_t\delta}^t
		\bigl\|
		S_\nu(t-s)
		\bigl[
		F({\mathfrak{u}}^{\varepsilon, \nu}(m_t\delta),\hat {\mathfrak{v}}^{\varepsilon, \nu}(s))
		- F({\mathfrak{u}}^{\varepsilon, \nu}(s))
		\bigr]
		\bigr\|\,ds
		\bigg)^{2p} \bigg]\\[1ex]
		&\le
		\mathbb{E}\bigg[\sup_{0\le t\le T}
		\bigg(
		\int_{m_t\delta}^t
		\bigl\|
		F({\mathfrak{u}}^{\varepsilon, \nu}(m_t\delta),\hat {\mathfrak{v}}^{\varepsilon, \nu}(s))
		- F({\mathfrak{u}}^{\varepsilon, \nu}(s))
		\bigr\|\,ds
		\bigg)^{2p} \bigg]\\[1ex]
		&\leq C \mathbb{E}\bigg[ \sup_{0 \leq t \leq T} \bigg\{\bigg(\int_{m_t\delta}^t 1 ds\bigg)^{2p-1}  \bigg(\int_{m_t\delta}^t \bigl\|
		F({\mathfrak{u}}^{\varepsilon, \nu}(m_t\delta),\hat {\mathfrak{v}}^{\varepsilon, \nu}(s))
		- F({\mathfrak{u}}^{\varepsilon, \nu}(s))
		\bigr\|^{2p} ds\bigg)   \bigg\}  \bigg]\\
		&\leq C \mathbb{E}  \bigg[ \bigg\{\sup_{0 \leq t \leq T} \bigg(\int_{m_t\delta}^t 1 ds\bigg)^{2p-1}\bigg\}  \bigg\{\sup_{0 \leq t \leq T} \bigg(\int_{m_t\delta}^t \bigl\|
		F({\mathfrak{u}}^{\varepsilon, \nu}(m_t\delta),\hat {\mathfrak{v}}^{\varepsilon, \nu}(s))
		- F({\mathfrak{u}}^{\varepsilon, \nu}(s))
		\bigr\|^{2p} ds\bigg)\bigg\}     \bigg]\\
		&= C\sup_{0 \leq t \leq T} (t-m_t\delta)^{2p-1}  \mathbb{E} \bigg[\sup_{0 \leq t \leq T} \bigg(\int_{m_t\delta}^t \bigl\|
		F({\mathfrak{u}}^{\varepsilon, \nu}(m_t\delta),\hat {\mathfrak{v}}^{\varepsilon, \nu}(s))
		- F({\mathfrak{u}}^{\varepsilon, \nu}(s))
		\bigr\|^{2p} ds\bigg) \bigg]\\
		&\le C\,\delta^{2p-1}\,
		\mathbb{E}\bigg[\sup_{0\le t\le T}
		\int_{m_t\delta}^{t}
		\Big(1+\|{\mathfrak{u}}^{\varepsilon, \nu}(m_t\delta)\|^{2p}
		+\|\hat {\mathfrak{v}}^{\varepsilon, \nu}(s)\|^{2p}
		+\|{\mathfrak{u}}^{\varepsilon, \nu}(s)\|^{2p}\Big) ds \bigg]\\
		&\le C\,\delta^{2p-1}\,
		\mathbb{E}\bigg[\sup_{0\le t\le T}
		\int_{0}^{T}
		\Big(1+\sup_{0 \leq t \leq T}\|{\mathfrak{u}}^{\varepsilon, \nu}(m_t\delta)\|^{2p}
		+\|\hat {\mathfrak{v}}^{\varepsilon, \nu}(s)\|^{2p}
		+\|{\mathfrak{u}}^{\varepsilon, \nu}(s)\|^{2p}\Big) ds \bigg]\\
		&\le C\,\delta^{2p-1}\,
		\mathbb{E}\bigg[
		\int_{0}^{T}
		\Big(1+\sup_{0 \leq t \leq T}\|{\mathfrak{u}}^{\varepsilon, \nu}(m_t\delta)\|^{2p}
		+\|\hat {\mathfrak{v}}^{\varepsilon, \nu}(s)\|^{2p}
		+\|{\mathfrak{u}}^{\varepsilon, \nu}(s)\|^{2p}\Big) ds \bigg]\\
		&\le
		C\,\delta^{2p-1}.
	\end{align*}
	For $J_{14}$, we use the boundedness of the semigroup, Lipschitz property of $\bar F$, the H\"older inequality in time and \eqref{uhatuerror} to estimate 
	\begin{align*}
		&\mathbb{E}\bigg[\sup_{0\le t\le T\wedge\tau_n^{\varepsilon,\nu}}
		\|J_{14}\|^{2p}\bigg]
		=
		\mathbb{E}\bigg[\sup_{0\le t\le T}
		\bigg\|
		\int_0^t
		S_\nu(t-s)
		\bigl[
		\bar F({\mathfrak{u}}^{\varepsilon, \nu}(s))
		- \bar F(\hat {\mathfrak{u}}^{\varepsilon, \nu}(s))
		\bigr]ds
		\bigg\|^{2p}\bigg]\\
		&\le
		\mathbb{E}\bigg[
		\sup_{0\le t\le T}
		\int_0^t
		\bigl\|
		S_\nu(t-s)
		\bigl[
		\bar F({\mathfrak{u}}^{\varepsilon, \nu}(s))
		- \bar F(\hat {\mathfrak{u}}^{\varepsilon, \nu}(s))
		\bigr]
		\bigr\|\,ds
		\bigg]^{2p}\le
		C\,\mathbb{E}
		\bigg[	
		\sup_{0\le t\le T}
		\int_0^t
		\|{\mathfrak{u}}^{\varepsilon, \nu}(s)-\hat {\mathfrak{u}}^{\varepsilon, \nu}(s)\|\,ds
		\bigg]	^{2p} \\[1ex]
		&\le
		C\,\mathbb{E}
		\bigg[
		\int_0^T
		\|{\mathfrak{u}}^{\varepsilon, \nu}(s)-\hat {\mathfrak{u}}^{\varepsilon, \nu}(s)\|\,ds
		\bigg]^{2p} \le
		C\,\mathbb{E}\bigg[
		\int_0^T
		\|{\mathfrak{u}}^{\varepsilon, \nu}(s)-\hat {\mathfrak{u}}^{\varepsilon, \nu}(s)\|^{2p}\,ds\bigg] \le
		C\bigg(
		\delta^{p}
		+\frac{\delta^{p+1}}{\varepsilon}e^{\frac{C}{\varepsilon}\delta}
		\bigg).
	\end{align*}
	To estimate for $J_{15}$, we compute
	\begin{align*}
		&\mathbb{E}
		\bigg[	\sup_{0\le t\le T\wedge\tau_n^{\varepsilon,\nu}}
		\|J_{15}\|^{2p}\bigg]
		=
		\mathbb{E}
		\bigg[	\sup_{0\le t\le T\wedge\tau_n^{\varepsilon,\nu}}
		\bigg\|
		\int_0^t
		S_\nu(t-s)
		\bigl[
		\bar F(\hat {\mathfrak{u}}^{\varepsilon, \nu}(s))
		- \bar F(\bar{\mathfrak{u}}^\nu(s))
		\bigr]ds
		\bigg\|^{2p}\bigg] \\[1ex]
		&\le
		\mathbb{E}
		\bigg[\sup_{0\le t\le T\wedge\tau_n^{\varepsilon,\nu}}
		\bigg(
		\int_0^t
		\bigl\|
		S_\nu(t-s)
		\bigl[
		\bar F(\hat {\mathfrak{u}}^{\varepsilon, \nu}(s))
		- \bar F(\bar{\mathfrak{u}}^\nu(s))
		\bigr]
		\bigr\|\,ds
		\bigg)^{2p}\bigg] \\[1ex]
		&\le
		\mathbb{E}
		\bigg[\sup_{0\le t\le T\wedge\tau_n^{\varepsilon,\nu}}
		\bigg(
		\int_0^t
		\bigl\|
		\bar F(\hat {\mathfrak{u}}^{\varepsilon, \nu}(s))
		- \bar F(\bar{\mathfrak{u}}^\nu(s))
		\bigr\|\,ds
		\bigg)^{2p}\bigg]\le
		C\,\mathbb{E}\bigg[
		\int_0^{T\wedge\tau_n^{\varepsilon,\nu}}
		\|\hat {\mathfrak{u}}^{\varepsilon, \nu}(s)-\bar{\mathfrak{u}}^\nu(s)\|^{2p}\,ds \bigg]\\[1ex]
		&\le
		C\,\mathbb{E}\bigg[
		\int_0^{T\wedge\tau_n^{\varepsilon,\nu}}
		\sup_{0\le r\le s\wedge\tau_n^{\varepsilon,\nu}}\|\hat {\mathfrak{u}}^{\varepsilon, \nu}(s)-\bar{\mathfrak{u}}^\nu(s)\|^{2p}\,ds \bigg]\le
		C\,\mathbb{E}\bigg[
		\int_0^{T}
		\sup_{0\le r\le s\wedge\tau_n^{\varepsilon,\nu}}
		\|\hat {\mathfrak{u}}^{\varepsilon, \nu}(r)-\bar{\mathfrak{u}}^\nu(r)\|^{2p}\,ds\bigg] \\[1ex]
		&\le
		C\int_0^{T}
		\mathbb{E}
		\bigg[	\sup_{0\le r\le s\wedge\tau_n^{\varepsilon,\nu}}
		\|\hat {\mathfrak{u}}^{\varepsilon, \nu}(r)-\bar{\mathfrak{u}}^\nu(r)\|^{2p}\bigg]\,ds .
	\end{align*}
	We estimate for $J_2$ as
	\begin{align*}
		J_2
		&=
		\int_0^t
		S_\nu(t-s)
		\bigl[
		\mathcal{N}({\mathfrak{u}}^{\varepsilon, \nu}(t_s))
		- \mathcal{N}({\mathfrak{u}}^{\varepsilon, \nu}(s))
		\bigr]ds +
		\int_0^t
		S_\nu(t-s)
		\bigl[
		\mathcal{N}({\mathfrak{u}}^{\varepsilon, \nu}(s))
		- \mathcal{N}(\hat {\mathfrak{u}}^{\varepsilon, \nu}(s))
		\bigr]ds \\
		&\quad +
		\int_0^t
		S_\nu(t-s)
		\bigl[
		\mathcal{N}(\hat {\mathfrak{u}}^{\varepsilon, \nu}(s))
		- \mathcal{N}(\bar{\mathfrak{u}}^\nu(s))
		\bigr]ds \\
		&=: J_{21}+J_{22}+J_{23}.
	\end{align*}
	For $J_{21}$, we recall \eqref{liptypeoff} and get
	\begin{align*}
		\mathbb{E}\bigg[\sup_{0\le t\le T}
		\|J_{21}\|^{2p}\bigg]
		&=
		\mathbb{E}\bigg[\sup_{0\le t\le T}
		\bigg\|
		\int_0^t
		S_\nu(t-s)
		\bigl[
		\mathcal{N}({\mathfrak{u}}^{\varepsilon, \nu}(t_s))
		- \mathcal{N}({\mathfrak{u}}^{\varepsilon, \nu}(s))
		\bigr]ds
		\bigg\|^{2p} \bigg]\le
		C\,\delta^p .
	\end{align*}
	For $J_{22}$, again we recall \eqref{liptypeoff}, \eqref{uhatuerror}, use the estimates in \Cref{uniformestimates} and the Jensen inequality to conclude
	\begin{align*}
		\mathbb{E}\bigg[\sup_{0\le t\le T}
		\|J_{22}\|^{2p}\bigg]
		&=
		\mathbb{E}\bigg[\sup_{0\le t\le T}
		\bigg\|
		\int_0^t
		S_\nu(t-s)
		\bigl[
		\mathcal{N}({\mathfrak{u}}^{\varepsilon, \nu}(s))
		- \mathcal{N}(\hat {\mathfrak{u}}^{\varepsilon, \nu}(s))
		\bigr]ds
		\bigg\|^{2p}\bigg] \\[1ex]
		&\le
		C
		\bigg(
		\int_0^T
		\mathbb{E}\Big[
		\|{\mathfrak{u}}^{\varepsilon, \nu}(s)-\hat {\mathfrak{u}}^{\varepsilon, \nu}(s)\|^{4p}\Big]
		\,ds
		\bigg)^{\frac12}  
		\bigg(\mathbb{E}\bigg[\int_0^T
		\big(\|{\mathfrak{u}}^{\varepsilon, \nu}(s)\|_{H^1}^{\beta -1}
		+ \|\hat {\mathfrak{u}}^{\varepsilon, \nu}(s)\|_{H^1}^{\beta -1}\big)^{4p}
		\,ds\bigg]\bigg)^{\frac12} \\[1ex]
		&\le	C	\bigg(	\delta^p	+ 	\bigg(	\frac{\delta^{2p+1}}{\varepsilon} e^{\frac{C}{\varepsilon}\delta}	\bigg)^{\frac12}	\bigg).
	\end{align*}
	Let $\tau_n^{\varepsilon,\nu}$ be the stopping time defined in \eqref{stoppingtime}. Then, we estimate $J_{23}$ as
	\begin{align*}
		\mathbb{E}
		\bigg[	\sup_{0\le t\le T\wedge\tau_n^{\varepsilon,\nu}}
		\|J_{23}\|^{2p}\bigg]
		&\le
		\mathbb{E}
		\bigg[	\sup_{0\le t\le T\wedge\tau_n^{\varepsilon,\nu}}
		\bigg\|
		\int_0^t
		S_\nu(t-s)
		\bigl[
		\mathcal{N}(\hat {\mathfrak{u}}^{\varepsilon, \nu}(s))
		- \mathcal{N}(\bar{\mathfrak{u}}^\nu(s))
		\bigr]ds
		\bigg\|^{2p}\bigg] \\[1ex]
		&\le
		C\,\mathbb{E}
		\bigg[	\sup_{0\le t\le T\wedge\tau_n^{\varepsilon,\nu}}
		\bigg(
		\int_0^t
		\bigl\|
		\bigl[
		\mathcal{N}(\hat {\mathfrak{u}}^{\varepsilon, \nu}(s))
		- \mathcal{N}(\bar{\mathfrak{u}}^\nu(s))
		\bigr]
		\bigr\|\,ds
		\bigg)^{2p}\bigg] \\[1ex]
		&\le
		C\,\mathbb{E}
		\bigg[\sup_{0\le t\le T\wedge\tau_n^{\varepsilon,\nu}}
		\bigg(
		\int_0^t
		\|\hat {\mathfrak{u}}^{\varepsilon, \nu}(s)-\bar{\mathfrak{u}}^\nu(s)\|
		\bigl(
		\|\hat {\mathfrak{u}}^{\varepsilon, \nu}(s)\|_{H^1}
		+\|\bar{\mathfrak{u}}^\nu(s)\|_{H^1}
		\bigr)^{\beta-1}
		\,ds
		\bigg)^{2p}\bigg].
	\end{align*}
	From the definition of stopping times in \eqref{stoppingtime}, we deduce 
	\begin{align*}
		\mathbb{E}
		\bigg[\sup_{0\le t\le T\wedge\tau_n^{\varepsilon,\nu}}
		\|J_{23}\|^{2p}\bigg]
		&\le
		C\,\mathbb{E}
		\bigg[\sup_{0\le t\le T\wedge\tau_n^{\varepsilon,\nu}}
		\bigg(
		\int_0^t
		\|\hat {\mathfrak{u}}^{\varepsilon, \nu}(s)-\bar{\mathfrak{u}}^\nu(s)\|
		\bigl(
		\|\hat {\mathfrak{u}}^{\varepsilon, \nu}(s)\|_{H^1}
		+\|\bar{\mathfrak{u}}^\nu(s)\|_{H^1}
		\bigr)^{\beta-1}
		\,ds
		\bigg)^{2p} \bigg]\\[1ex]
		&\le
		C n^{2p(\beta-1)}
		\mathbb{E}
		\bigg[\sup_{0\le t\le T\wedge\tau_n^{\varepsilon,\nu}}
		\bigg(
		\int_0^t
		\|\hat {\mathfrak{u}}^{\varepsilon, \nu}(s)-\bar{\mathfrak{u}}^\nu(s)\|
		\,ds
		\bigg)^{2p}\bigg] \\[1ex]
		&\le
		C n^{2p(\beta-1)}
		\mathbb{E}
		\bigg[\sup_{0\le t\le T\wedge\tau_n^{\varepsilon,\nu}}
		\int_0^t
		\|\hat {\mathfrak{u}}^{\varepsilon, \nu}(s)-\bar{\mathfrak{u}}^\nu(s)\|^{2p}
		\,ds\bigg] \\[1ex]
		&\le
		C n^{2p(\beta-1)}
		\mathbb{E}\bigg[
		\int_0^{T\wedge\tau_n^{\varepsilon,\nu}}
		\|\hat {\mathfrak{u}}^{\varepsilon, \nu}(s)-\bar{\mathfrak{u}}^\nu(s)\|^{2p}
		\,ds \bigg]\\
		&\le C n^{2p(\beta-1)}\mathbb{E}\bigg[
		\int_0^{T\wedge\tau_n^{\varepsilon,\nu}}
		\sup_{0\le r\le s\wedge\tau_n^{\varepsilon,\nu}}
		\|{\mathfrak{u}}^{\varepsilon, \nu}(r)-\bar{\mathfrak{u}}^\nu(r)\|^{2p}\,ds\bigg] \\[1ex]
		&\le C n^{2p(\beta-1)}\mathbb{E}\bigg[
		\int_0^{T}
		\sup_{0\le r\le s\wedge\tau_n^{\varepsilon,\nu}}
		\|{\mathfrak{u}}^{\varepsilon, \nu}(r)-\bar{\mathfrak{u}}^\nu(r)\|^{2p}\,ds \bigg]\\[1ex]
		&\le C n^{2p(\beta-1)}
		\int_0^{T}
		\mathbb{E}
		\bigg[	\sup_{0\le r\le s\wedge\tau_n^{\varepsilon,\nu}}
		\|{\mathfrak{u}}^{\varepsilon, \nu}(r)-\bar{\mathfrak{u}}^\nu(r)\|^{2p}\bigg]\,ds .
	\end{align*}
	For $J_{3}$, we write
	\begin{align*}
		J_3
		&=\int_0^t S_\nu(t-s)\Big[
		\Sigma_1({\mathfrak{u}}^{\varepsilon, \nu}(s)) - \Sigma_1(\bar{\mathfrak{u}}^\nu(s))
		\Big]\,d\mathcal{W}_1(s) \\
		&=\int_0^t
		S_\nu(t-s)
		\bigl[\Sigma_1({\mathfrak{u}}^{\varepsilon, \nu}(s))
		-\Sigma_1(\hat {\mathfrak{u}}^{\varepsilon, \nu}(s))\bigr]\,d\mathcal{W}_1(s) +
		\int_0^t
		S_\nu(t-s)
		\bigl[\Sigma_1(\hat {\mathfrak{u}}^{\varepsilon, \nu}(s))
		-\Sigma_1(\bar{\mathfrak{u}}^\nu(s))\bigr]\,d\mathcal{W}_1(s) \\[1ex]
		&=: J_{31}+J_{32}.
	\end{align*}
	For $J_{31}$, using the H\"older inequality, the Burkholder-Davis-Gundy inequality in \Cref{BurkholderDavisGundy inequality}, the assumptions in \Cref{main assumptions} and the error estimate \eqref{uhatuerror}, we get
	\begin{align*}
		&\mathbb{E}
		\bigg[\sup_{0\le t\le T\wedge\tau_n^{\varepsilon,\nu}}
		\|J_{31}\|^{2p}\bigg]
		=
		\mathbb{E}
		\bigg[\sup_{0\le t\le T\wedge\tau_n^{\varepsilon,\nu}}
		\bigg\|
		\int_0^t
		S_\nu(t-s)
		\bigl[
		\Sigma_1({\mathfrak{u}}^{\varepsilon, \nu}(s))
		-\Sigma_1(\hat {\mathfrak{u}}^{\varepsilon, \nu}(s))
		\bigr]\,d\mathcal{W}_1(s)
		\bigg\|^{2p} \bigg]\\[1ex]
		&\le
		C\mathbb{E}
		\bigg(
		\int_0^T
		\|\Sigma_1({\mathfrak{u}}^{\varepsilon, \nu}(s))
		-\Sigma_1(\hat {\mathfrak{u}}^{\varepsilon, \nu}(s))\|^2\,ds
		\bigg)^p \le
		C\mathbb{E}
		\int_0^T
		\|({\mathfrak{u}}^{\varepsilon, \nu}-\hat {\mathfrak{u}}^{\varepsilon, \nu})(s)\|^{2p}\,ds \le
		C\bigg(
		\delta^p
		+
		\frac{\delta^{p+1}}{\varepsilon}
		\mathrm{e}^{\frac{C}{\varepsilon}\delta}
		\bigg).
	\end{align*}
	For $J_{32}$, once again we use the Burkholder-Davis-Gundy inequality and assumptions in \Cref{main assumptions} to estimate
	\begin{align*}
		&\mathbb{E}
		\bigg[\sup_{0\le t\le T\wedge\tau_n^{\varepsilon,\nu}}
		\|J_{32}\|^{2p}\bigg]
		=
		\mathbb{E}
		\bigg[	\sup_{0\le t\le T\wedge\tau_n^{\varepsilon,\nu}}
		\bigg\|
		\int_0^t
		S_\nu(t-s)
		\bigl[
		\Sigma_1(\hat {\mathfrak{u}}^{\varepsilon, \nu}(s))
		-\Sigma_1(\bar{\mathfrak{u}}^\nu(s))
		\bigr]\,d\mathcal{W}_1(s)
		\bigg\|^{2p}\bigg] \\[1ex]
		&\le
		C\mathbb{E}
		\bigg[
		\int_0^{T\wedge\tau_n^{\varepsilon,\nu}}
		\|\Sigma_1(\hat {\mathfrak{u}}^{\varepsilon, \nu}(s))
		-\Sigma_1(\bar{\mathfrak{u}}^\nu(s))\|^2\,ds
		\bigg]^p \le
		C\mathbb{E}\bigg[
		\int_0^{T\wedge\tau_n^{\varepsilon,\nu}}
		\|\hat {\mathfrak{u}}^{\varepsilon, \nu}(s)-\bar{\mathfrak{u}}^\nu(s)\|^{2p}\,ds\bigg] \\[1ex]
		&\le
		C\mathbb{E}\bigg[
		\int_0^{T\wedge\tau_n^{\varepsilon,\nu}}
		\sup_{0\le r\le s\wedge\tau_n^{\varepsilon,\nu}}
		\|\hat {\mathfrak{u}}^{\varepsilon, \nu}(r)-\bar{\mathfrak{u}}^\nu(r)\|^{2p}\,ds \bigg]\le
		C\int_0^T
		\mathbb{E}
		\bigg[	\sup_{0\le r\le s\wedge\tau_n^{\varepsilon,\nu}}
		\|\hat {\mathfrak{u}}^{\varepsilon, \nu}(r)-\bar{\mathfrak{u}}^\nu(r)\|^{2p}\bigg]\,ds.
	\end{align*}
	Therefore, using all the estimates for $J_1, J_2$ and $J_3$, from \eqref{j123}, we derive
	\begin{align*}
		&\mathbb{E}
		\bigg[\sup_{0\le t\le T\wedge\tau_n^{\varepsilon,\nu}}
		\|\hat {\mathfrak{u}}^{\varepsilon, \nu}(t)-\bar{\mathfrak{u}}^\nu(t)\|^{2p}\bigg]
		\le
		C\Bigg(
		\sqrt{\frac{\varepsilon}{\delta}}
		+\delta^p
		+\delta^{2p-1}
		+\frac{\delta^{p+1}}{\varepsilon}\,e^{\frac{C}{\varepsilon}\delta}
		+\bigg(\frac{\delta^{2p+1}}{\varepsilon}\,e^{\frac{C}{\varepsilon}\delta}\bigg)^{\frac12}
		\Bigg) \\[1ex]
		&\quad
		+ C n^{2p(\beta-1)}
		\int_0^T
		\mathbb{E}
		\bigg[\sup_{0\le r\le s\wedge\tau_n^{\varepsilon,\nu}}
		\|\hat {\mathfrak{u}}^{\varepsilon, \nu}(r)-\bar{\mathfrak{u}}^\nu(r)\|^{2p}\bigg]\,ds 
		+ C
		\int_0^T
		\sup_{0\le r\le s\wedge\tau_n^{\varepsilon,\nu}}
		\|\hat {\mathfrak{u}}^{\varepsilon, \nu}(r)-\bar{\mathfrak{u}}^\nu(r)\|^{2p}\,ds \\[1ex]
		&\le
		C\Bigg(
		\sqrt{\frac{\delta}{\varepsilon}}
		+\delta^p
		+\delta^{2p-1}
		+\frac{\delta^{p+1}}{\varepsilon}\,e^{\frac{C}{\varepsilon}\delta}
		+\bigg(\frac{\delta^{2p+1}}{\varepsilon}\,e^{\frac{C}{\varepsilon}\delta}\bigg)^{\frac12}
		\Bigg) \\[1ex]
		&\quad
		+ C n^{2p(\beta-1)}
		\int_0^T
		\mathbb{E}
		\bigg[	\sup_{0\le r\le s\wedge\tau_n^{\varepsilon,\nu}}
		\|\hat {\mathfrak{u}}^{\varepsilon, \nu}(r)-\bar{\mathfrak{u}}^\nu(r)\|^{2p}\bigg]\,ds .
	\end{align*}
	Using the Gr\"onwall inequality, we get
	\begin{align*}
		\mathbb{E}
		\bigg[	\sup_{0\le t\le T\wedge\tau_n^{\varepsilon,\nu}}
		\|\hat {\mathfrak{u}}^{\varepsilon, \nu}(t)-\bar{\mathfrak{u}}^\nu(t)\|^{2p}\bigg]
		\le
		C\Bigg(
		\sqrt{\frac{\delta}{\varepsilon}}
		+\delta^p
		+\delta^{2p-1}
		+\frac{\delta^{p+1}}{\varepsilon}\,e^{\frac{C}{\varepsilon}\delta}
		+\bigg(\frac{\delta^{2p+1}}{\varepsilon}\,e^{\frac{C}{\varepsilon}\delta}\bigg)^{\frac12}
		\Bigg)
		e^{C n^{2p(\beta-1)}} .
	\end{align*}
	This asserts that
	\begin{align*}
		\mathbb{E}
		\bigg[\sup_{0\le t\le T}
		\|\hat {\mathfrak{u}}^{\varepsilon, \nu}(t)-\bar{\mathfrak{u}}^\nu(t)\|^{2p}
		\chi_{\{T<\tau_n^{\varepsilon,\nu}\}}\bigg]
		\le
		C\Bigg(
		\sqrt{\frac{\delta}{\varepsilon}}
		+\delta^p
		+\delta^{2p-1}
		+\frac{\delta^{p+1}}{\varepsilon}\,e^{\frac{C}{\varepsilon}\delta}
		+\bigg(\frac{\delta^{2p+1}}{\varepsilon}\,e^{\frac{C}{\varepsilon}\delta}\bigg)^{\frac12}
		\Bigg)
		e^{C n^{2p(\beta-1)}} .
	\end{align*}
	Now, using the difinition of stopping times and the H\"older inequality, we estimate
	\begin{align*}
		\mathbb{E}
		\bigg[	\sup_{0\le t\le T}
		\|\hat {\mathfrak{u}}^{\varepsilon, \nu}(t)-\bar{\mathfrak{u}}^\nu(t)\|^{2p}
		\chi_{\{T>\tau_n^{\varepsilon,\nu}\}}\bigg]
		&\le\Bigg(	\mathbb{E}	\bigg[\sup_{0\le t\le T}
		\|{\mathfrak{u}}^{\varepsilon, \nu}(t)-\bar{\mathfrak{u}}^\nu(t)\|^{4p}\bigg]
		\Bigg)^{\frac12}
		 \bigg(\mathbb{E} \chi_{\{T>\tau_n^{\varepsilon,\nu}\}}\bigg)^{\frac12}\\[1ex]
		&\leq C  \bigg(\frac{C}{n}\bigg)^{\frac12}\le
		\frac{C}{\sqrt{n}} .
	\end{align*}
	Hence, we conclude
	\begin{align*}
		\mathbb{E}
		\bigg[
		\sup_{0\le t\le T}
		\|\hat {\mathfrak{u}}^{\varepsilon, \nu}(t)-\bar{\mathfrak{u}}^\nu(t)\|^{2p}
		\bigg]
		\le
		C\Bigg(
		\sqrt{\frac{\delta}{\varepsilon}}
		+\delta^p
		+\delta^{2p-1}
		+\frac{\delta^{p+1}}{\varepsilon}\,e^{\frac{C}{\varepsilon}\delta}
		+\bigg(\frac{\delta^{2p+1}}{\varepsilon}\,e^{\frac{C}{\varepsilon}\delta}\bigg)^{\frac12}
		\Bigg)
		e^{C n^{2p(\beta-1)}}
		+\frac{C}{\sqrt{n}} .
	\end{align*}
	This completes the proof of \Cref{propuhatuerror}.
\end{proof}
\subsection{Proof of \Cref{apofviscoused}:}\label{avgprnprf}
Combining the estimates in \Cref{uuhatvvhaterror} and \Cref{propuhatuerror}, we have
\begin{align*}
	\mathbb{E}
	\left[
	\sup_{0\le t\le T}
	\|{\mathfrak{u}}^{\varepsilon, \nu}(t)-\bar{\mathfrak{u}}^\nu(t)\|^{2p}
	\right]
	&\le
	C(p)\mathbb{E}
	\left[
	\sup_{0\le t\le T}
	\|{\mathfrak{u}}^{\varepsilon, \nu}(t)-\hat {\mathfrak{u}}^{\varepsilon, \nu}(t)\|^{2p}
	\right]
	+ C(p)\mathbb{E}
	\left[
	\sup_{0\le t\le T}
	\|\hat {\mathfrak{u}}^{\varepsilon, \nu}(t)-\bar{\mathfrak{u}}^\nu(t)\|^{2p}
	\right] \\[1ex]
	&\le
	C\Bigg(
	\sqrt{\frac{\varepsilon}{\delta}}
	+\delta^p
	+\delta^{2p-1}
	+\frac{\delta^{p+1}}{\varepsilon}e^{\frac{C}{\varepsilon}\delta}
	+\left(\frac{\delta^{2p+1}}{\varepsilon}e^{\frac{C}{\varepsilon}\delta}\right)^{\frac12}
	\Bigg)e^{C n^{2p(\beta-1)}}
	+\frac{C}{\sqrt{n}} .
\end{align*}
We choose 
\begin{align*}
	\delta=\varepsilon(-\ln\varepsilon)^{\frac12}, n=\left(\frac{1}{8C}\ln(-\ln\varepsilon)\right)^{\frac{1}{2p(\beta-1)}},
\end{align*}
and consequently, we have 
\begin{align*}
	e^{C n^{4p}}=(-\ln\varepsilon)^{\frac18}.
\end{align*}
Therefore, we conclude 
\begin{align*}
&	\mathbb{E}
	\bigg[
	\sup_{0\le t\le T}
	\|{\mathfrak{u}}^{\varepsilon, \nu}(t)-\bar{\mathfrak{u}}^\nu(t)\|^{2p}
	\bigg]
	\le
	C(-\ln\varepsilon)^{-\frac14}	(-\ln\varepsilon)^{\frac18}+ (-\ln\varepsilon)^{\frac18}\Bigg(\varepsilon^p(-\ln\varepsilon)^{\frac{p}{2}}+\varepsilon^{2p-1} (-\ln\varepsilon)^{\frac{2p-1}{2}}\\
	&\quad+\varepsilon^p(-\ln\varepsilon)^{\frac{p+1}{2}}e^{C(-\ln\varepsilon)^{\frac{1}{2}}}          +\left\{\varepsilon^{2p}(-\ln\varepsilon)^{\frac{2p+1}{2}}e^{C(-\ln\varepsilon)^{\frac{1}{2}}}\right\}^{\frac{1}{2}}    \Bigg)+\frac{C}{\left(\frac{1}{8C}\ln(-\ln\varepsilon)\right)^{\frac{1}{4p(\beta-1)}}} \\[1ex]
	&=
	C(-\ln\varepsilon)^{-\frac18}+ (-\ln\varepsilon)^{\frac18}\Bigg(\varepsilon^p(-\ln\varepsilon)^{\frac{p}{2}}+\varepsilon^{2p-1} (-\ln\varepsilon)^{\frac{2p-1}{2}}+\varepsilon^p(-\ln\varepsilon)^{\frac{p+1}{2}}e^{C(-\ln\varepsilon)^{\frac{1}{2}}}  \\
	&\quad        +\left\{\varepsilon^{2p}(-\ln\varepsilon)^{\frac{2p+1}{2}}e^{C(-\ln\varepsilon)^{\frac{1}{2}}}\right\}^{\frac{1}{2}}    \Bigg)	+C\left(\frac{1}{\ln(-\ln\varepsilon)}\right)^{\frac{1}{4p(\beta-1)}} .
\end{align*}
Taking $\varepsilon\to 0$, we get 
\begin{align*}
	\lim_{\varepsilon\to 0}\mathbb{E}\left[\sup_{0\le t\le T}\|{\mathfrak{u}}^{\varepsilon, \nu}(t)-\bar{\mathfrak{u}}^\nu(t)\|^{2p}\right]=0
\end{align*}
This completes the proof of \Cref{apofviscoused}.
\section{Proof of the main result}\label{secpfofmainresult}
After establishing \textbf{Step II} and \textbf{Step III} in the previous sections, we are ready to prove \textbf{Step IV} which is our main result.
\begin{proof}[Proof of \Cref{mainresult}]
	We use all the results established in \textbf{Step II} and \textbf{Step III}, and take limit $\varepsilon\to 0$ to complete the proof.
We have 
\begin{align}\label{decomposition}
	\|{\mathfrak{u}}^\varepsilon(t)-\bar{\mathfrak{u}}(t)\|
	\le
	\|{\mathfrak{u}}^{\varepsilon, \nu}(t)-{\mathfrak{u}}^\varepsilon(t)\|
	+
	\|{\mathfrak{u}}^{\varepsilon, \nu}(t)-\bar{\mathfrak{u}}^\nu(t)\|
	+
	\|\bar{\mathfrak{u}}^\nu(t)-\bar{\mathfrak{u}}(t)\|.
\end{align}
For any given $R>0$, \Cref{ubarnuandubar} implies that there exists $\nu_1$ such that
$\nu\in(0,\nu_1]$ and 
\begin{align}\label{ubarnuubarfinal}
	\sup_{0\le t\le T}
	\mathbb{E}\Big[|\bar{\mathfrak{u}}^\nu(t)-\bar{\mathfrak{u}}(t)\|^2	\Big]<R,
\end{align}
\Cref{uepsnuueps} asserts that there exists $\nu_2$ such that $\nu\in(0,\nu_2]$ and
\begin{align}\label{uepsnuuepsfinal}
	\sup_{\varepsilon\in(0,1)}
	\left(
	\sup_{0\le t\le T}
	\mathbb{E}\Big[\|{\mathfrak{u}}^{\varepsilon, \nu}(t)-{\mathfrak{u}}^\varepsilon(t)\|^2\Big]
	\right)<R.
\end{align}
Choosing $\nu=\min\{\nu_1,\nu_2\}$, $p=1$ in \Cref{apofviscoused}, we get
\begin{align*}
	\lim_{\varepsilon\to0}
	\mathbb{E}
	\left[
	\sup_{0\le t\le T}
	\|{\mathfrak{u}}^{\varepsilon, \nu}(t)-\bar{\mathfrak{u}}^\nu(t)\|^2
	\right]
	=0.
\end{align*}
Thus, there exists $\varepsilon_0$ such that for any
$\varepsilon\in(0,\varepsilon_0)$,
\begin{align}\label{uepsnuubarnufinal}
	\sup_{0\le t\le T}
	\mathbb{E}\Big[\|{\mathfrak{u}}^{\varepsilon, \nu}(t)-\bar{\mathfrak{u}}^\nu(t)\|^2\Big]
	\le
	\mathbb{E}
	\bigg[
	\sup_{0\le t\le T}
	\|{\mathfrak{u}}^{\varepsilon, \nu}(t)-\bar{\mathfrak{u}}^\nu(t)\|^2
	\bigg]
	< R .
\end{align}
Using \eqref{decomposition}-\eqref{uepsnuubarnufinal}, we conclude
\begin{align*}
	\sup_{0\le t\le T}
	\mathbb{E}\Big[\|{\mathfrak{u}}^\varepsilon(t)-\bar{\mathfrak{u}}(t)\|^2\Big]
	< 3R .
\end{align*}
Since $R$ is arbitrary, we have 
\begin{align*}
	\lim_{\varepsilon \to 0} \sup_{0 \leq t \leq T} \mathbb{E} \Big[\|{\mathfrak{u}}^\varepsilon(t) - \bar{\mathfrak{u}}(t)\|^2\Big] = 0.
\end{align*}
This completes the proof of \Cref{mainresult}.

\end{proof}

	\noindent \textbf{Data availibility:} Data sharing not applicable to this article as no datasets were generated or analysed during the current study.\\
	
	\noindent \textbf{Funding:}  Not applicable.\\
	
	
	\noindent \textbf{\large Declarations}\\
	
	\noindent\textbf{Conflict of interest:} The author has no competing interests to declare that are relevant to the content of this article.

	\bibliography{AP.bib}
	\bibliographystyle{abbrv}
\end{document}